\documentclass[10pt,reqno,a4paper]{amsart}

\usepackage{amssymb}
\usepackage[utf8]{inputenc}
\usepackage{t1enc}
\usepackage{dsfont} 
\usepackage{mathrsfs}
\usepackage[hidelinks]{hyperref}
\usepackage{upgreek}
\usepackage{mathabx}
\usepackage[all]{xy}

\usepackage{geometry}
\usepackage[bbgreekl]{mathbbol}                  
\DeclareSymbolFontAlphabet{\mathbb}{AMSb}        
\DeclareSymbolFontAlphabet{\mathbbl}{bbold}      

\newtheorem{proposition}{Proposition}[section]
\newtheorem{theorem}[proposition]{Theorem}
\newtheorem{corollary}[proposition]{Corollary}
\newtheorem{lemma}[proposition]{Lemma}
\newtheorem{conjecture}{Conjecture}
\theoremstyle{definition}
\newtheorem{definition}[proposition]{Definition}
\newtheorem{remark}[proposition]{Remark}
\newtheorem{example}[proposition]{Example}

\newcommand{\cst}{\ensuremath{\mathrm{C}^*}}
\newcommand{\comp}{\circ}
\newcommand{\dd}{\,\mathrm{d}}
\newcommand{\ee}{\mathrm{e}}
\newcommand{\eps}{\varepsilon}
\newcommand{\ph}{\varphi}
\newcommand{\hh}[1]{\widehat{#1}}
\newcommand{\I}{\mathds{1}}
\newcommand{\id}{\mathrm{id}}
\newcommand{\ii}{\mathrm{i}}
\newcommand{\Int}{\int\limits}
\newcommand{\is}[2]{{\left\langle{#1}\,\vline\,#2\right\rangle}}
\newcommand{\ket}[1]{{\left|#1\right\rangle}}
\newcommand{\bra}[1]{{\left\langle#1\right|}}
\newcommand{\tens}{\otimes}
\newcommand{\vtens}{\overset{\rule[-1pt]{5pt}{0.5pt}}{\tens}}
\newcommand{\oon}{\operatorname}

\newcommand{\half}{\scriptscriptstyle{1\:\!\!/\;\!\!2}}
\newcommand{\ihalf}{\scriptscriptstyle{\ii\;\!\!/\;\!\!2}}
\newcommand{\mihalf}{\scriptscriptstyle{-\ii\;\!\!/\;\!\!2}}
\newcommand{\mhalf}{\scriptscriptstyle{-1\:\!\!/\;\!\!2}}

\newcommand{\threehalf}{\scriptscriptstyle{3\:\!\!/\;\!\!2}}
\newcommand{\bigthalf}{\scriptstyle{t\;\!\!/\;\!\!2}}

\newcommand{\CC}{\mathbb{C}}
\newcommand{\HH}{\mathbb{H}}
\newcommand{\GG}{\mathbb{G}}

\newcommand{\NN}{\mathbb{N}}
\newcommand{\QQ}{\mathbb{Q}}
\newcommand{\RR}{\mathbb{R}}
\newcommand{\TT}{\mathbb{T}}
\newcommand{\ZZ}{\mathbb{Z}}
\newcommand{\bbGamma}{\mathbbl{\Gamma}}

\newcommand{\cH}{\mathscr{H}}
\newcommand{\cK}{\mathscr{K}}

\newcommand{\cQ}{\mathcal{Q}}
\newcommand{\cV}{\mathcal{V}}

\newcommand{\bh}{{\boldsymbol{h}}}
\newcommand{\blambda}{{\boldsymbol{\lambda}}}
\newcommand{\bP}{{\boldsymbol{P}}}
\newcommand{\bQ}{{\boldsymbol{Q}}}
\newcommand{\bt}{{\boldsymbol{t}}}

\newcommand{\cc}{\text{\tiny{\rm{c}}}}
\newcommand{\LL}{\text{\tiny{\rm{L}}}}
\newcommand{\uu}{\text{\tiny{\rm{u}}}}

\newcommand{\la}{\langle}
\newcommand{\ra}{\rangle}
\newcommand{\ww}{\mathrm{W}}
\newcommand{\vv}{\mathrm{V}}

\DeclareMathOperator{\Ad}{Ad}
\DeclareMathOperator{\B}{B}
\DeclareMathOperator{\C}{C}

\DeclareMathOperator{\HS}{HS}
\DeclareMathOperator{\Irr}{Irr}
\DeclareMathOperator{\Lone}{\mathsf{L}^1}
\DeclareMathOperator{\Ltwo}{\mathsf{L}^2}
\DeclareMathOperator{\Linf}{\mathsf{L}^{\!\infty}}

\DeclareMathOperator{\Pol}{Pol}
\DeclareMathOperator{\qdim}{\dim_{\text{\tiny{\rm{q}}}}\!}
\DeclareMathOperator{\cZ}{\mathscr{Z}}
\DeclareMathOperator{\Tr}{Tr}
\DeclareMathOperator{\lwt}{lwt}
\DeclareMathOperator{\rwt}{rwt}
\DeclareMathOperator{\wt}{wt}

\newcommand{\modOp}[1][\hspace{0.87pt}]{\nabla_{\!#1}}
\newcommand{\modOphalf}[1][\hspace{0.87pt}]{\nabla^{\scriptscriptstyle{1\:\!\!/\;\!\!2}}_{\!#1}}
\newcommand{\modOpmhalf}[1][\hspace{0.87pt}]{\nabla^{\scriptscriptstyle{-1\:\!\!/\;\!\!2}}_{\!#1}}
\newcommand{\rM}{\mathrm{M}}
\DeclareMathOperator{\Sd}{\mathnormal{S}\!\!\;\mathnormal{d}}
\newcommand{\Upsi}[1]{\Upsilon_{\!#1}}

\DeclareMathOperator{\Ttau}{\mathnormal{T}^{\tau\!}}
\DeclareMathOperator{\TtauInn}{\mathnormal{T}^{\tau}_{\oon{Inn}\!}}
\DeclareMathOperator{\TtauAInn}{\mathnormal{T}^{\tau}_{\oon{\overline{Inn}}}}
\DeclareMathOperator{\Tsig}{\mathnormal{T}^{\sigma\!}}
\DeclareMathOperator{\TsigInn}{\mathnormal{T}^{\sigma}_{\oon{Inn}\!}}
\DeclareMathOperator{\TsigAInn}{\mathnormal{T}^{\sigma}_{\oon{\overline{Inn}}}}
\DeclareMathOperator{\Mod}{Mod}

\usepackage{stackengine}
\newcommand{\tp}{\text{\tiny$\stackMath\mathbin{\stackinset{c}{0ex}{c}{-0.17ex}{\top}{\bigcirc}}$}}

\numberwithin{equation}{section}

\title[Invariants of compact quantum groups]{On certain invariants of compact quantum groups}

\author{Jacek Krajczok}
\address{Vrije Universiteit Brussel}
\email{jacek.krajczok@vub.be}

\author{Piotr M.~So{\l}tan}
\address{Department of Mathematical Methods in Physics, Faculty of Physics, University of Warsaw}
\email{piotr.soltan@fuw.edu.pl}

\keywords{compact quantum group, von Neumann algebra, scaling group, invariant}

\makeatletter
\@namedef{subjclassname@2020}{\textup{2020} Mathematics Subject Classification}
\makeatother

\subjclass[2020]{46L67, 20G42}

\begin{document}

\begin{abstract}
We introduce and study a number of invariants of locally compact quantum groups defined by their scaling and modular groups and the spectrum of their modular elements. Focusing mainly on compact quantum groups we consider the question whether triviality of one of the invariants is equivalent to the quantum group being of Kac type and show that it has a positive answer in many cases including duals of second countable type $\mathrm{I}$ discrete quantum groups. We perform a complete calculation of the invariants for all $q$-deformations of compact, simply connected, semisimple Lie groups as well as for some non-compact quantum groups and the compact quantum groups $\operatorname{U}_F^+$. Finally we introduce a family of conditions for discrete quantum groups which for classical discrete groups are all equivalent to the fact that the group is i.c.c. We show that the above mentioned question about characterization of Kac type quantum groups by one of our invariants has a positive answer for duals of discrete quantum groups satisfying such an i.c.c.-type condition and illustrate this with the example of $\operatorname{U}_F^+$.
\end{abstract}

\maketitle

\tableofcontents

\allowdisplaybreaks

\section{Introduction}

In this paper we investigate certain invariants of compact and locally compact quantum groups which grew out of our previous investigations into compact quantum groups $\GG$ with $\Linf(\GG)$ a factor (\cite{faktory}). In that paper we constructed various families of compact quantum groups and the techniques used to show that members of these families are pairwise non-isomorphic included looking at three invariants in the form of subgroups of $\RR$ assigned to a locally compact quantum group. They are all similar to the famous Connes invariant $T$ for von Neumann algebras, but the modular group is replaced by the scaling group of the quantum group.

In this paper we take a much more systematic approach to these invariants. First of all we extend the family of invariants to seven items (see detailed discussion in Section \ref{sect:defInv}) for $\GG$ and seven for its Pontriagin dual $\hh{\GG}$. The additional invariants are related to the modular group of the Haar measure (in fact one of our invariants is exactly $T(\Linf(\GG))$) and the spectrum of the modular element.

The leitmotif of the paper is the statement and study of a conjecture relating triviality of one of the invariants for a compact quantum group $\GG$ with the question whether $\GG$ is of Kac type (Conjecture \ref{conj:main}). Although at first sight the conjecture does not likely to hold in much generality, we have been able to prove it in several cases. This includes the case of second countable compact quantum groups with $2$-dimensional representation satisfying $\dim{U}<\qdim{U}$ (Section \ref{sect:UqU}) and duals of second countable discrete quantum groups of type $\mathrm{I}$ (Section \ref{sec:typeI}). Furthermore, we discuss a family of conditions which we call i.c.c.-type conditions which for duals of classical discrete groups are all equivalent to the group being i.c.c., and turn out to imply the validity of our main conjecture.

Sections \ref{sec:qDeformations} and \ref{sect:otherEx} are devoted to computation of the values of our invariants. In Section \ref{sec:qDeformations} we develop a rather extensive machinery in order to compute all the invariants for $q$-deformations $G_q$ of all compact, simply connected, semisimple Lie groups. Along the way we encounter an interesting phenomenon that there are compact quantum groups (e.g.~$\oon{SU}_q(3)$) with non-trivial inner scaling automorphisms which are not implemented by a group-like unitary. In fact the implementing unitaries do not even belong to $\C(G_q)$ (see Section \ref{sect:implemUnit}). Next, in Section \ref{sect:otherEx}, we include computations of the invariants for two families of non-compact quantum groups (the quantum deformations of the $\oon{E}(2)$ and the ``$az+b$'' groups) and for the compact quantum group $\oon{U}_F^+$. This compact quantum group also serves as an illustration of the results about i.c.c.~type conditions mentioned above (see Section \ref{sect:exUFpl}).

Throughout the paper we have adhered to the standard conventions of the theory of compact and locally compact quantum groups. In particular, given a locally compact quantum group $\GG$ the symbol $\Linf(\GG)$ denotes the associated von Neumann algebra acting on the GNS Hilbert space $\Ltwo(\GG)$ of the left Haar measure and $\Lone(\GG)$ stands for the predual of $\Linf(\GG)$. The reduced \cst-algebra of functions on $\GG$ is denoted by $\C_0(\GG)$ or by $\C(\GG)$ (if $\GG$ is compact) or $\mathrm{c}_0(\GG)$ (if $\GG$ is discrete). A locally compact quantum group $\GG$ is called \emph{second countable} if either of the following conditions holds:
\begin{itemize}
\item $\C_0(\GG)$ is separable,
\item $\Lone(\GG)$ is separable,
\item $\Ltwo(\GG)$ is separable
\end{itemize}
(\cite[Lemma 14.6]{KrajczokTypeI}). We say that a locally compact quantum group $\GG$ is \emph{type $\mathrm{I}$} if $\C_0^\uu(\hh{\GG})$ is a type $\mathrm{I}$ \cst-algebra (also see \cite{KrajczokTypeI}). For the general theory of locally compact quantum groups we refer the reader to \cite{KustermansVaes,KustermansVaesVNA,VanDaeleQGvN} as well as the books \cite{Timmermann,Tuset}, while for compact quantum groups we recommend the book \cite{NeshveyevTuset}.

The theory of locally compact quantum groups uses many standard notions from the theory of operator algebras and the Tomita-Takesaki theory, for which the reader is advised to consult \cite{DixmierC,DixmiervNA,Sakai,Takesaki2} and \cite{Connes} as well as any of the many other monographs on the subject. In particular, if $\theta$ is a normal semifinite faithful (n.s.f.) weight on a von Neumann algebra $\rM$ then $\mathfrak{N}_{\theta}=\bigl\{x\in\rM\,\bigr|\bigl.\,\theta(x^*x)<+\infty\bigr\}$ and $\mathfrak{M}_{\theta}=\oon{span}\mathfrak{N}_\theta^*\mathfrak{N}_\theta$. The GNS map for $\theta$ will be denoted by $\Lambda_\theta$.

We will also use some of the less standard notation, namely for a compact quantum group $\GG$ we will write $\Linf(\GG)^\sigma$ for the elements of the von Neumann algebra $\Linf(\GG)$ invariant under the modular group of the Haar measure of $\GG$, which itself will be denoted by $\bh_\GG$ (i.e.~$\Linf(\GG)^\sigma$ is the \emph{centralizer} of $\bh_\GG$). The canonical cyclic vector for the GNS representation of $\Linf(\GG)$ associated to $\bh_\GG$ will be denoted by $\Omega_\GG$.

In order to stay consistent with the notation of \cite{modular,KrajczokTypeI} and \cite{KustermansVaes}, a representation $U$ of a compact quantum group $\GG$ on a (in all cases finite dimensional) Hilbert space $\cH_U$ is a unitary element of $\C(\GG)\tens\B(\cH_U)$ satisfying $(\Delta_{\GG}\tens\id)(U)=U_{13}U_{23}$.

The dimension of $U$ will be denoted by $\dim{U}$ and the quantum dimension (\cite[Section 1.4]{NeshveyevTuset}) by $\qdim{U}$. Note that if $\alpha$ denotes the equivalence class of $U$ then the notions of dimension and quantum dimension of $\alpha$ make perfect sense and we will often use symbols like $\dim{\alpha}$ or $\qdim{\alpha}$.

Let $\GG$ be a compact quantum group and let $\Irr{\GG}$ denote the set of equivalence classes of all irreducible representations of $\GG$. We will often choose unitary representatives $U^\alpha\in\alpha$ (in fact we will only consider unitary representations) and then choose bases in the corresponding Hilbert spaces $\{\cH_\alpha\}_{\alpha\in\Irr{\GG}}$ in which the operators $\uprho_\alpha$ are diagonal with eigenvalues $\uprho_{\alpha,1},\dotsc,\uprho_{\alpha,\dim{\alpha}}$ (sometimes we will also ask that these eigenvalues are in increasing order). The corresponding matrix elements $U^\alpha_{i,j}$ of the representations $U^\alpha$ transform under the modular and scaling groups of $\GG$ as follows:
\begin{equation}\label{sigmatau}
\left\{\begin{aligned}
\sigma^\bh_t(U^\alpha_{i,j})&=\uprho_{\alpha,i}^{\ii{t}}U^\alpha_{i,j}\uprho_{\alpha,j}^{\ii{t}}\\
\tau^\GG_t(U^\alpha_{i,j})&=\uprho_{\alpha,i}^{\ii{t}}U^\alpha_{i,j}\uprho_{\alpha,j}^{-\ii{t}}
\end{aligned}\right.,\qquad{t}\in\RR,\:i,j\in\{1,\dotsc,\dim{\alpha}\}.
\end{equation}

Finally let us mention that in Sections \ref{sec:typeI} and \ref{sect:prel} we will use a very convenient notation for eigenspaces and eigenprojections of self-adjoint (and other) operators. Namely if $\cH$ is a Hilbert space and $T$ is a self-adjoint operator on $\cH$ (not necessarily bounded) then the subspace $\ker(\lambda\I-T)$ will be denoted by $\cH(T=\lambda)$, while the orthogonal projection onto $\cH(T=\lambda)$ will be written as $\chi(T=\lambda)$. Similarly, given an interval $[a,b]\subset\RR$ the projection $\chi_{[a,b]}(T)$ will be denoted by the symbol $\chi(a\leq{T}\leq{b})$ and its range by $\cH(a\leq{T}\leq{b})$. We will also use an obvious extension of this notation to any normal operator $N$ and any measurable subset $\Omega\subset\CC$ according to which $\chi_\Omega(N)$ is denoted $\chi(N\in\Omega)$ and the range of this projection is $\cH(N\in\Omega)$.

\subsection*{Acknowledgments}
The work of the first named author was partially supported by EPSRC grants EP/T03064X/1 and EP/T030992/1 as well as the FWO grant 1246624N. The second author was partially supported by NCN (National Science Centre, Poland) grant no.~2022/47/B/ST1/00582. The research was also partially supported by the University of Warsaw Thematic Research Program "Quantum Symmetries".

\section{The invariants and some of their properties}\label{sect:defInv}

Let $\GG$ be a locally compact quantum group. In \cite{faktory} we introduced the following invariants of $\GG$, related to the behavior of the scaling group:

\begin{equation}\label{eq18}
\begin{aligned}
\Ttau(\GG)&=\bigl\{t\in\RR\,\bigr|\bigl.\,\tau_t^\GG=\id\bigr\},\\
\TtauInn(\GG)&=\bigl\{t\in\RR\,\bigr|\bigl.\,\tau_t^\GG\in\oon{Inn}(\Linf(\GG))\bigr\},\\
\TtauAInn(\GG)&=\bigl\{t\in\RR\,\bigr|\bigl.\,\tau_t^\GG\in\oon{\overline{Inn}}(\Linf(\GG))\bigr\},
\end{aligned}
\end{equation}
where $\oon{Inn}(\Linf(\GG))$ and $\oon{\overline{Inn}}(\Linf(\GG))$ are respectively the group of inner automorphisms and the group of approximately inner automorphisms of the von Neumann algebra $\Linf(\GG)$.

Some basic properties and examples of calculation of these invariants were recorded in \cite[Section 5]{faktory}, in particular for any locally compact quantum group $\GG$ equality $\Ttau(\GG)=\Ttau(\hh{\GG})$ holds.

In addition to the invariants \eqref{eq18} we will study invariants related to modular automorphisms and the modular element.

\begin{definition}\label{def1}
Let $\GG$ be a locally compact quantum group with left Haar measure $\ph$ and modular element $\delta$. We define the following subsets of $\RR$:
\begin{align*}
\Tsig(\GG)&=\bigl\{t\in\RR\,\bigr|\bigl.\,\sigma^\ph_t=\id\bigr\},\\
\TsigInn(\GG)&=\bigl\{t\in\RR\,\bigr|\bigl.\,\sigma^\ph_t\in\oon{Inn}(\Linf(\GG))\bigr\},\\
\TsigAInn(\GG)&=\bigl\{t\in\RR\,\bigr|\bigl.\,\sigma_t^\ph\in\oon{\overline{Inn}}(\Linf(\GG))\bigr\},\\
\Mod(\GG)&=\bigl\{t\in\RR\,\bigr|\bigl.\,\delta^{\ii{t}}=\I\bigr\}.
\end{align*}
\end{definition}

It is sometimes convenient to use the notation $T^\circ_{\bullet\!}(\GG)$ allowing $\circ$ and $\bullet$ to stand for $\tau$ or $\sigma$, and $\oon{Inn}$, $\oon{\overline{Inn}}$ or nothing respectively. We will indicate it by considering $T^\circ_\bullet$ for $\circ\in\{\tau,\sigma\}$ and, rather unconventionally, $\bullet\in\bigl\{\ \,,\oon{Inn},\oon{\overline{Inn}}\bigr\}$.

\begin{remark}\label{rem:przedProp5}\hspace*{\fill}
\begin{enumerate}
\item The sets introduced in the Definition \ref{def1} are subgroups of $\RR$. Furthermore, $\Tsig(\GG)$, $\TsigAInn(\GG)$, and $\Mod(\GG)$ are closed. Indeed, for $\Tsig(\GG)$ and $\TsigAInn(\GG)$ the argument is analogous to \cite[Proposition 5.3]{faktory}, whereas closedness of $\Mod(\GG)$ follows from the spectral theorem for self-adjoint operators (cf.~e.g.~\cite[Theorem 10.5]{primer}).
\item\label{rem:przedProp5-2} We would obtain the same groups $\Tsig(\GG)$, $\TsigInn(\GG)$, and $\TsigAInn(\GG)$ if we chose the right Haar measure instead of the left one -- this is a consequence of the equalities
\[
\modOp[\psi]^{\ii{t}}=J_{\hh{\ph}}\modOp[\ph]^{-\ii{t}}J_{\hh{\ph}},\quad
\sigma_t^\psi(x)=\delta^{\ii{t}}\sigma_t^\ph(x)\delta^{-\ii{t}}
\]
valid for all $t\in\RR$ and any $x\in\Linf(\GG)$.
\item $\TsigInn(\GG)$ is equal to the Connes' invariant $T(\Linf(\GG))$. Consequently, $\TsigInn(\GG)$ depends only on the von Neumann algebra $\Linf(\GG)$. It is also the case for $\TsigAInn(\GG)$ (see \cite[Theorem VIII.3.3]{Takesaki2}).
\item If $\GG$ and $\HH$ are isomorphic locally compact quantum groups, then $T^\circ_{\bullet\!}(\GG)=T^\circ_{\bullet\!}(\HH)$ for all values of $\circ$ and $\bullet$ considered above and similarly $\Mod(\GG)=\Mod(\HH)$.
\end{enumerate}
\end{remark}

We will now gather some general properties of the above invariants.

\begin{proposition}\label{prop5}
For any locally compact quantum group $\GG$ we have
\begin{subequations}
\begin{align}
\Tsig(\GG)&=\Ttau(\GG)\cap\Mod(\hh{\GG}),\label{eq20}\\
\TsigInn(\GG)\cap\Mod(\hh{\GG})&=\TtauInn(\GG)\cap\Mod(\hh{\GG}),\label{eq21}\\
\TsigAInn(\GG)\cap\Mod(\hh{\GG})&=\TtauAInn(\GG)\cap\Mod(\hh{\GG}),\label{eq22}\\
\Mod(\GG)\cap\Mod(\hh{\GG})&\subset\tfrac{1}{2}\Ttau(\GG).\label{eq23}
\end{align}
\end{subequations}
\end{proposition}

Equation \eqref{eq20} can be seen as a ``local version'' of \cite[Lemma 6.2]{modular}.

\begin{proof}[Proof of Proposition \ref{prop5}]
Recall that for any $t\in\RR$ we have $\modOp[\psi]^{\ii{t}}=\hh{\delta}^{-\ii{t}}P^{-\ii{t}}$ (where $P$ is defined in \cite[Definition 6.9]{KustermansVaes}), thus if $t\in\Mod(\hh{\GG})$ then $\modOp[\psi]^{\ii{t}}=P^{-\ii{t}}$ and consequently $\sigma^\psi_t=\tau^\GG_{-t}$. From this we deduce \eqref{eq21}, \eqref{eq22} and containment $\supseteq$ in \eqref{eq20}.

Inclusion $\Tsig(\GG)\subset\Ttau(\GG)\cap\Mod(\hh{\GG})$ can be shown in the same way as in \cite[Lemma 6.2]{modular}, let us recall this reasoning. Take $t\in\Tsig(\GG)$. Then we have $P^{\ii{t}}={\hh{\delta}^{-\ii{t}}}\in\Linf(\hh{\GG})$. Since $P^{\ii{t}}$ commutes with $J_{\hh{\ph}}$, we have $P^{\ii{t}}\in\cZ(\Linf(\hh{\GG}))$ and consequently $\tau^{\hh{\GG}}_t=\id$, i.e.~$t\in\Ttau(\hh{\GG})$. As $P^{\ii{t}}$ implements $\tau^{\hh{\GG}}_t$ in the sense that $P^{\ii{t}}\Lambda_{\hh{\ph}}(x)=\hh{\nu}^{\bigthalf}\Lambda_{\hh{\ph}}(\tau^{\hh{\GG}}_t(x))=\hh{\nu}^{\bigthalf}\Lambda_{\hh{\ph}}(x)$ for $x$ square integrable for $\hh{\ph}$, we conclude $P^{\ii{t}}=\hh{\nu}^{\bigthalf}\I$ and taking norm of both sides gives $1=\hh{\nu}^{\bigthalf}$. Consequently $\hh{\delta}^{-\ii{t}}=\I$ and thus $t\in\Mod(\hh{\GG})$.

Inclusion \eqref{eq23} follows from the fact that $P^{-2\ii{t}}=\delta^{\ii{t}}(J_\ph\delta^{\ii{t}}J_\ph)\hh{\delta}^{\ii{t}}(J_{\hh{\ph}}\hh{\delta}^{\ii{t}}J_{\hh{\ph}})$ for all $t\in\RR$ (\cite[Theorem 5.20]{VanDaeleQGvN}).
\end{proof}

We obtain more information about the invariants if von Neumann algebra $\Linf(\GG)$ has particularly simple form.

\begin{proposition}\label{prop6}
If $\Linf(\GG)=\prod\limits_{i\in{I}}\B(\cH_i)$ for a family of Hilbert spaces $\{\cH_i\}_{i\in{I}}$, then $T^\tau_{\bullet\!}(\GG)=T^\sigma_{\bullet\!}(\GG)=\RR$ for $\bullet\in\bigl\{\oon{Inn},\oon{\overline{Inn}}\bigr\}$.
\end{proposition}

The above result applies in particular to discrete quantum groups.

\begin{proof}[Proof of Proposition \ref{prop6}]
Let $\tilde{\tau}_t$ be the scaling group of $\GG$ restricted to the center of $\prod\limits_{i\in{I}}\B(\cH_i)$ which is isomorphic to $\ell^{\infty}(I)$. Each $\tilde{\tau}_t$ is an automorphism, hence it corresponds to a bijection $I\to{I}$. Since $\{\tilde{\tau}_t\}_{t\in\RR}$ is point ultraweakly continuous, we conclude that $\tilde{\tau}_t=\id$. Let $p_i$ be the central projection corresponding to $\B(\cH_i)\subset\prod\limits_{j\in{I}}\B(\cH_j)$. For any $x\in\B(\cH_i)$ we have
\[
\tau^\GG_t(x)=\tau^\GG_t(p_ix)=\tilde{\tau}_t(p_i)\tau^\GG_t(x)=p_i\tau^\GG_t(x),
\]
and hence $\tau^\GG_t$ restricts to an automorphism of $\B(\cH_i)$. As all automorphisms of $\B(\cH_i)$ are inner (\cite[Corollary 2.9.32]{Sakai}) we can choose unitary $u_{i,t}\in\B(\cH_i)$ such that $\bigl.\tau^\GG_t\bigr|_{\B(\cH_i)}=\Ad(u_{i,t})$. Since the series $\sum\limits_{i\in{I}}u_{i,t}$ convergences in the ultraweak topology, we conclude that $\tau^\GG_t=\Ad\Bigl(\sum\limits_{i\in{I}}u_{i,t}\Bigr)$ is inner and $\TtauInn(\GG)=\RR$. The argument for $\TsigInn(\GG)$ is analogous (alternatively, observe that $\prod\limits_{i\in{I}}\B(\cH_i)$ is semifinite).
\end{proof}

\begin{remark}\label{remark1}
Given a locally compact quantum group $\GG$ we have fourteen associated invariants, namely $T^\circ_{\bullet\!}(\GG)$ with $\circ\in\{\tau,\sigma\}$ and $\bullet\in\bigl\{\, \,,\oon{Inn},\oon{\overline{Inn}}\bigr\}$, $\Mod(\GG)$ and similar invariants for $\hh{\GG}$. Equality \eqref{eq20} and $\Ttau(\GG)=\Ttau(\hh{\GG})$ reduces this number to eleven. If $\GG$ is compact, then propositions \ref{prop5}, \ref{prop6} further reduce this list and all the information is contained in six invariants:
\[
\Ttau(\GG),\quad
\TtauInn(\GG),\quad
\TtauAInn(\GG),\quad
\TsigInn(\GG),\quad
\TsigAInn(\GG),\quad
\Mod(\hh{\GG}).
\]
If $\GG$ is compact and $\Linf(\GG)$ is semifinite (e.g.~$\GG=G_q$ or more generally $\hh{\GG}$ is type $\mathrm{I}$, see Sections \ref{sec:typeI}, \ref{sec:qDeformations}) then $\TsigInn(\GG)=\TsigAInn(\GG)=\RR$ and the list reduces to four invariants.
\end{remark}

\section{Does \texorpdfstring{$\TtauInn(\GG)=\RR$}{T(G)=R} characterize Kac type?}

At the end of \cite[Section 5]{faktory} we posed two questions, one of which concerns the invariant $\TtauInn$. More precisely for a compact quantum group $\GG$ we asked whether $\GG$ is of Kac type whenever $\TtauInn(\GG)=\RR$. We also noted that this question will certainly have the negative answer without an additional assumption like that $\GG$ is second countable (see below). One of the aims of this paper is to show that this question has a positive answer for a broad class of compact quantum groups. Let us formulate the following conjecture:

\begin{conjecture}\label{conj:main}
Let $\GG$ be a second countable compact quantum group. If $\TtauInn(\GG)=\RR$ then $\GG$ is of Kac type.
\end{conjecture}

Conjecture \ref{conj:main} could be compared with the following famous result in theory of von Neumann algebras. Let $\rM$ be a von Neumann algebra with separable predual. Then Connes' invariant $T(\rM)$ equals $\RR$ if and only if $\rM$ is semifinite, i.e.~admits a tracial n.s.f.~weight (\cite[Theorem VIII.3.14]{Takesaki2}). So far we cannot prove Conjecture \ref{conj:main} in full generality, however, as mentioned above, in subsequent sections we will show its validity for a large class of compact quantum groups.

\begin{remark}
As noted in \cite[Question 5.32(1)]{faktory}, one can easily construct a non-Kac type compact quantum group $\GG$ with $\TtauInn(\GG)=\RR$. For example, fix $0<q<1$ and let $\RR_{\text{\tiny{disc}}}$ be the group of real numbers with discrete topology, acting on $\Linf(\oon{SU}_q(2))$ via the scaling automorphisms. Then the bicrossed product $\RR_{\text{\tiny{disc}}}\bowtie\oon{SU}_q(2)$ is a non-Kac type compact quantum group with all scaling automorphisms inner -- it is however not second countable.
\end{remark}

\subsection{Quantum groups with a representation \texorpdfstring{$U$}{U} such that \texorpdfstring{$2=\dim{U}<\qdim{U}$}{2=dimU<dimqU}}\label{sect:UqU}\hspace*{\fill}

Consider a compact quantum group $\GG$ with an irreducible representation $U$ such that $\qdim{U}>\dim{U}$. Then clearly $\GG$ is not of Kac type and hence $\Ttau(\GG)\neq\RR$. For such compact quantum groups the validity of Conjecture \ref{conj:main} boils down to whether $\TtauInn(\GG)=\RR$ or not. In this section we will demonstrate that the conjecture holds in the case of compact quantum groups which posses a sequence of irreducible representations satisfying certain conditions (listed in Theorem \ref{thm:Un} below). Later we will show that such a sequence of irreducible representations can be constructed for compact quantum groups with a two dimensional representation whose quantum dimension is strictly greater than two (such a representation must be irreducible).

We begin with a simple lemma which we can prove in the general setting of locally compact quantum groups.

\begin{lemma}\label{lem:centrv}
Let $\HH$ be a locally compact quantum group and let $v\in\Linf(\HH)$ be a unitary such that $\tau_t^\HH=\Ad(v)=v\cdot{v^*}$ for some $t\in\RR$. Then $\sigma^{\varphi}_s(v)=\sigma^{\psi}_s(v)=\nu_{\HH}^{-ist}v$ for all $s\in\RR$, where $\nu_{\HH}$ is the scaling constant of $\HH$. In particular, if $\HH$ is compact or discrete then $v\in\Linf(\HH)^\sigma$.
\end{lemma}

\begin{proof}
Take $y\in\mathfrak{N}_{\varphi}$. Using $\varphi\comp\tau_{-t}^{\HH}=\nu_{\HH}^{t}\varphi$ (see \cite{KustermansVaesVNA}) we have
\[
\varphi\bigl((yv)^*(yv)\bigr)=\varphi(v^*y^*yv)=\bigl(\varphi\comp\tau_{-t}^{\HH}\bigr)(y^*y)=\nu_{\HH}^{t}\varphi(y^*y)<+\infty,
\]
and hence $yv\in\mathfrak{N}_{\varphi}$. Similarly we check $yv^*\in\mathfrak{N}_{\varphi}$. Take $x\in\mathfrak{M}_{\varphi}$. The above properties imply $xv,vx\in\mathfrak{M}_{\varphi}$. Furthermore
\[
\varphi(vx)=\varphi(vxvv^*)=\bigl(\varphi\comp\tau_t^{\HH}\bigr)(xv)=\nu_{\HH}^{-t}\varphi(xv),
\]
hence \cite[Lemma 5.4]{VaesRN} gives $\sigma^{\varphi}_s(v)=\nu_{\HH}^{-its}v$ for $s\in\RR$. The second claim can be proved in an analogous way using $\psi\comp\tau^{\HH}_{-t}=\nu_{\HH}^t\psi$. The last statement is implied by the fact that compact and discrete quantum groups have trivial scaling constant.
\end{proof}

In what follows we will use the following notation: if $U$ is a representation of a compact quantum group $\GG$ and $\uprho_U$ the associated positive operator describing the modular properties of the Haar measure (\cite[Section 1.4]{NeshveyevTuset}) then we write
\begin{itemize}
\item $\Gamma(U)=\max\oon{Sp}\uprho_U=\|\uprho_U\|$,
\item $\gamma(U)=\min\oon{Sp}\uprho_U=\|\uprho_U^{-1}\|^{-1}$.
\end{itemize}
Note that $\Gamma(\overline{U})=\gamma(U)^{-1}$. Furthermore $\Gamma(U)$ and $\gamma(U)$ only depend on the (unitary) equivalence class of $U$ and hence we may use notation like $\Gamma(\alpha)$ and $\gamma(\alpha)$ for an equivalence class $\alpha$ of representations of $\GG$.

\begin{theorem}\label{thm:Un}
Let $\GG$ be a second countable compact quantum group and let $\{U^n\}_{n\in\NN}$ be a family of irreducible representations of $\GG$ such that
\begin{enumerate}
\item $\Gamma(U^n)\xrightarrow[n\to\infty]{}+\infty$,
\item $\inf\limits_{n\in\NN}\Bigl(\tfrac{1}{\gamma(U^n)\qdim{U^n}}\,\tfrac{\Gamma(U^n)}{\qdim{U^n}}\Bigr)>0$.
\end{enumerate}
Then $\TtauInn(\GG)\neq\RR$.
\end{theorem}

\begin{proof}
We argue by contradiction: assume that $\TtauInn(\GG)=\RR$. Then by \cite[Theorem 0.1]{Kallman} there exists a strongly continuous one-parameter group $\{v_t\}_{t\in\RR}$ of unitary operators in $\Linf(\GG)$ such that $\tau_t^\GG=\Ad(v_t)$ for all $t$. Set $\eps_t=\|v_t-\I\|_2$ and note that
\begin{equation}\label{eq:2epst}
\|v_t\tens{v_t}-\I\tens\I\|_2=\bigl\|(v_t-\I)\tens{v_t}+\I\tens(v_t-\I)\bigr\|_2
\leq{2}\|v_t-\I\|_2=2\eps_t.
\end{equation}

For each $n$ let us fix an orthonormal basis of the carrier Hilbert space of $U^n$ in which $\uprho_{U^n}$ is diagonal with eigenvalues $\uprho_{n,1}\leq\dotsm\leq\uprho_{n,d_n}$, where $d_n=\dim{U^n}$. Then the action of the modular group of $\bh_\GG$ and of the scaling group on the matrix elements of $U^n$ with respect to this basis is
\begin{align*}
\left\{\begin{aligned}
\sigma_t^{\bh_\GG}(U^n_{k,l})&=\uprho_{n,k}^{\ii{t}}U^n_{k,l}\uprho_{n,l}^{\ii{t}}\\
\tau_t^{\GG}(U^n_{k,l})&=\uprho_{n,k}^{\ii{t}}U^n_{k,l}\uprho_{n,l}^{-\ii{t}}
\end{aligned}\right.,\qquad{t}\in\RR.
\end{align*}
(cf.~\eqref{sigmatau}).

Next, consider the quantum group $\HH=\GG\times\GG$ and let $X_n=U^n_{1,d_n}\tens\overline{U^n}_{\!\!\!\!\!\!\rule{0pt}{1.5Ex}d_n,1}\in\Pol(\HH)$, where the matrix elements of $\overline{U^n}$ are with respect to the basis dual to the one for the carrier Hilbert space of $U^n$. Using the fact that the lowest eigenvalue of $\uprho_{U^n}$ is $\uprho_{n,1}=\gamma(U^n)$ and the highest one is $\uprho_{n,d_n}=\Gamma(U^n)$ and that $\uprho_{\overline{U^n}}=(\uprho_{U^n}^{-1})^{\top}=\oon{diag}\bigl(\uprho_{n,1}^{-1},\dotsc,\uprho_{n,d_n}^{-1}\bigr)$ we find that
\begin{equation}\label{eq:sigmaXn}
\left\{\begin{aligned}
\sigma^{\bh_\HH}_t(X_n)&=X_n\\
\tau^{\HH}_t(X_n)&=\bigl(\tfrac{\gamma(U^n)}{\Gamma(U^n)}\bigr)^{2\ii{t}}X_n
\end{aligned}\right.,\qquad{t}\in\RR
\end{equation}
and using orthogonality relations we obtain
\begin{align*}
\|X_n\|_2^2&=\|X_n\Omega_\HH\|^2=\bh_\HH(X_n^*X_n)\\
&=\bh_\GG\bigl({U^n_{1,d_n}}^*U^n_{1,d_n}\bigr)
\bh_\GG\bigl({\overline{U^n}_{\!\!\!\!\!\!\rule{0pt}{1.5Ex}d_n,1}}^*\overline{U^n}_{\!\!\!\!\!\!\rule{0pt}{1.5Ex}d_n,1}\bigr)\\
&=\tfrac{\left(\uprho_{U^n}^{-1}\right)_{1,1}}{\qdim{U^n}}\tfrac{\left(\uprho_{\overline{U^n}}^{-1}\right)_{d_n,d_n}}{\qdim{\overline{U^n}}}
=\tfrac{\uprho_{n,1}^{-1}}{\qdim{U^n}}
\tfrac{\uprho_{n,d_n}}{\qdim{U^n}}
=\tfrac{1}{\gamma(U^n)\qdim{U^n}}
\tfrac{\Gamma(U^n)}{\qdim{U^n}}.
\end{align*}
Consequently putting $c^2=\inf\limits_{n\in\NN}\Bigl(\tfrac{1}{\gamma(U^n)\qdim{U^n}}\,\tfrac{\Gamma(U^n)}{\qdim{U^n}}\Bigr)$ (recall that by assumption $c>0$) we obtain $\|X_n\|_2\geq{c}$.

On the other hand
\begin{equation}\label{eq:nier1}
\begin{aligned}
\bigl|\bigl(\tfrac{\gamma(U^n)}{\Gamma(U^n)}\bigr)^{2\ii{t}}-1\bigr|c&
\leq\bigl|\bigl(\tfrac{\gamma(U^n)}{\Gamma(U^n)}\bigr)^{2\ii{t}}-1\bigr|\|X_n\|_2=
\bigl\|\bigl(\tfrac{\gamma(U^n)}{\Gamma(U^n)}\bigr)^{2\ii{t}}X_n\Omega_\HH-X_n\Omega_\HH\bigr\|\\
&=\bigl\|\tau_t^\HH(X_n)\Omega_\HH-X_n\Omega_\HH\bigr\|=\|(v_t\tens{v_t})X_n(v_t^*\tens{v_t^*})\Omega_\HH-X_n\Omega_\HH\|\\
&\leq\bigl\|(v_t\tens{v_t})X_n(v_t^*\tens{v_t^*})\Omega_\HH-(v_t\tens{v_t})X_n\Omega_\HH\bigr\|+\bigl\|(v_t\tens{v_t})X_n\Omega_\HH-X_n\Omega_\HH\bigr\|\\
&\leq\bigl\|(v_t\tens{v_t})X_n\bigr\|\bigl\|(v_t^*\tens{v_t^*})\Omega_\HH-\Omega_\HH\bigr\|+\bigl\|(v_t\tens{v_t})X_n\Omega_\HH-X_n\Omega_\HH\bigr\|.
\end{aligned}
\end{equation}
Now, since $\sigma_t^{\bh_\HH}(v_t\tens{v_t})=\sigma_t^{\bh_\GG}(v_t)\tens\sigma_t^{\bh_\GG}(v_t)=v_t\tens{v_t}$, we have $\modOphalf[\bh_\HH](v_t\tens{v_t})\Omega_\HH=(v_t\tens{v_t})\Omega_\HH$, so
\begin{align*}
\bigl\|(v_t^*\tens{v_t^*})\Omega_\HH-\Omega_\HH\bigr\|
&=\bigl\|J_{\bh_\HH}\modOphalf[\bh_\HH](v_t\tens{v_t})\Omega_\HH-\Omega_\HH\bigr\|=\bigl\|J_{\bh_\HH}(v_t\tens{v_t})\Omega_\HH-\Omega_\HH\bigr\|\\
&=\bigl\|J_{\bh_\HH}(v_t\tens{v_t})\Omega_\HH-J_{\bh_\HH}\Omega_\HH\bigr\|
=\bigl\|(v_t\tens{v_t})\Omega_\HH-\Omega_\HH\bigr\|.
\end{align*}
In view of this, the fact that $\|X_n\|\leq{1}$, and \eqref{eq:2epst} the estimate \eqref{eq:nier1} can be continued as
\begin{equation}\label{eq:nier2}
\begin{aligned}
\bigl|\bigl(\tfrac{\gamma(U^n)}{\Gamma(U^n)}\bigr)^{2\ii{t}}-1\bigr|c&\leq\bigl\|(v_t\tens{v_t})X_n\bigr\|\bigl\|(v_t\tens{v_t})\Omega_\HH-\Omega_\HH\bigr\|+\bigl\|(v_t\tens{v_t})X_n\Omega_\HH-X_n\Omega_\HH\bigr\|\\
&\leq2\eps_t+\bigl\|(v_t\tens{v_t})X_n\Omega_\HH-X_n\Omega_\HH\bigr\|.
\end{aligned}
\end{equation}
Now by \eqref{eq:sigmaXn} we have $X_n\in\Linf(\HH)^\sigma$, so
\begin{align*}
(v_t\tens{v_t})X_n\Omega_\HH
&=J_{\bh_\HH}\modOphalf[\bh_\HH]X_n^*J_{\bh_\HH}\modOphalf[\bh_\HH](v_t\tens{v_t})\Omega_\HH\\
&=J_{\bh_\HH}\modOphalf[\bh_\HH]X_n^*\modOpmhalf[\bh_\HH]J_{\bh_\HH}(v_t\tens{v_t})\Omega_\HH
=J_{\bh_\HH}X_n^*J_{\bh_\HH}(v_t\tens{v_t})\Omega_\HH,
\end{align*}
similarly $X_n\Omega_{\HH}=J_{\bh_\HH}X_n^*J_{\bh_\HH}\Omega_{\HH}$ and we can continue \eqref{eq:nier2} as
\begin{equation}\label{eq:nier3}
\begin{aligned}
\bigl|\bigl(\tfrac{\gamma(U^n)}{\Gamma(U^n)}\bigr)^{2\ii{t}}-1\bigr|c&\leq{2}\eps_t+\bigl\|J_{\bh_\HH}X_n^*J_{\bh_\HH}(v_t\tens{v_t})\Omega_\HH-J_{\bh_\HH}X_n^*J_{\bh_\HH}\Omega_\HH\bigr\|\\
&\leq2\eps_t+\|J_{\bh_\HH}X_n^*J_{\bh_\HH}\|\bigl\|(v_t\tens{v_t})\Omega_\HH-\Omega_\HH\bigr\|\leq4\eps_t.
\end{aligned}
\end{equation}

Since $t\mapsto{v_t}$ is strongly continuous, there exits $t_1>0$ such that $\eps_t<\tfrac{c}{4}$ for $t\in\left]0,t_1\right]$. Furthermore, as $\Gamma(U^n)\xrightarrow[n\to\infty]{}+\infty$ and $\gamma(U^n)\leq{1}$, there exists $m$ such that $\log\bigl(\tfrac{\Gamma(U^m)}{\gamma(U^m)}\bigr)\geq\tfrac{\pi}{2t_1}$ and thus
\[
t_2=\tfrac{\pi}{2}\log\bigl(\tfrac{\Gamma(U^m)}{\gamma(U^m)}\bigr)^{-1}
\]
satisfies $t_2\in\left]0,t_1\right]$ and we have $\bigl|\bigl(\tfrac{\Gamma(U^m)}{\gamma(U^m)}\bigr)^{2\ii{t_2}}-1\bigr|=2$. Combining this with \eqref{eq:nier3} yields the contradiction
\[
2c=\bigl|\bigl(\tfrac{\Gamma(U^m)}{\gamma(U^m)}\bigr)^{2\ii{t_2}}-1\bigr|c\leq4\eps_{t_2}<c.\qedhere
\]
\end{proof}

As already mentioned at the beginning of this section, a sequence of irreducible representations satisfying the conditions of Theorem \ref{thm:Un} can be constructed from a two dimensional representation $U$ such that $\qdim{U}>2$ (which is automatically irreducible).

\begin{proposition}\label{prop:Un}
Let $\GG$ be a compact quantum group and let $U$ be a representation of $\GG$ such that $\dim{U}=2<\qdim{U}$. Then there exists a sequence $(U^n)_{n\in\NN}$ of irreducible representations such that
\begin{enumerate}
\item\label{item:Un1} $U^1=U$ and $U^n$ is a subrepresentation of a tensor product of copies of $U$ and $\overline{U}$,
\item $\gamma(U^n)=\gamma(U)^{2n}$, $\Gamma(U^n)=\Gamma(U)^{2n}$ and $\gamma(U^n)=\Gamma(U^n)^{-1}$ for all $n\geq{2}$,\footnote{Let us also mention that the condition $\gamma(V)=\Gamma(V)^{-1}$ for any representation $V$ is automatic for many compact quantum groups (\cite{symmetry}).}
\item\label{item:Un3} $\inf\limits_{n\in\NN}\tfrac{\Gamma(U^n)}{\qdim{U^n}}>0$.
\end{enumerate}
\end{proposition}

We immediately obtain

\begin{corollary}
Let $\GG$ be a second countable compact quantum group which admits a unitary representation $U$ with $\dim{U}=2<\qdim{U}$. Then $\TtauInn(\GG)\neq\RR$.
\end{corollary}

We start with an auxiliary result of independent interest. The proof of Proposition \ref{prop:Un} will only use statement \eqref{prop1:2}.

\begin{proposition}\label{prop1}
Let $\GG$ and $\HH$ be compact quantum groups and let $\pi\colon\C^\uu(\GG)\to\C^\uu(\HH)$ be a unital $*$-homomorphism respecting coproducts. Furthermore let $U\in\C^\uu(\GG)\tens\B(\cH_U)$ be a finite dimensional unitary representation of $\GG$. Then $V=(\pi\tens\id)U$ is a unitary representation of $\HH$. Furthermore
\begin{enumerate}
\item if $V=\bigoplus\limits_{a=1}^{A}V^a$ is the decomposition of $V$ into irreducible representations, then
\[
\bigcup\limits_{a=1}^{A}\oon{Sp}\bigl(\uprho_{V^a}\tens\uprho_{V^a}^{-1}\bigr)\subset\oon{Sp}\bigl(\uprho_U\tens\uprho_U^{-1}\bigr),
\]
\item\label{prop1:2} if $V$ is irreducible, then $\uprho_V=\uprho_U$.
\end{enumerate}
\end{proposition}

\begin{remark}
The example of $\GG=\oon{U}_F^+$ with fundamental representation $U$ satisfying $1\not\in\oon{Sp}(\uprho_U)$, $\HH=\{e\}$, and $\pi$ equal to the counit shows that $\oon{Sp}(\uprho_V)\subset\oon{Sp}(\uprho_U)$ can fail.
\end{remark}

\begin{proof}[Proof of Proposition \ref{prop1}]
It is clear that $V$ is a unitary representation of $\HH$. Let us decompose $V$ into irreducible representations $V=\bigoplus\limits_{a=1}^{A}V^a$, and let $\cH_U=\cH_V=\bigoplus\limits_{a=1}^{A}\cH_{V^a}$ be the corresponding decomposition of Hilbert spaces. Fix orthonormal bases in the carrier Hilbert spaces $\cH_U$ and $\cH_{V^1},\dotsc\cH_{V^A}$ in which the corresponding operators $\uprho_U,\uprho_{V^1},\dotsc,\uprho_{V^A}$ are diagonal with eigenvalues $\uprho_{U,i}$ and $\uprho_{V^a,j}$ ($1\leq{i}\leq\dim{U}$, $1\leq{a}\leq{A}$, $1\leq{j}\leq\dim{V^a}$) and let $\bigl\{e^a_{k,l}\,\bigr|\bigl.\,1\leq{k,l}\leq\dim{V^a}\bigr\}$ and $\bigl\{e^U_{i,j}\,\bigr|\bigl.\,1\leq{i,j}\leq\dim{U}\bigr\}$ be the corresponding matrix units. We have
\[
\sum_{a=1}^{A}\sum_{k,l=1}^{\dim{V^a}}V^a_{k,l}\tens{e^a_{k,l}}=\bigoplus_{a=1}^{A}V^a=V=(\pi\tens\id)U=\sum_{i,j=1}^{\dim{U}}\pi(U_{i,j})\tens{e^U_{i,j}}.
\]
Applying $\tau_t^{\HH}$ gives 
\begin{align*}
\sum_{a=1}^{A}\sum_{k,l=1}^{\dim{V^a}}\bigl(\tfrac{\uprho_{V^a,k}}{\uprho_{V^a,l}}\bigr)^{\ii{t}}V^a_{k,l}\tens{e^a_{k,l}}
&=(\tau^{\HH}_t\tens\id)V=\bigl((\tau^\HH_t\comp\pi)\tens\id\bigr)U\\
&=\bigl((\pi\comp\tau^{\GG}_t)\tens\id\bigr)U=\sum_{i,j=1}^{\dim{U}}\bigl(\tfrac{\uprho_{U,i}}{\uprho_{U,j}}\bigr)^{\ii{t}}\pi(U_{i,j})\tens{e^U_{i,j}}.
\end{align*}
Take $\tfrac{\lambda}{\lambda'}$ for some $1\leq{a'}\leq{A}$, $\lambda,\lambda'\in\oon{Sp}(\uprho_{V^{a'}})$ and let $m$ be an invariant mean on $\RR$. The above equality implies that for arbitrary $\xi,\xi'\in\Ltwo(\GG)\tens\cH_U$ we have
\begin{align*}
\sum_{\substack{{a,k,l\text{ such that}}\\{\uprho_{V^a,k}/\uprho_{V^a,l}=\lambda/\lambda'}}}
\is{\xi}{(V^a_{k,l}\tens{e^a_{k,l}})\xi'}
&=
m\biggl(t\mapsto
\sum_{a=1}^{A}\sum_{k,l=1}^{\dim{V^a}}\bigl(\tfrac{\lambda'}{\lambda}\bigr)^{\ii{t}}\bigl(\tfrac{\uprho_{V^a,k}}{\uprho_{V^a,l}}\bigr)^{\ii{t}}\is{\xi}{(V^a_{k,l}\tens{e^a_{k,l}})\xi'}\biggr)\\
&=m\biggl(t\mapsto
\sum_{i,j=1}^{\dim{U}}\bigl(\tfrac{\lambda'}{\lambda}\bigr)^{\ii{t}}\bigl(\tfrac{\uprho_{U,i}}{\uprho_{U,j}}\bigr)^{\ii{t}}\is{\xi}{\bigl(\pi(U_{i,j})\tens{e^U_{i,j}}\bigr)\xi'}\biggr)\\
&=
\sum_{\substack{{i,j\text{ such that}}\\{\uprho_{U,i}/\uprho_{U,j}=\lambda/\lambda'}}}
\is{\xi}{\bigl(\pi(U_{i,j})\tens{e^U_{i,j}}\bigr)\xi'}
\end{align*}
and hence we can conclude that
\[
\sum_{\substack{{a,k,l\text{ such that}}\\{\uprho_{V^a,k}/\uprho_{V^a,l}=\lambda/\lambda'}}}
V^a_{k,l}\tens{e^a_{k,l}}=\sum_{\substack{{i,j\text{ such that}}\\{\uprho_{U,i}/\uprho_{U,j}=\lambda/\lambda'}}}\pi(U_{i,j})\tens{e^U_{i,j}}.
\]
The left hand side is by assumption non-zero, hence so is the right hand side and $\tfrac{\lambda}{\lambda'}\in\oon{Sp}(\uprho_U)$.

Ad \eqref{prop1:2}. Assume now that $V$ is irreducible. To conclude that $\uprho_V=\uprho_U$, it is enough to show that $\uprho_U\in\oon{Mor}(V,V^{\text{\rm{\tiny{cc}}}})$ and $\Tr(\uprho_U)=\Tr(\uprho_U^{-1})$ (\cite[Proposition 1.4.4]{NeshveyevTuset}). Clearly the second property holds, for the first one recall that $\pi$ preserves the antipode (\cite[Proposition 1.3.17]{Timmermann}) which allows us to calculate as follows
\begin{align*}
(\I\tens\uprho_U)V&=(\pi\tens\id)\bigl((\I\tens\uprho_U)U\bigr)\\
&=(\pi\tens\id)\bigl(U^{\text{\rm{\tiny{cc}}}}(\I\tens\uprho_U)\bigr)\\
&=\bigl((\pi\comp{S^2_{\GG}})\tens\id)\bigl(U(\I\tens\uprho_U)\bigr)\\
&=\bigl((S^2_{\HH}\comp\pi)\tens\id\bigr)\bigl(U(\I\tens\uprho_U)\bigr)
=(S^2_{\HH}\tens\id)\bigl(V(\I\tens\uprho_U)\bigr)=V^{\text{\rm{\tiny{cc}}}}(\I\tens\uprho_U).
\end{align*}
This shows that $\uprho_V=\uprho_U$.
\end{proof}

In the proof of Proposition \ref{prop:Un}, $\oon{U}_F^+$ will denote the universal unitary quantum group corresponding to an invertible matrix $F$ constructed in \cite{WangVanDaele} (denoted by $A_u(Q)$ in that paper). Let us recall that irreducible representations of $\oon{U}_F^+$ can be labeled by free product of monoids $\ZZ_+\star\ZZ_+=\langle\alpha,\beta\rangle$ in such a way that $\alpha$ corresponds to the canonical fundamental representation and passing to the conjugate representation corresponds to the unique anti-multiplicative map $x\mapsto\overline{x}$ on $\ZZ_+\star\ZZ_+$ satisfying $\overline{\alpha}=\beta$ and $\overline{\beta}=\alpha$. Moreover, denoting by $u^x$ a unitary representation of $\oon{U}_F^+$ in the class $x\in\ZZ_+\star\ZZ_+$ the tensor product of representations from the classes $x,y\in\ZZ_+\star\ZZ_+$ decomposes into irreducible as
\begin{equation}\label{eq41}
u^x\tp{u^y}
=\bigoplus_{\substack{{a,b,c\,\in\,\ZZ_+\star\ZZ_+}\\{x=ac,\:y=\overline{c}b}}}u^{ab}
\end{equation}
up to equivalence (cf.~\cite[Th\'eor\`eme 1]{banica}).

For $n\in\NN$ and $\delta\in\{\alpha,\beta\}$ let $w^n_\delta$ be the word $\delta\overline{\delta}\delta\dotsm\in\ZZ_+\star\ZZ_+$ ($n$ letters). It will be convenient to denote the (class of) the trivial representation by $w^0_\delta$. The dimension and quantum dimension of $w^n_\delta$ are known \cite[Lemma 4.13]{KrajczokWasilewski}.

\begin{lemma}
For $\delta\in\{\alpha,\beta\}$ and $n\in\ZZ_+$ we have
\begin{equation}\label{dimdeltan}
\begin{aligned}
\gamma(w^{2n}_\delta)&=\gamma(\delta)^n\gamma(\overline{\delta})^n,
&\Gamma(w^{2n}_\delta)&=\Gamma(\delta)^n\Gamma(\overline{\delta})^n,\\
\gamma(w^{2n+1}_\delta)&=\gamma(\delta)^{n+1}\gamma(\overline{\delta})^n,
&\Gamma(w^{2n+1}_\delta)&=\Gamma(\delta)^{n+1}\Gamma(\overline{\delta})^n.
\end{aligned}
\end{equation}
\end{lemma}

\begin{proof}
It is enough to consider the case $\gamma(\alpha)<1$, otherwise $\oon{U}_F^+$ is of Kac type and the result holds trivially.

We prove the claim by induction. The equalities \eqref{dimdeltan} hold for $n=0$, so let us assume their validity for some $n\in\ZZ_+$. We have $w_\delta^{2n+1}=w_\delta^{2n}\delta$, hence $w_\delta^{2n+1}\tp\overline{\delta}=w_\delta^{2n+2}\oplus{w_\delta^{2n}}$ and
\begin{align*}
\gamma(\delta)^{n+1}\gamma(\overline{\delta})^{n+1}
&=\gamma(w_\delta^{2n+1})\gamma(\overline{\delta})=\gamma(w_{\delta}^{2n+1}\tp\overline{\delta})=
\gamma(w_\delta^{2n+2}\oplus{w_\delta^{2n}})\\
&=\min\bigl\{\gamma(w_\delta^{2n+2}),\gamma(w_\delta^{2n})\bigr\}=
\min\bigl\{\gamma(w_\delta^{2n+2}),\gamma(\delta)^n\gamma(\overline{\delta})^n\bigr\}=\gamma(w_\delta^{2n+2}).
\end{align*}
This shows $\gamma(w_\delta^{2n+2})=\gamma(\delta)^{n+1}\gamma(\overline{\delta})^{n+1}$. Similarly $w_\delta^{2n+2}\tp\delta=w_\delta^{2n+3}\oplus{w_\delta^{2n+1}}$ and, as above, we deduce that $\gamma(w_\delta^{2n+3})=\gamma(\delta)^{n+2}\gamma(\overline{\delta})^{n+1}$. The argument for $\Gamma$ is analogous.
\end{proof}

\begin{proof}[Proof of Proposition \ref{prop:Un}]
By \cite[Proposition 6.4.5]{Timmermann} (cf.~\cite[Theorem 1.3(2)]{WangVanDaele}) there exists $F\in\oon{GL}(2,\CC)$ and a unital $*$-homomorphism $\pi\colon\C^\uu(\oon{U}_F^+)\to\C^\uu(\GG)$ respecting coproducts and mapping the matrix elements of the fundamental representation $u^\alpha$ of $\oon{U}_F^+$ to the corresponding matrix elements of $U$. Since $\dim{u^\alpha}=2$, we have $\gamma(u^\alpha)=\Gamma(u^\alpha)^{-1}=\gamma(\overline{u^\alpha})$.

Let us introduce a family $\{W^n\}_{n\in\ZZ_+}$ of (not necessarily irreducible) representations of $\GG$ via $W^n=(\pi\tens\id)u^{w^n_\alpha}$ for $n\in\ZZ_+$. In particular $W^0$ is the trivial representation and $W^1=(\pi\tens\id)u^\alpha=U$. It is irreducible, hence by Proposition \ref{prop1}\eqref{prop1:2} $\uprho_{W^1}=\uprho_U=\uprho_{u^\alpha}$ and consequently
\[
\gamma(U)=\gamma(u^\alpha),\quad\Gamma(U)=\Gamma(u^\alpha).
\]
For each $n\in\NN$, choose an irreducible subrepresentation $\mathcal{U}^n$ of $W^n$ such that $\gamma(\mathcal{U}^n)=\gamma(W^n)$. Next, set $U^1=\mathcal{U}^1=U$ and $U^n=\mathcal{U}^{2n}$ for $n\geq{2}$. We claim that the family $\{U^n\}_{n\in\NN}$ meets the criteria of Proposition \ref{prop:Un}.

First, as $u^{w^n_\alpha}\subset{u^{\alpha}}\tp\overline{u^{\alpha}}\tp\dotsm$ ($n$ times), we have $\mathcal{U}^n\subset{W^n}=(\pi\tens\id)u^{w^n_\alpha}\subset{U}\tp\overline{U}\tp\dotsm$, which shows \eqref{item:Un1}. Next, for $n\geq{1}$ and $\delta=\alpha$ if $n$ is even, $\delta=\beta$ if $n$ is odd, we calculate:
\[
u^{w^n_\alpha}\tp{u^\delta}=u^{w^{n+1}_\alpha}\oplus{u^{w^{n-1}_\alpha}}
\]
hence
\begin{equation}\label{eq1}
\bigl((\pi\tens\id)u^{w^n_\alpha}\bigr)\tp\bigl((\pi\tens\id)u^\delta\bigr)
=\bigl((\pi\tens\id)u^{w^{n+1}_\alpha}\bigr)\oplus\bigl((\pi\tens\id)u^{w^{n-1}_\alpha}\bigr)
\end{equation}
and
\begin{align*}
\gamma(W^{n})\gamma(U)
&=\gamma\bigl((\pi\tens\id)u^{w^n_\alpha}\bigr)\gamma(u^\delta)
=\gamma\Bigl(\bigl((\pi\tens\id)(u^{w^n_\alpha})\bigr)\tp\bigl((\pi\tens\id)u^\delta\bigr)\Bigr)\\
&=\gamma\Bigl(\bigl((\pi\tens\id)u^{w^{n+1}_\alpha}\bigr)\oplus\bigl((\pi\tens\id)u^{w^{n-1}_\alpha}\bigr)\Bigr)
=\min\bigl\{\gamma(W^{n+1}),\gamma(W^{n-1})\bigr\}.
\end{align*}
As $\gamma(W^0)=1$ and $\gamma(U)<1$, we have by induction $\gamma(W^{n+1})<\gamma(W^n)$ and hence
\[
\gamma(W^n)\gamma(U)=\gamma(W^{n+1}).
\]
From this we can conclude that
\begin{subequations}\label{eq2}
\begin{equation}
\gamma(\mathcal{U}^n)=\gamma(W^n)=\gamma(U)^n=\gamma(u^{w^n_\alpha}),\qquad{n}\geq{1},
\end{equation}
in particular
\begin{equation}
\gamma(U^n)=\gamma(\mathcal{U}^{2n})=\gamma(U)^{2n},\qquad{n}\geq{2}
\end{equation}
and similarly
\begin{equation}
\Gamma(W^n)=\Gamma(U)^n=\Gamma(u^{w^n_\alpha}),\qquad{n}\ge{1}.
\end{equation}
\end{subequations}

A similar inductive reasoning in conjunction with $\qdim{W^1}=\qdim{u^\alpha}$ shows that \eqref{eq1} implies also
\begin{equation}\label{eq3}
\qdim{W^n}=\qdim{u^{w^n_\alpha}}\qquad{n}\in\NN.
\end{equation}
The next step is to show that $\Gamma(U^n)=\Gamma(\mathcal{U}^{2n})=\Gamma(W^{2n})$ for $n\geq{2}$. For this, we will use \cite[Theorem 3.4]{symmetry}. To that end, for any finite dimensional unitary representation $Z_0$ of a compact quantum group let us introduce function
\[
P_{Z_0}\colon\NN\ni{m}\longmapsto\max_{Z\subset{Z_0^{\tp{m}}}}\dim{Z}\in\NN,
\]
where the maximum runs over all representations $Z$ appearing in the decomposition of $Z_0^{\tp{m}}$ into irreducible representations. Since $\mathcal{U}^{2n}\subset{W^{2n}}$, we have $P_{\mathcal{U}^{2n}}\leq{P_{W^{2n}}}$. Next, because $W^{2n}$ is the image of $u^{w^{2n}_\alpha}$ under $\pi\tens\id$, we also have $P_{W^{2n}}\leq{P_{u^{w^{2n}_\alpha}}}$. The fusion rules in $\oon{U}_F^+$ show that any representation $Z$ which appears in the decomposition of $(u^{w^{2n}_\alpha})^{\tp{m}}$ is of the form $u^{w^{k}_\alpha}$ for $1\leq{k}\leq{2nm}$. Thus, using \cite[Lemma 4.13]{KrajczokWasilewski} we obtain
\[
P_{\mathcal{U}^{2n}}(m)\leq{P_{u^{w^{2n}_\alpha}}(m)}\leq\max_{1\leq{k}\leq{2nm}}\dim{u^{w^k_\alpha}}=2nm+1
\]
(remember that we assume $\dim{U}=\dim{u^\alpha}=2$). Since the function $P_{\mathcal{U}^{2n}}$ is of subexponential growth and by definition $U^n=\mathcal{U}^{2n}$, \cite[Theorem 3.4]{symmetry} and equalities \eqref{eq2} imply
\[
\Gamma(U^n)=
\Gamma(\mathcal{U}^{2n})=\gamma(\mathcal{U}^{2n})^{-1}=
\gamma(U)^{-2n}=\Gamma(U)^{2n}
=\gamma(U^n)^{-1},\qquad{n}\in\NN.
\]
It is left to check condition \eqref{item:Un3}. Set $q=\gamma(U)$ and note immediately that $0<q<1$. Since $\dim{U}=2$ we have $\qdim{U}=\qdim{u^\alpha}=q+q^{-1}$. Equation \eqref{eq3} and \cite[Lemma 4.13]{KrajczokWasilewski} give
\[
\qdim{\mathcal{U}^n}\leq\qdim{W^n}=\qdim{u^{w^n_\alpha}}=[n+1]_q,
\]
where $[n+1]_q=\tfrac{q^{-n-1}-q^{n+1}}{q^{-1}-q}$ ($n\in\ZZ_+$) are the $q$-numbers (see e.g.~\cite[Section 2.1.1]{KlimykSchmudgen} or \cite[Section 2.4]{NeshveyevTuset}). To conclude the proof we calculate for $n\geq{2}$
\begin{align*}
\tfrac{\Gamma(U^n)}{\qdim{U^n}}=
\tfrac{\Gamma(\mathcal{U}^{2n})}{\qdim{\mathcal{U}^{2n}}}=
\tfrac{1}{\gamma(U)^{2n}\qdim{\mathcal{U}^{2n}}}\geq\tfrac{q^{-1}-q}{q^{2n}(q^{-2n-1}-q^{2n+1})}
\xrightarrow[n\to\infty]{}\tfrac{q^{-1}-q}{q^{-1}}=1-q^2>0.
\end{align*}
\end{proof}

\subsection{Duals of type \texorpdfstring{$\mathrm{I}$}{I} discrete quantum groups}\label{sec:typeI}\hspace*{\fill}

In this section we prove validity of Conjecture \ref{conj:main} for duals of second countable discrete quantum groups of type $\mathrm{I}$. This class includes all $q$-deformations of compact, simply connected, semisimple Lie groups such as $\oon{SU}_q(n)$ (see Section \ref{sec:qDeformations}). Additionally the techniques used in the proof of the main technical result (Theorem \ref{thm:Lambdat}) owe much to \cite{qdisk,typeI} and \cite[Section 6]{faktory}.

For a thorough account of the theory of locally compact quantum groups of type $\mathrm{I}$ we refer the reader to \cite{modular,KrajczokTypeI,CaspersKoelink}, while here we will only list the important objects associated with a second countable discrete quantum group $\bbGamma$ of type $\mathrm{I}$, which we will need in what follows.

Let $\bbGamma$ be a second countable discrete quantum group of type $\mathrm{I}$. Then $\hh{\bbGamma}$ is coamenable (see e.g.~\cite[Theorem 2.8]{typeI}) and $\Irr{\bbGamma}$, the spectrum of $\C(\hh{\bbGamma})$, is a standard Borel space. There exists a unitary operator
\[
\cQ_\LL\colon\ell^2(\bbGamma)\longrightarrow{\Int_{\Irr{\bbGamma}}}^{\!\!\oplus}\oon{HS}(\cH_\pi)\dd{\mu}(\pi)
\]
(with $\oon{HS}(\cH_\pi)$ denoting the space of Hilbert-Schmidt operators on the carrier Hilbert space $\cH_\pi$ of the representation $\pi$). The unitary $\cQ_\LL$ implements the isomorphism
\[
\Linf(\hh{\bbGamma})\ni{x}\longmapsto\cQ_\LL{x}\cQ_\LL^*\in{\Int_{\Irr{\bbGamma}}}^{\!\!\oplus}\bigl(\B(\cH_\pi)\tens\I_{\overline{\cH_\pi}}\bigr)\dd{\mu}(\pi)
\]
(here we use the isomorphism $\oon{HS}(\cH_\pi)=\cH_\pi\tens\overline{\cH_\pi}$ which yields $\B(\oon{HS}(\cH_\pi))=\B(\cH_\pi)\vtens\B(\overline{\cH_\pi})$) and writing $\cQ_\LL{x}\cQ_\LL^*$ as ${\Int_{\Irr{\bbGamma}}}^{\!\!\!\oplus}x_\pi\tens\I_{\overline{\cH_\pi}}\dd{\mu}(\pi)$ the value of the Haar measure $\bh_{\hh{\bbGamma}}$ on $x$ is expressed as follows
\[
\bh_{\hh{\bbGamma}}(x)=\Int_{\Irr{\bbGamma}}\Tr(D_\pi^{-2}x_\pi)\dd{\mu}(\pi)
\]
for a certain measurable field $\bigl(D_\pi\bigr)_{\pi\in\Irr{\bbGamma}}$
of non-singular positive self-adjoint operators.

\begin{theorem}\label{thm:Lambdat}
Let $\bbGamma$ be a second countable type $\mathrm{I}$ discrete quantum group. Then
\begin{enumerate}
\item\label{thm:Lambdat1} for any $\alpha\in\Irr{\hh{\bbGamma}}$ we have $\gamma(\alpha)=\Gamma(\alpha)^{-1}$,
\item\label{thm:Lambdat2} for any $t\in\TtauInn(\hh{\bbGamma})$ the set $\Lambda_t=\bigl\{\Gamma(\beta)^{2\ii{t}}\,\bigr|\bigl.\,\beta\in\Irr{\hh{\bbGamma}}\bigr\}\subset\TT$ is finite.
\end{enumerate}
\end{theorem}

\begin{proof}
Since $\hh{\bbGamma}$ is compact, we have
\[
+\infty>\bh_{\hh{\bbGamma}}(\I)=\Int_{\Irr{\bbGamma}}\Tr(D_\pi^{-2})\dd{\mu}(\pi),
\]
so $D_\pi^{-1}$ is a Hilbert-Schmidt operator for $\pi$ in some full-measure subset $O_1\subset\Irr{\bbGamma}$. By rescaling $\mu$ we can assume that $\|D_\pi^{-1}\|=1$ for almost all $\pi\in\Irr{\bbGamma}$. Next we recall from \cite[Theorem 5.4]{modular} that
\begin{equation}\label{eq:NablaInt}
\cQ_\LL\modOp[\bh_{\hh{\bbGamma}}]^{\ii{s}}\cQ_\LL^*={\Int_{\Irr{\bbGamma}}}^{\!\!\oplus}D_\pi^{-2\ii{s}}\tens(D_\pi^{2\ii{s}})^\top\dd{\mu}(\pi),\qquad{s}\in\RR.
\end{equation}

Returning to the compact quantum group $\hh{\bbGamma}$ we use standard conventions to fix a unitary representative $U^\alpha$ of each class $\alpha\in\Irr{\hh{\bbGamma}}$ and we fix an orthonormal basis of the carrier Hilbert space in which $\uprho_\alpha$ (short-hand for $\uprho_{U^\alpha}$) is diagonal with diagonal elements $\uprho_{\alpha,1}\leq\dotsm\leq\uprho_{\alpha,d_\alpha}$ (similarly to what we did in Section \ref{sect:UqU}, we denote $\dim{U^\alpha}$ by $d_\alpha$). If $\uprho_\alpha=\I$, there is nothing to prove, so from now on we will only deal with $\alpha$ such that $\gamma(\alpha)<1<\Gamma(\alpha)$. For $i,j\in\{1,\dotsc,d_\alpha\}$ write
\[
\cQ_\LL{U_{i,j}^\alpha}\cQ_\LL^*={\Int_{\Irr{\bbGamma}}}^{\!\!\oplus}U_{i,j;\pi}^\alpha\tens\I_{\overline{\cH_\pi}}\dd{\mu}(\pi)
\]
for some measurable field of operators $\bigl(U_{i,j;\pi}^\alpha\bigr)_{\pi\in\Irr{\bbGamma}}$. Similarly, for $t\in\TtauInn(\hh{\bbGamma})$ we choose a unitary $v_t\in\Linf(\hh{\bbGamma})$ such that $\tau_t^{\hh{\bbGamma}}=\Ad(v_t)$ and write
\[
\cQ_\LL{v_t}\cQ_\LL^*={\Int_{\Irr{\bbGamma}}}^{\!\!\oplus}v_{t;\pi}\tens\I_{\overline{\cH_\pi}}\dd{\mu}(\pi)
\]
for some measurable field $\bigl(v_{t;\pi}\bigr)_{\pi\in\Irr{\bbGamma}}$ of unitary operators. Consequently
\begin{equation}\label{eq:tauInt}
\cQ_\LL\tau_t^{\hh{\bbGamma}}(U_{i,j}^\alpha)\cQ_\LL^*={\Int_{\Irr{\bbGamma}}}^{\!\!\oplus}v_{t;\pi}U_{i,j;\pi}^\alpha{v_{t;\pi}^*}\tens\I_{\overline{\cH_\pi}}\dd{\mu}(\pi)
\end{equation}
and in the same way, from \eqref{eq:NablaInt} we get
\begin{equation}\label{eq:NablaInt2}
\cQ_\LL\sigma_s^{\bh_{\hh{\bbGamma}}}(U_{i,j}^\alpha)\cQ_\LL^*={\Int_{\Irr{\bbGamma}}}^{\!\!\oplus}D_\pi^{-2\ii{s}}U_{i,j;\pi}^\alpha{D_\pi^{2\ii{s}}}\tens\I_{\overline{\cH_\pi}}\dd{\mu}(\pi),\qquad{s}\in\RR.
\end{equation}

Now from Lemma \ref{lem:centrv} we know that each $v_t$ is invariant under the modular group, so in view of \eqref{eq:NablaInt2} we obtain
\[
{\Int_{\Irr{\bbGamma}}}^{\!\!\oplus}D_\pi^{-2\ii{s}}v_{t;\pi}D_\pi^{2\ii{s}}\tens\I_{\overline{\cH_\pi}}\dd{\mu}(\pi)
={\Int_{\Irr{\bbGamma}}}^{\!\!\oplus}v_{t;\pi}^\alpha\tens\I_{\overline{\cH_\pi}}\dd{\mu}(\pi),\qquad{s}\in\RR.
\]
Hence for each $s$ there is a full-measure subset of $\Irr{\bbGamma}$ consisting only of $\pi$ such that $D_\pi^{-2\ii{s}}v_{t;\pi}D_\pi^{2\ii{s}}=v_{t;\pi}$. Taking the intersection of these subsets over rational $s$ and then intersection with $O_1$ we obtain a subset $O_2\subset{O_1}$ of full measure such that
\begin{equation}\label{eq:DvD}
D_\pi^{-2\ii{s}}v_{t;\pi}D_\pi^{2\ii{s}}=v_{t;\pi},\qquad\pi\in{O_2},\:s\in\RR
\end{equation}
(at first this holds only for $s\in\QQ$, but then extends to $\RR$ by continuity). It follows that for $\pi\in{O_2}$ the operator $v_{t;\pi}$ commutes with $f(D_\pi)$ for any bounded measurable function $f$ on $\oon{Sp}(D_\pi)$.

Take now $0<a<b$. We will show that for $\pi$ in a full-measure subset of $\Irr{\bbGamma}$ the operator $U_{i,j;\pi}^\alpha$ shifts the spectral subspace $\cH_\pi(a\leq{D_\pi}\leq{b})$ as follows:
\begin{equation}\label{eq:UDDU}
U_{i,j;\pi}^\alpha\bigl(\cH_\pi(a\leq{D_\pi}\leq{b})\bigr)\subset\cH_\pi\bigl((\uprho_{\alpha,i}\uprho_{\alpha,j})^{\mhalf}a\leq{D_\pi}\leq(\uprho_{\alpha,i}\uprho_{\alpha,j})^{\mhalf}b\bigr).
\end{equation}
Indeed, \eqref{sigmatau} and \eqref{eq:NablaInt2} imply
\[
D_\pi^{-2\ii{s}}U_{i,j;\pi}^\alpha{D_\pi^{2\ii{s}}}=(\uprho_{\alpha,i}\uprho_{\alpha,j})^{\ii{s}}U_{i,j;\pi}^\alpha,\qquad{s}\in\RR
\]
for $\pi$ in a full-measure subset $O_{3,i,j}\subset\Irr{\bbGamma}$. Setting $O_3=O_2\cap\bigcap\limits_{i,j=1}^{d_\alpha}O_{3,i,j}$ we obtain
\[
U_{i,j;\pi}^\alpha{D_\pi^{2\ii{s}}}=\bigl((\uprho_{\alpha,i}\uprho_{\alpha,j})^{\half}D_\pi\bigr)^{2\ii{s}}U_{i,j;\pi}^\alpha,\qquad{s}\in\RR,\:\pi\in{O_3}
\]
with $O_3$ of full measure. Thus for any Schwartz class function $g$
\begin{align*}
U_{i,j;\pi}^\alpha{g(\log{D_\pi})}
&=\tfrac{1}{\sqrt{2\pi}}\Int_{-\infty}^{+\infty}U_{i,j;\pi}^\alpha{D_\pi^{\ii{t}}}\hh{g}(t)\dd{t}\\
&=\tfrac{1}{\sqrt{2\pi}}\Int_{-\infty}^{+\infty}\bigl((\uprho_{\alpha,i}\uprho_{\alpha,j})^{\half}D_\pi\bigr)^{\ii{t}}U_{i,j;\pi}^\alpha\hh{g}(t)\dd{t}
=g\bigl(\log((\uprho_{\alpha,i}\uprho_{\alpha,j})^{\half}D_\pi)\bigr)U_{i,j;\pi}^\alpha.
\end{align*}
Approximating the characteristic function of $[\log{a},\log{b}]$ pointwise with $g$ as above yields
\begin{align*}
U_{i,j;\pi}^\alpha\chi(a\leq{D_\pi}\leq{b})
&=\chi\bigl(a\leq(\uprho_{\alpha,i}\uprho_{\alpha,j})^{\half}{D_\pi}\leq{b}\bigr)U_{i,j;\pi}^\alpha\\
&=\chi\bigl((\uprho_{\alpha,i}\uprho_{\alpha,j})^{\mhalf}a\leq{D_\pi}\leq(\uprho_{\alpha,i}\uprho_{\alpha,j})^{\mhalf}b\bigr)U_{i,j;\pi}^\alpha
\end{align*}
and \eqref{eq:UDDU} follows.

Similarly the operators $U_{i,j;\pi}^\alpha$ also shift eigenspaces of $v_{t;\pi}$.
More precisely if $V$ is a measurable subset of $\TT$ then then for almost all $\pi$
\begin{equation}\label{eq:UvvU}
U_{i,j;\pi}^\alpha\bigl(\cH_\pi(v_{t;\pi}\in{V})\bigr)\subset\cH(v_{t;\pi}\in\uprho_{\alpha,i}^{\ii{t}}\uprho_{\alpha,j}^{-\ii{t}}V)
\end{equation}
because \eqref{sigmatau} and \eqref{eq:tauInt} imply
\[
v_{t;\pi}U_{i,j;\pi}^\alpha{v_{t;\pi}^*}=
\uprho_{\alpha,i}^{\ii{t}}U_{i,j;\pi}^\alpha\uprho_{\alpha,j}^{-\ii{t}},\qquad{i,j}\in\{1,\dotsc,d_\alpha\},\:\pi\in{O_4}
\]
for a full-measure subset $O_4\subset{O_3}$.

Recall that for $\pi\in{O_4}$ the operator $D_\pi^{-1}$ is Hilbert-Schmidt and we have $\|D_\pi^{-1}\|=1$. Consequently there exists (finite or infinite, depending on the dimension of $\pi$) sequence $1=q_{1;\pi}>{q_{2;\pi}}>\dotsm$ of non-zero positive numbers such that each eigenspace $\cH_\pi(D_\pi^{-1}=q_{n;\pi})$ is finite dimensional and
\[
\cH_\pi=\bigoplus_{n}\cH_\pi(D_\pi^{-1}=q_{n;\pi}).
\]

Equation \eqref{eq:DvD} implies that $v_{t;\pi}$ preserves the decomposition of $\cH_\pi$ into eigenspaces of $D_{\pi}^{-1}$. In particular $v_{t;\pi}$ preserves the finite dimensional eigenspace $\cH_\pi(D_\pi^{-1}=1)=\cH_\pi(D_\pi=1)$ which consequently decomposes as
\[
\cH_\pi(D_\pi=1)=\bigoplus_{k=1}^{K}\cH_\pi(v_{t;\pi}=\mu_k)\cap\cH_\pi(D_\pi=1)
\]
for some $\mu_1,\dotsc,\mu_K\in\TT$. By \eqref{eq:UDDU} and \eqref{eq:UvvU}
\begin{equation}\label{eq:cont}
\begin{aligned}
U_{i,d_\alpha;\pi}^\alpha\bigl(\cH_\pi(v_{t;\pi}=\mu_k)&\cap\cH_\pi(D_\pi=1)\bigr)\\
&\subset\cH_\pi\bigl(v_{t;\pi}
=\uprho_{\alpha,i}^{\ii{t}}\Gamma(\alpha)^{-\ii{t}}\mu_k\bigr)\cap\cH_\pi\Bigl(D_\pi=\bigl(\uprho_{\alpha,i}\Gamma(\alpha)\bigr)^{\mhalf}\Bigr).
\end{aligned}
\end{equation}

Finally let us note that for each $\beta\in\Irr{\hh{\bbGamma}}$ we have $\sum\limits_{i=1}^{d_{\beta}}{U_{i,d_{\beta}}^\beta}^*U_{i,d_{\beta}}^\beta=\I$ which implies that
\begin{equation}\label{eq:UU1pi}
\sum_{i=1}^{\dim{\beta}}{U_{i,d_{\beta};\pi}^\beta}^*U_{i,d_{\beta};\pi}^\beta=\I
\end{equation}
for all $\pi$ in some full-measure subset $O_5\subset{O_4}$ and all $\beta\in\Irr{\hh{\bbGamma}}$ (recall that $\Irr{\hh{\bbGamma}}$ is countable).

We are now ready to prove the two claims of our theorem. If $\gamma(\alpha)>\Gamma(\alpha)^{-1}$ then for any $i\in\{1,\dotsc,d_\alpha\}$
\[
\uprho_{\alpha,i}\Gamma(\alpha)\geq\gamma(\alpha)\Gamma(\alpha)>1,
\]
so that $\bigl(\uprho_{\alpha,i}\Gamma(\alpha)\bigr)^{\mhalf}<1$, and consequently
\[
\cH_\pi\Bigl(D_\pi=\bigl(\uprho_{\alpha,i}\Gamma(\alpha)\bigr)^{\mhalf}\Bigr)=\{0\}
\]
because all eigenvalues of $D_\pi$ are greater or equal to $1$ (as we have $\|D_\pi^{-1}\|=1$). This, however, contradicts the combination of \eqref{eq:cont} and \eqref{eq:UU1pi} for $\beta=\alpha$. Similarly, if $\gamma(\alpha)<\Gamma(\alpha)^{-1}$ then the equality $\oon{Sp}(\uprho_{\overline{\alpha}})=\oon{Sp}(\uprho_\alpha)^{-1}$ leads to a contradiction with \eqref{eq:UU1pi} for $\beta=\overline{\alpha}$ which ends the proof of claim \eqref{thm:Lambdat1}.

Ad \eqref{thm:Lambdat2}. Fix $\pi\in{O_5}$ and take an arbitrary $\beta\in\Irr{\hh{\bbGamma}}$, $i\in\{1,\dotsc,d_{\beta}\}$, and let $\mu_1,\dotsc,\mu_K$ be eigenvalues of $v_{t;\pi}$ on $\cH_\pi(D_\pi=1)$. Furthermore let $\xi_k$ be a non-zero element of $\cH_\pi(v_{t;\pi}=\mu_k)\cap\cH_\pi(D_\pi=1)$.

If $\uprho_{\beta,i}>\gamma(\beta)$ then $\bigl(\uprho_{\beta,i}\Gamma(\beta)\bigr)^{\mhalf}<1$ and hence, by the same argument as above, we have
\[
U_{i,d_{\beta};\pi}^\beta\xi_k=0.
\]
It follows that
\[
0\neq\xi_k=\sum_{i=1}^{\dim{\beta}}{U_{i,d_{\beta};\pi}^\beta}^*U_{i,d_{\beta};\pi}^\beta\xi_k
=\sum_{\uprho_{\beta,i}=\gamma(\beta)}{U_{i,d_{\beta};\pi}^\beta}^*U_{i,d_{\beta};\pi}^\beta\xi_k
\]
and hence there exists $i$ such that $\uprho_{\beta,i}=\gamma(\beta)=\Gamma(\beta)^{-1}$ (by \eqref{thm:Lambdat1}) and
\[
0\neq{U_{i,d_{\beta};\pi}^\beta}\xi_k\in\cH_\pi\bigl(v_{t;\pi}=\Gamma(\beta)^{-2\ii{t}}\mu_k\bigr)\cap\cH_\pi(D_\pi=1).
\]
Thus $\Gamma(\beta)^{-2\ii{t}}\mu_k$ is an eigenvalue of $v_{t;\pi}$, so $\Gamma(\beta)^{2\ii{t}}=\mu_l\mu_k^{-1}$ for some $l$. Consequently it belongs to the finite set
\[
\oon{Sp}\bigl(\chi(D_\pi=1)v_{t;\pi}\bigr)\cdot\oon{Sp}\bigl(\chi(D_\pi=1)v_{t;\pi}\bigr)^{-1}.\qedhere
\]
\end{proof}

\begin{corollary}\label{cor:rootU}
Let $\bbGamma$ be a second countable type $\mathrm{I}$ discrete quantum group, $\alpha\in\Irr{\hh{\bbGamma}}$ and $t\in\TtauInn(\hh{\bbGamma})$. Then the number $\Gamma(\alpha)^{2\ii{t}}$ is a root of unity.
\end{corollary}

\begin{proof}
Suppose $\Gamma(\alpha)^{2\ii{t}}$ is not a root of unity for some $\alpha$ and some $t\in\TtauInn(\hh{\bbGamma})$. Then setting $\alpha_n$ to be the class of the subrepresentation of $(U^\alpha)^{\tp{n}}$ such that $\Gamma(\alpha_n)=\Gamma(\alpha)^n$, we obtain an infinite sequence of different numbers $\bigl(\Gamma(\alpha_n)^{2\ii{t}}\bigr)_{n\in\NN}=\bigl(\Gamma(\alpha)^{2\ii{nt}}\bigr)_{n\in\NN}$ belonging to the finite set $\Lambda_t$ described in statement \eqref{thm:Lambdat2} of Theorem \ref{thm:Lambdat}.
\end{proof}

Corollary \ref{cor:rootU} provides a simple proof that Conjecture \ref{conj:main} holds for duals of second countable discrete quantum groups of type $\mathrm{I}$:

\begin{corollary}
Let $\bbGamma$ be a second countable type $\mathrm{I}$ discrete quantum group. If $\hh{\bbGamma}$ is not of Kac type then $\TtauInn(\hh{\bbGamma})$ is countable. In particular $\TtauInn(\hh{\bbGamma})=\RR$ implies that $\hh{\bbGamma}$ is of Kac type.
\end{corollary}

\begin{proof}
Assume that $\TtauInn(\hh{\bbGamma})$ is uncountable and $\hh{\bbGamma}$ is not of Kac type. Next take $\alpha\in\Irr{\hh{\bbGamma}}$ such that $\Gamma(\alpha)>1$ and consider the set $S_\alpha=\bigl\{\Gamma(\alpha)^{2\ii{t}}\,\bigr|\bigl.\,t\in\TtauInn(\hh{\bbGamma})\bigr\}$. Since $\TtauInn(\hh{\bbGamma})$ is an uncountable subgroup of $\RR$ there exists $k\in\ZZ$ such that the set
\[
\TtauInn(\hh{\bbGamma})\cap\left[\tfrac{k\pi}{\log{\Gamma(\alpha)}},\tfrac{(k+1)\pi}{\log{\Gamma(\alpha)}}\right[
\]
is also uncountable. Since the map $t\mapsto\Gamma(\alpha)^{2\ii{t}}$ is injective on this set, we find that $S_\alpha$ is uncountable as well. This, however, contradicts the fact that, by Corollary \ref{cor:rootU}, all elements of $S_\alpha$ are roots of unity.
\end{proof}

\section{Invariants of \texorpdfstring{$q$}{q}-deformations of a compact, simply connected, semisimple Lie groups}\label{sec:qDeformations}

In this section we will compute the invariants introduced in Section \ref{sect:defInv} for $q$-deformations of compact, simply connected, semisimple Lie groups. The main tool will be a precise description of the action of the scaling group of the $q$-deformation $G_q$ in terms of the direct integral decomposition of $\Linf(G_q)$ over the maximal torus $T$ (Theorem \ref{thm1}). The necessary prerequisite for this is the description of Plancherel objects for the dual $\hh{G_q}$ of $G_q$ which is carried out in Section \ref{sect:generalRS} based on some general results from Section \ref{sect:prel}.

\subsection{Preliminaries}\label{sect:prel}\hspace*{\fill}

We begin by recalling the definition of $G_q$ and its basic properties. Let $\mathfrak{g}$ be the complexified Lie algebra of a compact, simply connected, semisimple Lie group $G$, $T$ a maximal torus of $G$ with complexified Lie algebra $\mathfrak{h}$, $\la\cdot|\cdot\ra$ a positive multiple of the form on $\mathfrak{h}^*$ obtained from the Killing form on $\mathfrak{g}$ and $W$ the Weyl group of $\mathfrak{g}$ with longest element $w_\circ$. We assume the normalization $\la\alpha|\alpha\ra=2$ for short roots $\alpha$. By $\Phi\subset\mathfrak{h}^*$ we denote the root system, $\Phi^+\subset\Phi$ the set of positive roots, and by $\{\alpha_1,\dotsc,\alpha_r\}\subset\Phi^+$ the set of simple roots (so that $r$ is the rank of $\mathfrak{g}$). Next, $\bQ=\oon{span}_{\ZZ}\Phi$ is the root lattice, $\hh{T}\cong\bP\subset\mathfrak{h}^*$ the weight lattice, $\bP^+\subset\bP$ the positive cone, and $\varpi_1,\dotsc,\varpi_r\in\bP^+$ are the fundamental weights, which by definition satisfy $2\tfrac{\la\varpi_i|\alpha_j\ra}{\la\alpha_j|\alpha_j\ra}=\delta_{i,j}$ for all $i,j\in\{1,\dotsc,r\}$. Finally,
\[
\rho=\tfrac{1}{2}\sum_{\alpha\in\Phi^+}\alpha=\sum_{i=1}^{r}\varpi_i\in\bP^+
\]
is the Weyl vector.

For details of the construction of $G_q$ we refer to \cite{NeshveyevTuset} -- we will however follow the conventions of \cite{DeCommer,ReshetikhinYakimov} (which differ from \cite{NeshveyevTuset}) as we want to use results from these papers. Here we will only outline some of the aspects of the construction. The construction depends on the deformation parameter $q\in\left]0,1\right[$. Having fixed $q$ one first constructs the Hopf $*$-algebra $\mathcal{U}_q(\mathfrak{g})$, which is a deformation of the universal enveloping algebra of $\mathfrak{g}$. As a unital $*$-algebra, $\mathcal{U}_q(\mathfrak{g})$ is generated by elements $E_i,F_i,K_i,K_i^{-1}$ ($1\leq{i}\leq{r}$) satisfying the following relations:
\[
\begin{array}{c}
\begin{array}{r@{\;=\;}lr@{\;=\;}l}
K_iK_i^{-1}&K_i^{-1}K_i=\I,&K_iK_j&K_jK_i,\\[3pt]
K_iE_j&q^{\la\alpha_i|\alpha_j\ra}E_jK_i,&K_iF_j&q^{-\la\alpha_i|\alpha_j\ra}F_jK_i,
\end{array}\\[12 pt]
E_iF_j-F_jE_i=\delta_{i,j}\tfrac{K_i^{-1}-K_i}{q_i^{-1}-q_i},\\[3pt]
\left.\begin{array}{l@{\;=\;}r}
\displaystyle\sum_{k=0}^{1-a_{i,j}}(-1)^k\begin{bmatrix}1-a_{i,j}\\k\end{bmatrix}_{q_i}E_i^kE_jE_i^{1-a_{i,j}-k}&0\\
\displaystyle\sum_{k=0}^{1-a_{i,j}}(-1)^k\begin{bmatrix}1-a_{i,j}\\k\end{bmatrix}_{q_i}F_i^kF_jF_i^{1-a_{i,j}-k}&0
\end{array}\right\}\qquad{i}\neq{j}\hspace{-5Em}
\end{array}
\]
with $q_i=q^{\la\alpha_i|\alpha_i\ra/2}$, $a_{i,j}=2\tfrac{\la\alpha_i|\alpha_j\ra}{\la\alpha_i|\alpha_i\ra}$ (these are the entries of the Cartan matrix), and
$\begin{bmatrix}m\\k\end{bmatrix}_q$ denoting the $q$-binomial coefficients $\tfrac{[m]_q!}{[m-k]_q![k]_q!}$, where $[m]_q!=[m]_q[m-1]_q\dotsm[1]_q$, $[m]_q=\tfrac{q^{-m}-q^m}{q^{-1}-q}$.

Next we define comultiplication and $*$-structure on $\mathcal{U}_q(\mathfrak{g})$ via
\[
\hh{\Delta}(K_i)=K_i\tens{K_i},\quad
\hh{\Delta}(E_i)=E_i\tens{K_i}+\I\tens{E_i},\quad
\hh{\Delta}(F_i)=F_i\tens\I+K_i^{-1}\tens{F_i}
\]
and
\[
K_i^*=K_i,\quad
E_i^*=F_iK_i,\quad
F_i^*=K_i^{-1}E_i.
\]
It is known that $\bigl(\mathcal{U}_q(\mathfrak{g}),\hh{\Delta}\bigr)$ is a Hopf $*$-algebra with antipode and counit
\begin{align*}
\hh{\eps}(K_i)&=1,&\hh{\eps}(E_i)&=0,&\hh{\eps}(F_i)&=0,\\
\hh{S}(K_i)&=K_i^{-1},&\hh{S}(E_i)&=-E_iK_i^{-1},&\hh{S}(F_i)&=-K_iF_i
\end{align*}
for $1\leq{i}\leq{r}$. For $\alpha=\sum\limits_{i=1}^{r}n_i\alpha_i\in\bQ$ we will use the shorthand notation $K_\alpha=K_1^{n_1}\dotsm{K_r^{n_r}}$.

One defines $\Pol(G_q)\subset\mathcal{U}_q(\mathfrak{g})^*$ as the linear span of matrix coefficients of all \emph{admissible}, finite dimensional $*$-representations of $\mathcal{U}_q(\mathfrak{g})$ (roughly speaking, these are the representations in which all $K_i$ are diagonalizable, see e.g.~\cite[Definition 2.4.3]{NeshveyevTuset}). By construction, $\mathcal{U}_q(\mathfrak{g})$ embeds into $\Pol(G_q)^*$.

As suggested by the notation, $\Pol(G_q)$ indeed turns out to be a C.Q.G.~Hopf algebra, i.e.~a Hopf $*$-algebra associated to a compact quantum group. The set $\Irr{G_q}$ of equivalence classes of irreducible representations of $G_q$ can be identified with $\bP^+$ and for any $\varpi\in\bP^+=\Irr{G_q}$ we will denote by $\pi_\varpi$ the associated $*$-representation of the $*$-algebra $\mathcal{U}_q(\mathfrak{g})$ on the Hilbert space $\cH_\varpi$ whose dimension will be denoted by $\dim{\varpi}$. These identifications are described by equality $\pi_\varpi(\omega)=(\omega\tens\id)U^{\varpi}$ valid for $\varpi\in\bP^+$ and $\omega\in\mathcal{U}_q(\mathfrak{g})$. One defines the \emph{weights of $\varpi$} as these $\mu\in\bP$ for which space
\[
\cH_\varpi(\mu)=\bigl\{\xi\in\cH_\varpi\,\bigr|\bigl.\,\pi_\varpi(K_i)\xi=q^{\la\mu|\alpha_i\ra}\xi,\:1\leq{i}\leq{r}\bigr\}
\]
is non-zero. For example, any $\varpi$ is a weight of $\varpi$ in this sense. Moreover we have the equality $\cH_\varpi=\bigoplus\limits_\mu\cH_\varpi(\mu)$ with $\mu$ running over the weights of $\varpi$.

Whenever $\varpi\in\bP^+$ and $\xi\in\cH_\varpi(\mu),\eta\in\cH_\varpi(\mu')$ are vectors of weights $\mu,\mu'\in\bP$, we define \emph{left} and \emph{right weights} of the matrix element $U^\varpi(\xi,\eta)=(\id\tens\omega_{\xi,\eta})U^\varpi\in\Pol(G_q)$ by
\[
\lwt\bigl(U^\varpi(\xi,\eta)\bigr)=\wt(\xi)=\mu,\quad
\rwt\bigl(U^\varpi(\xi,\eta)\bigr)=\wt(\eta)=\mu'.
\]
In formulas like these, we will always implicitly assume that $\xi,\eta$ have well defined weights. Next, we extend definition of $\lwt,\rwt$ in the obvious way, so that we obtain $\bP$-gradings of $\Pol(G_q)$. Let us also fix an orthonormal basis $\{\xi^\varpi_1,\dotsc,\xi^\varpi_{\dim{\varpi}}\}$ of $\cH_\varpi$ in which each vector $\xi^{\varpi}_i$ has a well defined weight.

One can choose unit vectors $\xi_\varpi\in\cH_\varpi(\varpi)$ and $\eta_{w_\circ\varpi}\in\cH_\varpi(w_\circ\varpi)$ of highest (resp.~lowest) weight so that the corresponding elements
\[
b_\varpi=U^\varpi(\xi_\varpi,\eta_{w_\circ\varpi})\in\Pol(G_q),\qquad\varpi\in\bP^+
\]
have convenient properties -- for details see \cite[Section 5]{DeCommer}. These elements (especially $b_\rho$) will play an important role in what follows. One can obtain detailed information about these elements -- see equation \eqref{eq33}. In particular, one can calculate that in the case of $G_q=\oon{SU}_q(2)$ we have $b_N=(-\gamma)^N$ for all $N\in\Irr(\oon{SU}_q(2))=\ZZ_+$ (where $\gamma$ is one of the standard generators of $\Pol(\oon{SU}_q(2))$, cf.~\cite[Section 1]{su2}).

Irreducible representations of the \cst-algebra $\C(G_q)$ were classified by Korogodski and Soibelman (\cite{KorogodskiSoibelman}) and the Haar measure on $G_q$ was described by Reshetikhin and Yakimov in \cite{ReshetikhinYakimov}. Let us now briefly describe some of these results in the conventions of \cite{DeCommer}.

Let $\theta_q\colon\C(\oon{SU}_q(2))\to\B(\ell^2(\ZZ_+))$ be the representation defined by
\[
\left.\begin{array}{c@{\;=\;}l}
\theta_q(\alpha)e_n&\sqrt{1-q^{2n}}e_{n-1}\\
\theta_q(\gamma)e_n&q^ne_n
\end{array}\right\}\qquad{n}\in\ZZ_+
\]
where $\{e_n\}_{n\in\ZZ_+}$ is the standard orthonormal basis of $\ell^2(\ZZ_+)$ and $e_{-1}=0$. For $1\leq{i}\leq{r}$ let
\[
\rho_i\colon\C(\oon{SU}_{q_i}(2))\longrightarrow\C(G_q)
\]
be the map corresponding to the embedding $\mathcal{U}_{q_i}(\mathfrak{sl}(2,\CC))\to\mathcal{U}_q(\mathfrak{g})$. By taking compositions of these maps we obtain representations of $\C(G_q)$ defined as
\[
\pi_{s_i}\colon\!\!\xymatrix{\C(G_q)\ar[r]^(.43){\rho_i}&\C\bigl(\oon{SU}_{q_i}(2)\bigr)\ar[r]^(.51){\theta_{q_i}}&\B\bigl(\ell^2(\ZZ_+)\bigr),}
\]
where $s_1,\dotsc,s_r$ denote the generators of the Weyl group $W$. For any $w\in{W}\setminus\{e\}$ choose a reduced presentation $w=s_1\dotsm{s_{\ell(w)}}$ and define
\[
\pi_w=\pi_{s_1}\star\dotsm\star\pi_{s_{\ell(w)}}=
(\pi_{s_1}\tens\dotsm\tens\pi_{s_{\ell(w)}})\comp\Delta_{G_q}^{(\ell(w)-1)}\colon\C(G_q)\longrightarrow\B\bigl(\ell^2(\ZZ_+)^{\tens\ell(w)}\bigr).
\]
(where $\Delta_{G_q}^{(n)}$ is defined inductively as $\Delta_{G_q}^{(1)}=\Delta_{G_q}$, $\Delta_{G_q}^{(k+1)}=(\id\tens\dotsm\tens\id\tens\Delta_{G_q})\comp\Delta_{G_q}^{(k)}$).
The representation $\pi_w$ is irreducible and up to equivalence it does not depend on expression of $w$ as a word in the generators.

Characters of $\C(G_q)$ are in bijection with the torus $T$. For $\lambda\in{T}$, the corresponding character is given by
\[
\pi_\lambda\colon\C(G_q)\ni{U^\varpi(\xi,\eta)}\longmapsto
\is{\xi}{\eta}\la\wt(\xi),\lambda\ra=\is{\xi}{\eta}\la\wt(\eta),\lambda\ra\in\CC.
\]
Korogodski and Soibelman showed that irreducible representations of $\C(G_q)$ are (up to equivalence) precisely\footnote{We use representations $\pi_{\lambda}\star\pi_w$ (rather than $\pi_{w}\star\pi_\lambda$) to be in agreement with \cite{DeCommer}. This choice gives an equivalent parametrization of the spectrum, see \cite[Proposition 6.2.6]{KorogodskiSoibelman}.}
\[
\bigl\{\pi_{\lambda,w}\,\bigr|\bigl.\,\lambda\in{T},\:w\in{W}\bigr\}
\]
where $\pi_{\lambda,e}=\pi_\lambda$ ($\lambda\in{T}$) and
\[
\pi_{\lambda,w}=\pi_\lambda\star\pi_w=(\pi_\lambda\tens\pi_w)\Delta_{G_q},\qquad\lambda\in{T},\:w\in{W}\setminus\{e\}.
\]
Furthermore, $\pi_{\lambda,w}$ and $\pi_{\lambda',w'}$ are non-equivalent for $(\lambda,w)\neq(\lambda',w')$. Consequently the spectrum of $\C(G_q)$ can be identified with $T\times{W}$ as a set. For $x=U^\varpi(\xi,\eta)$ where $\xi,\eta\in\cH_\varpi$ have well defined weights, and $(\lambda,w)\in{T}\times{W}$ we have the following useful formula:
\begin{equation}\label{eq17}
\begin{aligned}
\pi_{\lambda,w}(x)&=\sum_{i=1}^{\dim{\varpi}}
\pi_\lambda\bigl(U^{\varpi}(\xi,\xi^\varpi_i)\bigr)\pi_w\bigl(U^\varpi(\xi^\varpi_i,\eta)\bigr)\\
&=\sum_{i=1}^{\dim{\varpi}}\is{\xi}{\xi^\varpi_i}\la\wt(\xi),\lambda\ra\,\pi_w\bigl(U^{\varpi}(\xi^\varpi_i,\eta)\bigr)=\la\lwt(x),\lambda\ra\pi_{w}(x).
\end{aligned}
\end{equation}

It is well-known to the experts that $\C(G_q)$ is a type $\mathrm{I}$ \cst-algebra with the obvious measurable structure on its spectrum. However we could not precisely locate a complete proof in the literature, so we record it here for the reader's convenience. We also show that the field of representations $(\pi_{\lambda,w})_{(\lambda,w)\in{T}\times{W}}$ is measurable -- a result needed for concrete calculations. Note that it follows from the fact that $\C(G_q)$ is type $\mathrm{I}$ that $G_q$ is coamenable (\cite[Theorem 2.8]{typeI}, see also \cite[Theorem 2.7.14]{NeshveyevTuset}).

\begin{lemma}
\noindent
\begin{enumerate}
\item The \cst-algebra $\C(G_q)$ is separable, unital and type $\mathrm{I}$.
\item\label{lemma3.2} The Mackey-Borel structure on its spectrum is equal to the natural Borel $\sigma$-algebra on $T\times{W}$.
\item\label{lemma3.3} For $w\in{W}\setminus\{e\}$ fix a unitary operator $U_w\colon\ell^2(\ZZ_+)^{\tens\ell(w)}\rightarrow\cH_{\boldsymbol{n}}$, where $\cH_{\boldsymbol{n}}$ is the standard Hilbert space of dimension $\boldsymbol{n}=\dim{\pi_{w}}=\aleph_0$, set also $U_e=1$. Then the map $[\pi_{\lambda,w}]\mapsto{U_w\pi_{\lambda,w}U^*_w}$ between the spectrum of $\C(G_q)$ and the space of irreducible representations, is measurable.
\end{enumerate}
\end{lemma}

\begin{proof}
In this proof we will write $\oon{Sp}{\C(G_q)}$ for the spectrum of $\C(G_q)$ (i.e.~the set of equivalence classes of irreducible representations) equipped with its Mackey-Borel structure (\cite[Section 3.8]{DixmierC}) and $\oon{\mathfrak{Irr}}\C(G_q)$ for the set of irreducible representations acting on standard Hilbert spaces $\cH_{\boldsymbol{n}}$ of dimension $\boldsymbol{n}\leq\aleph_0$, equipped with Borel $\sigma$-algebra (see \cite[Section 3.8]{DixmierC}). We will also use measurable subspaces $\oon{\mathfrak{Irr}}_{\boldsymbol{n}\!\!}\C(G_q)$ of $\oon{\mathfrak{Irr}}\C(G_q)$ consisting of representations of dimension $\boldsymbol{n}$. Both $\oon{\mathfrak{Irr}}_{\boldsymbol{n}\!\!}\C(G_q)$ and $\oon{\mathfrak{Irr}}\C(G_q)$ are standard (\cite[Proposition 3.7.4]{DixmierC}). Throughout the proof, we will distinguish a representation $\pi$ from its class $[\pi]$. Recall that the Mackey-Borel structure is the quotient $\sigma$-algebra for the natural surjection $\oon{\mathfrak{Irr}}\C(G_q)\rightarrow\oon{Sp}{\C(G_q)}$.

By results of Korogodski and Soibelman, any irreducible representation of $\C(G_q)$ is equivalent to $\pi_{\lambda,w}$ for some $(\lambda,w)\in{T}\times{W}$. Furthermore, \cite[Proposition 4.3]{ReshetikhinYakimov} shows that the image of each representation $\pi_{\lambda,w}$ contains a non-zero compact operator, hence \cite[Definition 4.6.3, Theorem 4.6.4]{Sakai} implies that $\C(G_q)$ is a type $\mathrm{I}$ \cst-algebra. Consequently, $\oon{Sp}{\C(G_q)}$ is also standard (\cite[Proposition 4.6.1]{DixmierC}).

Fix $w\in{W}$. We first claim that the corresponding torus $\bigl\{[\pi_{\lambda,w}]\,\bigr|\bigl.\,\lambda\in{T}\bigr\}$ is measurable in $\oon{Sp}{\C(G_q)}$ or, equivalently, that $\bigl\{\rho\,\bigr|\bigl.\,\exists\:\lambda\in{T}\;[\rho]=[\pi_{\lambda,w}]\bigr\}$ is Borel in $\oon{\mathfrak{Irr}}\C(G_q)$. To proceed, we will use the family of ideals $I_w\subset\C(G_q)$ ($w\in{W}$) introduced in \cite[Equation (5.1.1)]{KorogodskiSoibelman}. By \cite[Theorem 5.3.3]{KorogodskiSoibelman}, for any $\lambda\in{T}$ and $w'\in{W}$ we have that $I_{w'}\subset\ker\pi_{\lambda,w}$ if and only if $w'=w$. Let $\boldsymbol{n}=\dim{\pi_{w}}$ and choose countable dense subsets $\{\xi_i\}_{i\in\NN}$ in $\cH_{\boldsymbol{n}}$ and $\{x_k\}_{k\in\NN}$ in $I_w$. We have that
\begin{align*}
\bigl\{\rho\in\oon{\mathfrak{Irr}}\C(G_q)\,\bigr|\bigl.\,\exists\:\lambda\in{T}\;[\rho]=[\pi_{\lambda,w}]\bigr\}
&=\bigl\{\rho\in\oon{\mathfrak{Irr}}\C(G_q)\,\bigr|\bigl.\,\forall\:{x}\in{I_w}\;\rho(x)=0\bigr\}\\
&=\bigcap_{i,j,k,m=1}^{\infty}\Bigl\{\rho\in\oon{\mathfrak{Irr}}_{\boldsymbol{n}\!\!}\C(G_q)\,\Bigr|\Bigl.\,\bigl|\is{\xi_i}{\rho(x_k)\xi_j}\bigr|\leq\tfrac{1}{m}\Bigr\}
\end{align*}
which by definition is measurable in $\oon{\mathfrak{Irr}}\C(G_q)$.

Next we show that for each fixed $w\in{W}$, the set $\bigl\{U_w\pi_{\lambda,w}U_w^*\,\bigr|\bigl.\,\lambda\in{T}\bigr\}$ is measurable in $\oon{\mathfrak{Irr}}\C(G_q)$ or, equivalently, in $\oon{\mathfrak{Irr}}_{\boldsymbol{n}\!\!}\C(G_q)$ where $\boldsymbol{n}=\dim{\pi_{w}}$. Consider the injective function $T\ni\lambda\mapsto{U_w\pi_{\lambda,w}U_w^*}\in\oon{\mathfrak{Irr}}_{\boldsymbol{n}\!\!}\C(G_q)$. Since both $T$ and $\oon{\mathfrak{Irr}}_{\boldsymbol{n}\!\!}\C(G_q)$ are standard, it is enough to argue that this map is measurable (\cite[Appendix B 21]{DixmierC}). To that end take $\zeta,\zeta'\in\cH_{\boldsymbol{n}}$ and $x\in\C(G_q)$ -- we need to show that $f\colon{T}\ni\lambda\mapsto\is{\zeta}{U_w\pi_{\lambda,w}(x)U_w^*\zeta'}\in\CC$ is measurable. As any pointwise limit of a sequence of measurable functions is measurable, it is enough to consider the case of $x=U^{\varpi}(\xi,\eta)$ for $\varpi\in\bP^+$ and $\xi,\eta\in\cH_\varpi$ with well defined weights. For $\lambda\in{T}$ we have $\pi_{\lambda,w}(x)=\pi_w(x)\la\lwt(x),\lambda\ra$ (equation \eqref{eq17}), which implies that $f$ is measurable. This shows that $\bigl\{U_w\pi_{\lambda,w}U_w^*\,\bigr|\bigl.\,\lambda\in{T}\bigr\}$ is measurable in $\oon{\mathfrak{Irr}}\C(G_q)$ and consequently is standard (\cite[Appendix B 20]{DixmierC}). Then the canonical map
\begin{equation}\label{canonical_map}
\bigl\{U_w\pi_{\lambda,w}U_w^*\,\bigr|\bigl.\,\lambda\in{T}\bigr\}\ni{U_w}\pi_{\lambda,w}U_w^*\longmapsto[\pi_{\lambda,w}]\in\bigl\{[\pi_{\lambda,w}]\,\bigr|\bigl.\,\lambda\in{T}\bigr\}
\end{equation}
is a measurable bijection between standard Borel spaces, and hence it is an isomorphism and its inverse $[\pi_{\lambda,w}]\mapsto{U_w^*}\pi_{\lambda,w}U_w$ is measurable (\cite[Appendix B 22]{DixmierC}). Upon taking the union over all $w$, we obtain \eqref{lemma3.3}.

Finally let us fix $w\in{W}$. The the composition
\[
\xymatrix{T\ar[r]&\bigl\{U_w\pi_{\lambda,w}U_w^*\,\bigr|\bigl.\,\lambda\in{T}\bigr\}\ar[r]&\bigl\{[\pi_{\lambda,w}]\,\bigr|\bigl.\,\lambda\in{T}\bigr\}}
\]
of the map taking $\lambda$ to $U_w\pi_{\lambda,w}U_w^*$ and then \eqref{canonical_map}
is a measurable isomorphism which proves \eqref{lemma3.2}.
\end{proof}

Reshetikhin and Yakimov (\cite{ReshetikhinYakimov}, see also \cite[Theorem 9.11]{DeCommer}) showed that the Haar measure on $G_q$ is given by
\begin{equation}\label{eq5}
\bh_{G_q}(x)=\bigl(\prod_{\alpha\in\Phi^+}\bigl(1-q^{2\la\rho|\alpha\ra}\bigr)\bigr)
\Int_T\Tr\bigl(\pi_{\lambda,w_\circ}(x|b_\rho|^2)\bigr)\dd{\lambda},\qquad{x}\in\C(G_q),
\end{equation}
where $\mathrm{d}\lambda$ is the normalized Lebesgue measure on $T$. From now on we equip the spectrum of $\C(G_q)$ with the normalized Lebesgue measure on $\bigl\{\pi_{\lambda,w_\circ}\,\bigr|\bigl.\,\lambda\in{T}\bigr\}\cong{T}$, so that $\oon{Sp}{\C(G_q)}=T$ as measure spaces. Equation \eqref{eq17} implies
\begin{equation}\label{eq32}
\pi_{\lambda,w_\circ}\bigl(|b_\rho|\bigr)=\pi_{w_\circ}\bigl(|b_\rho|\bigr),\qquad{\lambda}\in{T}.
\end{equation}
Consequently, since $\bh_{G_q}(\I)=1$ it follows from \eqref{eq5} that $\pi_{\lambda,w_\circ}(b_\rho)\neq{0}$ for all $\lambda\in{T}$.

Our next goal is to look at $\hh{G_q}$ as a type $\mathrm{I}$ discrete quantum group and describe its Plancherel objects (\cite[Section 3]{KrajczokTypeI}). The main ingredient is formula \eqref{eq5}, but there is a number of properties to be checked. It will be convenient to consider a more general situation: let $\bbGamma$ be a second countable, type $\mathrm{I}$ discrete quantum group. Then $\hh{\bbGamma}$ is coamenable (\cite[Theorem 2.8]{typeI}) and $\Irr{\bbGamma}=\oon{Sp}{\C(\hh{\bbGamma})}$ is a standard Borel space. Equip it with the standard measurable field of Hilbert spaces $(\cH_{\boldsymbol{x}})_{\boldsymbol{x}\in\Irr{\bbGamma}}$, i.e.~$\cH_{\boldsymbol{x}}=\CC^{\dim{\boldsymbol{x}}}$. Then $\bigl(\HS(\cH_{\boldsymbol{x}})\bigr)_{\boldsymbol{x}\in\Irr{\bbGamma}}$ is also a measurable field of Hilbert spaces (\cite[Part II, Chapter 1]{DixmiervNA}). We fix a measurable field of representations $(\pi_{\boldsymbol{x}})_{\boldsymbol{x}\in\Irr{\bbGamma}}$ such that the class of $\pi_{\boldsymbol{x}}$ is equal to $\boldsymbol{x}$ (\cite[Section 8.6.1]{DixmierC}) -- from now on we will abuse notation and identify $\pi_{\boldsymbol{x}}$ with $\boldsymbol{x}$, so that e.g.~a field of vectors will be denoted by $(\xi_\pi)_{\pi\in\Irr{\bbGamma}}$. Let us denote the Haar measure on $\hh{\bbGamma}$ by $\bh_{\hh{\bbGamma}}$. Assume that $\mu$ is a probability measure on $\Irr{\bbGamma}$, $(a_\pi)_{\pi\in\Irr{\bbGamma}}$ is a measurable field of positive, bounded operators on $\cH_\pi$ such that $a_\pi\neq{0}$ for almost all $\pi$, and
\begin{equation}\label{eq6}
\bh_{\hh{\bbGamma}}(x)=\Int_{\Irr{\bbGamma}}\Tr\bigl(\pi(x)a_\pi^2\bigr)\dd{\mu}(\pi),\qquad{x}\in\Pol(\hh{\bbGamma}).
\end{equation}
In this situation we have the following:

\begin{proposition}\label{prop2}
\noindent\begin{enumerate}
\item
For almost all $\pi$, the operator $a_\pi$ is Hilbert-Schmidt with trivial kernel,
\item
equation \eqref{eq6} holds for all $x\in\C(\hh{\bbGamma})$,
\item
there is a unitary operator $\cQ_\LL\colon\Ltwo(\hh{\bbGamma})\to{\Int_{\Irr{\bbGamma}}}^{\!\!\!\oplus}\HS(\cH_\pi)\dd{\mu}(\pi)$ satisfying
\begin{equation}\label{eq7}
\cQ_\LL\Lambda_{\bh_{\hh{\bbGamma}}}(x)={\Int_{\Irr{\bbGamma}}}^{\!\!\oplus}\pi(x)a_\pi\dd{\mu}(\pi),\qquad{x}\in\C(\hh{\bbGamma}),
\end{equation}
\item\label{prop2,4} $\bigl(\cQ_\LL,(a_\pi^{-1})_{\pi\in\Irr{\bbGamma}},\mu\bigr)$ are the left Plancherel objects for $\bbGamma$.
\end{enumerate}
\end{proposition}

Statement \eqref{prop2,4} of Proposition \ref{prop2} means that $\bigl(\cQ_\LL,(a_\pi^{-1})_{\pi\in\Irr{\bbGamma}},\mu\bigr)$ satisfy all the properties of \cite[Theorem 3.3]{KrajczokTypeI}. As an immediate corollary we obtain

\begin{corollary}\label{cor1}
For all $\lambda\in{T}$, the operator $\pi_{w_\circ,\lambda}(|b_\rho|)$ is Hilbert-Schmidt with trivial kernel and
\[
\Biggl(\cQ_\LL,\biggl(
\Bigl(\prod_{\alpha\in\Phi^+}\bigl(1-q^{2\la\rho|\alpha\ra}\bigr)^{-\frac{1}{2}}\Bigr)\pi_{\lambda,w_\circ}\bigl(|b_\rho|\bigr)^{-1}\biggr)_{\lambda\in{T}},\dd\lambda\Biggr)
\]
are the left Plancherel objects for $\hh{G_q}$.
\end{corollary}

As we already noted in equation \eqref{eq32}, we have $\pi_{\lambda,w_\circ}(|b_\rho|)=\pi_{w_\circ}(|b_\rho|)$ for all $\lambda\in{T}$. Nonetheless, we will keep writing $\pi_{\lambda,w_\circ}(|b_\rho|)$ whenever we want to highlight that we see this operator as acting on $\cH_\lambda$.

\begin{proof}[Proof of Proposition \ref{prop2}]
Setting $x=\I$ in \eqref{eq6} gives $1=\Int_{\Irr\bbGamma}\Tr(a_\pi^2)\dd{\mu}(\pi)$, hence $a_\pi$ is Hilbert-Schmidt for almost all $\pi$. Then it follows by an approximation that \eqref{eq6} holds for all $x\in\C(\hh{\bbGamma})$.

Now we want to show that $\cQ_\LL$ can be defined by \eqref{eq7}. Take $x\in\C(\hh{\bbGamma})$ and let us argue that $(\pi(x)a_\pi)_{\pi\in\Irr{\bbGamma}}$ is a measurable field of Hilbert-Schmidt operators. Clearly for almost all $\pi$, the operator $\pi(x)a_\pi$ is Hilbert-Schmidt and by definition, the field of operators $(\pi(x))_{\pi\in\Irr{\bbGamma}}$ is measurable. Take measurable vector fields $(\xi_\pi)_{\pi\in\Irr{\bbGamma}}$, $(\eta_{\pi})_{\pi\in\Irr{\bbGamma}}$. The function
\[
\Irr{\bbGamma}\ni\pi\longmapsto\is{\xi_\pi\tens\overline{\eta_\pi}}{\pi(x)a_\pi}
=\Tr\bigl(\ket{\eta_\pi}\bra{\xi_\pi}\,\pi(x)a_\pi\bigr)
=\is{\xi_\pi}{\pi(x)a_\pi\eta_\pi}\in\CC
\]
is measurable and hence by \cite[Proposition 7.20]{Folland}, the field $(\pi(x)a_\pi)_{\pi\in\Irr{\bbGamma}}$ of Hilbert-Schmidt operators is measurable. Consequently, equation \eqref{eq6} allows us to define the operator $\cQ_\LL$ on $\Lambda_{\bh_{\hh{\bbGamma}}}(\C(\hh{\bbGamma}))$ via \eqref{eq7}. Since it is an isometry, it extends to the whole of $\Ltwo(\hh{\bbGamma})$.

In the next step we show that $a_\pi$ has trivial kernel for almost all $\pi$. Take $x\in\Pol(\hh{\bbGamma})$ and $y\in\C(\hh{\bbGamma})$. Then
\begin{equation}\label{eq8}
\begin{aligned}
\Int_{\Irr{\bbGamma}}\Tr\bigl(\pi(y)a_\pi^2\pi(x)\bigr)\dd{\mu}(\pi)
&=\Int_{\Irr{\bbGamma}}\Tr\bigl(\pi(xy)a_\pi^2\bigr)\dd{\mu}(\pi)=\bh_{\hh{\bbGamma}}(xy)\\
&=\bh_{\hh{\bbGamma}}\bigl(y\sigma^{\bh_{\hh{\bbGamma}}}_{-\ii}(x)\bigr)
=\Int_{\Irr{\bbGamma}}\Tr\Bigl(\pi(y)\pi\bigl(\sigma^{\bh_{\hh{\bbGamma}}}_{-\ii}(x)\bigr)a_\pi^2\Bigr)\dd{\mu}(\pi).
\end{aligned}
\end{equation}
We can interpret equation \eqref{eq8} as equality of two normal bounded functionals:
\begin{equation}\label{eq10}
{\Int_{\Irr{\bbGamma}}}^{\!\!\oplus}\Tr\bigl(\,\cdot\;a_\pi^2\pi(x)\bigr)\dd{\mu}(\pi)
\quad\text{and}\quad
{\Int_{\Irr{\bbGamma}}}^{\!\!\oplus}\Tr\Bigl(\,\cdot\;\pi\bigl(\sigma^{\bh_{\hh{\bbGamma}}}_{-\ii}(x)\bigr)a_\pi^2\Bigr)\dd{\mu}(\pi)
\end{equation}
on operators of the form ${\Int_{\Irr{\bbGamma}}}^{\!\!\!\oplus}\pi(y)\dd{\mu}(\pi)\in{\Int_{\Irr{\bbGamma}}}^{\!\!\!\oplus}\B(\cH_\pi)\dd{\mu}(\pi)$ with $y\in\C(\hh{\bbGamma})$. The density result \cite[Proposition 8.6.4]{DixmierC}
\begin{equation}\label{eq44}
\Biggl\{\,{\Int_{\Irr{\bbGamma}}}^{\!\!\oplus}\pi(z)\dd{\mu}(\pi)\,\biggr|\biggl.\,z\in\C(\hh{\bbGamma})\Biggr\}''
=\oon{Dec}\Biggl(\,{\Int_{\Irr{\bbGamma}}}^{\!\!\oplus}\cH_\pi\dd{\mu}(\pi)\Biggr)
={\Int_{\Irr{\bbGamma}}}^{\!\!\oplus}\B(\cH_\pi)\dd{\mu}(\pi)
\end{equation}
implies that functionals \eqref{eq10} are equal and consequently
\[
a_\pi^2\pi(x)=\pi\bigl(\sigma^{\bh_{\hh{\bbGamma}}}_{-\ii}(x)\bigr)a_\pi^2
\]
holds for almost all $\pi$. Fix such a $\pi$ and assume additionally that $\sigma^{\bh_{\hh{\bbGamma}}}_t(x)=\gamma^{2\ii{t}}x$, ($t\in\RR$) for some $\gamma>0$. Then $a_\pi^2\pi(x)=\pi(x)(\gamma{a_\pi})^2$ and by induction
\begin{equation}\label{eq12}
a_\pi^{2k}\pi(x)=\pi(x)(\gamma{a_\pi})^{2k},\qquad{k}\in\NN.
\end{equation}
and consequently $P(a_\pi^2)\pi(x)=\pi(x)P(\gamma^2{a_\pi}^2)$ for any polynomial $P$. Now if $\oon{Sp}(a_\pi^2)=\{\lambda_1,\lambda_2,\dotsc\}$ and $\lambda_1>\lambda_2>\dotsm$ then taking $(P_n)_{n\in\NN}$ to be a sequence of polynomials such that $P_n(0)=1$ and $\bigl|P_n(\lambda_k)\bigr|\leq\tfrac{1}{n}$ for $k\leq{n}$ and $|P_n|\leq{M}$ (for some constant $M$) on $\bigl[0,\|a_\pi^2\|\bigr]$ yields
\[
P_n(a_\pi^2)\xrightarrow[n\to\infty]{}\chi(a_\pi^2=0)\quad\text{and}\quad
P_n(\gamma^2a_\pi^2)\xrightarrow[n\to\infty]{}\chi(\gamma^2a_\pi^2=0)
\]
strongly. It follows that
\begin{equation}\label{eq9}
\begin{aligned}
\chi(a_\pi=0)\pi(x)&=\chi(a_\pi^2=0)\pi(x)=\lim_{n\to\infty}P_n(a_\pi^2)\pi(x)
\\
&=\pi(x)\lim_{n\to\infty}P_n(\gamma^2a_\pi^2)=\pi(x)\chi(\gamma^2a_\pi^2=0)=\pi(x)\chi(a_\pi=0)
\end{aligned}
\end{equation}
(the limits are in the strong topology).

The assumption $\sigma^{\bh_{\hh{\bbGamma}}}_t(x)=\gamma^{2\ii{t}}x$ for some $\gamma>0$ is satisfied by all matrix elements of irreducible representations of $\hh{\bbGamma}$ (in appropriate bases). Consequently, by linearity and continuity, equation \eqref{eq9} holds for all $x\in\C(\hh{\bbGamma})$ and almost all $\pi\in\Irr{\bbGamma}$ (recall that $\C(\hh{\bbGamma})$ is separable). It follows that $\cH_\pi(a_\pi=0)$ is an invariant subspace for $\pi$ -- since $\pi$ is irreducible and $a_\pi\neq{0}$, we have $\cH(a_\pi=0)=\{0\}$. Consequently $a_\pi$ has trivial kernel for almost all $\pi\in\Irr(\bbGamma)$.

Next we claim that $\cQ_\LL$ is surjective. If this is not the case, there is a norm $1$ vector $T={\Int_{\Irr{\bbGamma}}}^{\!\!\!\oplus}T_\pi\dd{\mu}(\pi)$ in ${\Int_{\Irr{\bbGamma}}}^{\!\!\!\oplus}\HS(\cH_\pi)\dd{\mu}(\pi)$ which is orthogonal to the image of $\cQ_\LL$. In particular
\[
0=\Int_{\Irr{\bbGamma}}\Tr\bigl(T_\pi^*\pi(x)a_\pi\bigr)\dd{\mu}(\pi),\qquad{x}\in\C(\hh{\bbGamma}).
\]
Fix $N\in\NN$ and let $\Omega_N$ be the measurable subset of those $\pi\in\Irr{\bbGamma}$ for which $\|T_\pi\|_{\HS(\cH_\pi)}^2\leq{N}$. By \eqref{eq44} and the Kaplansky density theorem, there exists a bounded sequence $(x_{n,N})_{n\in\NN}$ in $\C(\hh{\bbGamma})$ such that
\[
{\Int_{\Irr{\bbGamma}}}^{\!\!\oplus}\pi(x_{n,N})\dd{\mu}(\pi)\xrightarrow[n\to\infty]{}{\Int_{\Irr{\bbGamma}}}^{\!\!\oplus}\chi_{\Omega_N}(\pi)T_\pi\dd{\mu}(\pi)
\]
in the strong topology. Indeed, since $(T_\pi)_{\pi\in\Irr{\bbGamma}}$ is measurable as a vector field with respect to the field of Hilbert spaces $\bigl(\HS(\cH_\pi)\bigr)_{\pi\in\Irr{\bbGamma}}$, it is also measurable as a field of operators. After possibly passing to a subsequence, we obtain (\cite[A 79]{DixmierC})
\[
\pi(x_{n,N})\xrightarrow[n\to\infty]{}\chi_{\Omega_N}(\pi)T_\pi
\]
strongly for almost all $\pi\in\Irr{\bbGamma}$. Bounded convergence theorem gives
\begin{equation}\label{eq11}
0=\lim_{n\to\infty}\Int_{\Irr{\bbGamma}}\Tr\bigl(T_\pi^*\pi(x_{n,N})a_\pi\bigr)\dd{\mu}(\pi)
=\Int_{\Irr{\bbGamma}}\lim_{n\to\infty}\Tr\bigl(T_\pi^*\pi(x_{n,N})a_\pi\bigr)\dd{\mu}(\pi).
\end{equation}
Since for almost all $\pi$ the functional $\B(\cH_\pi)\ni{b}\mapsto\Tr(T_\pi^*ba_\pi)\in\CC$ is normal (and hence strongly continuous on bounded sets), we have
\[
\Tr\bigl(T_\pi^*\pi(x_{n,N})a_\pi\bigr)\xrightarrow[n\to\infty]{}
\chi_{\Omega_N}(\pi)\Tr(T_\pi^*T_\pi{a_\pi}).
\]
This convergence together with equation \eqref{eq11} show that $\Tr(T_\pi^*T_\pi{a_\pi})=0$ and consequently $T_\pi{a_\pi}=0$ for almost all $\pi\in\Omega_N$. Consequently, as $N$ was arbitrary and $a_\pi$ has trivial kernel, we obtain $T=0$ which proves surjectivity of $\cQ_\LL$. We are left to show point \eqref{prop2,4}.

First we show that $\cQ_\LL{J_{\bh_{\hh{\bbGamma}}}}\cQ_\LL^*$ is equal to the direct integral $J={\Int_{\Irr{\bbGamma}}}^{\!\!\!\oplus}J_\pi\dd{\mu}(\pi)$ of canonical conjugations on $\HS(\cH_\pi)$. Choose, as usual, orthonormal bases in $\cH_\alpha$ ($\alpha\in\Irr{\hh{\bbGamma}}$) in which operators $\uprho_\alpha$ are diagonal with eigenvalues $\uprho_{\alpha,i}$ and let $x$ be a matrix element $U^{\alpha}_{i,j}$. Then $\sigma^{\bh_{\hh{\bbGamma}}}_t(U^{\alpha}_{i,j})=(\uprho_{\alpha,i}\uprho_{\alpha_j})^{\ii{t}}U^{\alpha}_{i,j}$ ($t\in\RR$) and hence, by equation \eqref{eq12}, we have
\[
a_\pi^{2k}\pi(U^{\alpha}_{i,j})=\pi(U^{\alpha}_{i,j})\bigl(\sqrt{\uprho_{\alpha,i}\uprho_{\alpha,j}}a_\pi\bigr)^{2k},\qquad{k}\in\NN.
\]
Consequently
\[
a_\pi^{2\ii{t}}\pi(U^{\alpha}_{i,j})=\pi(U^{\alpha}_{i,j})\bigl(\sqrt{\uprho_{\alpha,i}\uprho_{\alpha,j}}a_\pi\bigr)^{2\ii{t}}
=\bigl(\uprho_{\alpha,i}\uprho_{\alpha,j}\bigr)^{\ii{t}}\pi(U^{\alpha}_{i,j})a_\pi^{2\ii{t}}
=\pi\bigl(\sigma^{\bh_{\hh{\bbGamma}}}_t(U^{\alpha}_{i,j})\bigr)a_\pi^{2\ii{t}},\qquad{t}\in\RR.
\]
By linearity and density we obtain
\[
a_\pi^{2\ii{t}}\pi(x)=\pi\bigl(\sigma^{\bh_{\hh{\bbGamma}}}_t(x)\bigr)a_\pi^{2\ii{t}},\qquad{t}\in\RR
\]
for all $x\in\C(\hh{\bbGamma})$. Take $y\in\Pol(\hh{\bbGamma})$. Using the above equation, we obtain
\begin{align*}
J\cQ_\LL\Lambda_{\bh_{\hh{\bbGamma}}}(y)
&=J{\Int_{\Irr{\bbGamma}}}^{\!\!\oplus}\pi(y)a_\pi\dd{\mu}(\pi)
={\Int_{\Irr{\bbGamma}}}^{\!\!\oplus}a_\pi\pi(y^*)\dd{\mu}(\pi)
={\Int_{\Irr{\bbGamma}}}^{\!\!\oplus}\pi\bigl(\sigma^{\bh_{\hh{\bbGamma}}}_{\mihalf}(y^*)\bigr)a_\pi\dd{\mu}(\pi)\\
&=\cQ_\LL\Lambda_{\bh_{\hh{\bbGamma}}}\bigr(\sigma^{\bh_{\hh{\bbGamma}}}_{\mihalf}(y^*)\bigr)
=\cQ_\LL\modOphalf[\bh_{\hh{\bbGamma}}]\Lambda_{\bh_{\hh{\bbGamma}}}(y^*)=\cQ_\LL{J_{\bh_{\hh{\bbGamma}}}}\Lambda_{\bh_{\hh{\bbGamma}}}(y).
\end{align*}
Hence the density of $\Lambda_{\bh_{\hh{\bbGamma}}}(\Pol(\hh{\bbGamma}))$ in $\Ltwo(\hh{\bbGamma})$ proves the claim.

To prove that $\bigl(\cQ_\LL,(a_\pi^{-1})_{\pi\in\Irr{\bbGamma}},\mu\bigr)$ are the left Plancherel objects for $\bbGamma$, we will use \cite[Theorem 3.3 (7)]{KrajczokTypeI}. The assumptions of that theorem include that almost all $\pi$ are finite dimensional. However, an inspection of the proof shows that a weaker assumption is sufficient -- that the operator $D_\pi^{-1}$, or in our case $a_\pi$, is bounded -- which we assume. We need to show that for all $\omega\in\ell^1(\bbGamma)$
\begin{subequations}\label{eq14}
\begin{align}
\cQ_\LL(\omega\tens\id)\ww^{\bbGamma}
&=\Biggl(\,{\Int_{\Irr{\bbGamma}}}^{\!\!\oplus}\pi\bigl((\omega\tens\id)\ww^{\bbGamma}\bigr)\tens\I_{\overline{\cH_\pi}}\dd{\mu}(\pi)\Biggr)\cQ_\LL,\label{eq14-1}\\
\cQ\LL(\id\tens\omega)\vv^{\bbGamma}
&=\Biggl(\,{\Int_{\Irr{\bbGamma}}}^{\!\!\oplus}\I_{\cH_\pi}\tens\pi^\cc\bigl((\omega\tens\id)\ww^{\bbGamma}\bigr)\dd{\mu}(\pi)\Biggr)\cQ\LL,\label{eq14-2}
\end{align}
\end{subequations}
where $\pi^\cc\colon\C(\hh{\bbGamma})\ni{x}\mapsto\pi(R_{\hh{\bbGamma}}(x))^{\top}\in\B(\overline{\cH_\pi})$ (cf.~\cite[Proposition 10]{SoltanWoronowicz}). Equality \eqref{eq14-1} follows readily from the definition of $\cQ_\LL$, so we only need to justify \eqref{eq14-2}. We will do this using $\cQ_\LL{J_{\bh_{\hh{\bbGamma}}}}\cQ_\LL^*=J$. Take $y\in\Pol(\hh{\bbGamma})$, a finitely supported $\omega\in\ell^1(\bbGamma)$ and recall that $\vv^{\bbGamma}
=(J_{\bh_{\hh{\bbGamma}}}\tens{J_{\bh_{\hh{\bbGamma}}}})\hh{\ww^{\bbGamma}}(J_{\bh_{\hh{\bbGamma}}}\tens{J_{\bh_{\hh{\bbGamma}}}})$. We have
\begin{align*}
\cQ_\LL\bigl((\id\tens\omega)\vv^{\bbGamma}\bigr)\Lambda_{\bh_{\hh{\bbGamma}}}(y)
&=\cQ_\LL{J_{\bh_{\hh{\bbGamma}}}}\Bigl(\bigl((\omega\comp{R_{\bbGamma}})\tens\id\bigr)\ww^{\bbGamma}\Bigr)^*J_{\bh_{\hh{\bbGamma}}}\Lambda_{\bh_{\hh{\bbGamma}}}(y)\\
&=\cQ_\LL{J_{\bh_{\hh{\bbGamma}}}}\cQ_\LL^*{\Int_{\Irr{\bbGamma}}}^{\!\!\oplus}\pi\Bigl(\bigl((\omega\comp{R_{\bbGamma}})\tens\id\bigr)\ww^{\bbGamma}\Bigr)^*a_\pi\pi(y^*)\dd{\mu}(\pi)\\
&={\Int_{\Irr{\bbGamma}}}^{\!\!\oplus}\pi(y)a_\pi\pi\Bigl(\bigl((\omega\circ{R_{\bbGamma}})\tens\id\bigr)\ww^{\bbGamma}\Bigr)\dd{\mu}(\pi)\\
&={\Int_{\Irr{\bbGamma}}}^{\!\!\oplus}\pi(y)a_\pi(\pi\comp{R_{\hh{\bbGamma}}})\bigl((\omega\tens\id)\ww^{\bbGamma}\bigr)\dd{\mu}(\pi)\\
&=\Biggl(\,{\Int_{\Irr{\bbGamma}}}^{\!\!\oplus}\I_{\cH_\pi}\tens\pi^\cc\bigl((\omega\tens\id)\ww^{\bbGamma}\bigr)\dd{\mu}(\pi)\Biggr)
\Biggl(\,{\Int_{\Irr{\bbGamma}}}^{\!\!\oplus}\pi(y)a_\pi\dd{\mu}(\pi)\Biggr)\\
&=\Biggl(\,{\Int_{\Irr{\bbGamma}}}^{\!\!\oplus}\I_{\cH_\pi}\tens\pi^\cc\bigl((\omega\tens\id)\ww^{\bbGamma}\bigr)\dd{\mu}(\pi)\Biggr)\cQ_\LL\Lambda_{\bh_{\hh{\bbGamma}}}(y).
\end{align*}
A density argument allows us to conclude that equations \eqref{eq14} hold. The last element we need to complete the proof is to show that condition (7.3) of \cite[Theorem 3.3]{KrajczokTypeI} holds, i.e.~that $\cQ_\LL$ transforms $\cZ(\Linf(\hh{\bbGamma}))$ into diagonalizable operators $\oon{Diag}\biggl(\,{\Int_{\Irr{\bbGamma}}}^{\!\!\!\oplus}\HS(\cH_\pi)\dd{\mu}(\pi)\biggr)$. This is a consequence of equations \eqref{eq14}. Indeed, from \eqref{eq14} we obtain
\[
\cQ_\LL\Linf(\hh{\bbGamma})\cQ_\LL^*={\Int_{\Irr{\bbGamma}}}^{\!\!\oplus}\B(\cH_\pi)\tens\I_{\overline{\cH_\pi}}\dd{\mu}(\pi),\quad
\cQ_\LL\Linf(\hh{\bbGamma})'\cQ_\LL^*={\Int_{\Irr{\bbGamma}}}^{\!\!\oplus}\I_{\cH_\pi}\tens\B(\overline{\cH_\pi})\dd{\mu}(\pi)
\]
and hence
\begin{align*}
\cQ_\LL\cZ\bigl(\Linf(\hh{\bbGamma})\bigr)\cQ_\LL^*&=\cQ_\LL\bigl(\Linf(\hh{\bbGamma})\cap\Linf(\hh{\bbGamma})'\bigr)\cQ_\LL^*\\
&={\Int_{\Irr{\bbGamma}}}^{\!\!\oplus}\B(\cH_\pi)\tens\I_{\overline{\cH_\pi}}\dd{\mu}(\pi)
\cap{\Int_{\Irr{\bbGamma}}}^{\!\!\oplus}\I_{\cH_\pi}\tens\B(\overline{\cH_\pi})\dd{\mu}(\pi)\\
&=\oon{Diag}\Biggl(\,{\Int_{\Irr{\bbGamma}}}^{\!\!\oplus}\HS(\cH_\pi)\dd{\mu}(\pi)\Biggr).
\qedhere
\end{align*}
\end{proof}

\subsection{General root systems}\label{sect:generalRS}\hspace*{\fill}

In this and the subsequent section, we will identify all the invariants for $G_q$ and its dual (see Remark \ref{remark1} and Theorem \ref{thm2}, Proposition \ref{prop4}, Theorem \ref{thm3}) and give a concrete description for the action of scaling automorphisms (Theorem \ref{thm1}). In particular, we show that in some cases, e.g.~for $\oon{SU}_q(3)$, there exists a non-trivial, inner, scaling automorphism which is not implemented by any group-like unitary. This shows that $\TtauInn(\oon{SU}_q(3))$ is not equal to the invariant $T(\Linf(\oon{SU}_q(3)),\Delta)$ of Vaes (\cite[Definition 3.4]{StrictlyOuter}).

We begin our analysis with formulas expressing the action of modular and scaling automorphisms on matrix coefficients of irreducible representations. As in the previous section, $G$ is a compact, simply connected, semisimple Lie group and $0<q<1$.

\begin{lemma}\label{lemma4}
Take $t\in\RR$, $\varpi\in\bP^+$ and $x\in\Pol(G_q)$ with well defined weights $\lwt(x),\rwt(x)\in\bP$. Then
\begin{enumerate}
\item\label{lemma4.1} $\uprho_\varpi=\pi_\varpi(K_{2\rho})\in\B(\cH_\varpi)$,
\item\label{lemma4.2} $\sigma^{\bh_{G_q}}_t(x)=q^{\la2\rho|\lwt(x)+\rwt(x)\ra\ii{t}}x$,
\item\label{lemma4.3} $\tau^{G_q}_t(x)=q^{\la2\rho|\lwt(x)-\rwt(x)\ra\ii{t}}x$.
\end{enumerate}
\end{lemma}

\begin{proof}
Take $\varpi\in\bP^+$ and $\xi,\eta\in\cH_\varpi$ with well defined weights. Clearly, in order to prove \eqref{lemma4.2} and \eqref{lemma4.3} it is enough to consider case $x=U^\varpi(\xi,\eta)$.

Let $\{f_z\,|\,z\in\CC\}$ be the Woronowicz characters of $G_q$ (\cite[Definition 1.7.1]{NeshveyevTuset}). From \cite[Proposition 2.4.10]{NeshveyevTuset} we know that $f_1$ is the image of $K_{2\rho}$ in $\Pol(G_q)^*$ (note that we changed $K_i$ to $K_i^{-1}$ to be consistent with the conventions of \cite{DeCommer}), hence we can calculate
\[
\is{\xi}{\uprho_\varpi\eta}=f_1\bigl(U^\varpi(\xi,\eta)\bigr)
=U^\varpi(\xi,\eta)(K_{2\rho})
=\is{\xi}{\pi_\varpi(K_{2\rho})\eta}.
\]
As $\xi,\eta$ with well defined weights span $\cH_\varpi$, we obtain \eqref{lemma4.1}.

Recall that $\sigma^{\bh_{G_q}}_t(x)=f_{\ii{t}}\star{x}\star{f_{\ii{t}}}$ (\cite[Page 30]{NeshveyevTuset}), so using $\uprho_\varpi^{\ii{t}}=\pi_\varpi(K_{2\rho})^{\ii{t}}$ we obtain
\begin{align*}
\sigma^{\bh_{G_q}}_t(x)
&=\sigma^{\bh_{G_q}}_t\bigl(U^\varpi(\xi,\eta)\bigr)
=\sum_{i,j=1}^{\dim{\varpi}}
f_{\ii{t}}\bigl(U^\varpi(\xi,\xi_i^\varpi)\bigr)U^\varpi(\xi_i^\varpi,\xi_j^\varpi)
f_{\ii{t}}\bigl(U^\varpi(\xi_j^\varpi,\eta)\bigr)\\
&=\sum_{i,j=1}^{\dim{\varpi}}
\is{\xi}{\uprho_\varpi^{\ii{t}}\xi_i^\varpi}\,
\is{\xi_j^\varpi}{\uprho_\varpi^{\ii{t}}\eta}\,
U^\varpi(\xi_i^\varpi,\xi_j^\varpi)\\
&=\sum_{i,j=1}^{\dim{\varpi}}
\is{\pi_\varpi(K_{2\rho})^{-\ii{t}}\xi}{\xi_i^\varpi}\,
\is{\xi_j^\varpi}{\pi_\varpi(K_{2\rho})^{\ii{t}}\eta}\,
U^\varpi(\xi_i^\varpi,\xi_j^\varpi)\\
&=\sum_{i,j=1}^{\dim{\varpi}}
\is{q^{-\la2\rho|\wt(\xi)\ra\ii{t}}\xi}{\xi_i^\varpi}\,
\is{\xi_j^\varpi}{q^{\la2\rho|\wt(\eta)\ra\ii{t}}\eta}\,
U^\varpi(\xi_i^\varpi,\xi_j^\varpi)\\
&=q^{\la2\rho|\lwt(x)+\rwt(x)\ra\ii{t}}x.
\end{align*}
Statement \eqref{lemma4.3} can be obtained in an analogous way this time using the formula $\tau^{G_q}_t(x)=f_{-\ii{t}}\star{x}\star{f_{\ii{t}}}$.
\end{proof}

As in the proof of Proposition \ref{prop5} we let $P$ denote the strictly positive, self-adjoint operator on $\Ltwo(G_q)$ which implements the scaling group of $G_q$. Using formula for $\tau^{G_q}_t$ from Lemma \ref{lemma4}, we are able to describe the action of $P^{\ii{t}}$ at the level of the direct integral ${\Int_T}^\oplus\HS(\cH_\lambda)\dd\lambda$. For this we first introduce the following:

\begin{lemma}\label{lemma5}
There is a continuous homomorphism $\RR\ni{t}\mapsto\blambda_t\in{T}$ characterized by $\la\varpi,\blambda_t\ra=q^{\la2\rho|\varpi\ra\ii{t}}$ for $\varpi\in\bP$.
\end{lemma}

\begin{proof}
Clearly $\bP\ni\varpi\mapsto{q^{\la2\rho|\varpi\ra\ii{t}}}\in\TT$ is a continuous homomorphism, hence the existence of $\blambda_t$ follows from Pontriagin duality. To see that $t\mapsto\blambda_t$ is continuous, consider the isomorphisms $T\cong\TT^r$ and $\bP\cong\ZZ^r$.
\end{proof}

\begin{proposition}\label{prop3}
For $t\in\RR$ the operator $\cQ_\LL{P^{\ii{t}}}\cQ_\LL^*\in\B\Bigl({\Int_T}^\oplus\HS(\cH_\lambda)\dd{\lambda}\Bigr)$ acts as
\[
\cQ_\LL{P^{\ii{t}}}\cQ_\LL^*
\Biggl(
{\Int_T}^{\oplus}A_{\lambda}\dd{\lambda}
\Biggr)
={\Int_T}^\oplus\pi_{w_\circ}
\bigl(|b_\rho|\bigr)^{-2\ii{t}}
A_{\lambda\blambda_{2t}}\pi_{w_\circ}\bigl({|b_\rho|}\bigr)^{2\ii{t}}\dd{\lambda}.
\]
\end{proposition}

\begin{proof}
Take $x=U^{\varpi}(\xi,\eta)$ for $\varpi\in\bP^+$ and $\xi,\eta\in\cH_\varpi$ with well defined weights. Using Lemma \ref{lemma4} we have
\begin{equation}\label{eq24}
P^{\ii{t}}\Lambda_{\bh_{G_q}}(x)
=\Lambda_{\bh_{G_q}}\Bigl(\tau^{G_q}_t\bigl(U^{\varpi}(\xi,\eta)\bigr)\Bigr)
=q^{\la2\rho|\lwt(x)-\rwt(x)\ra\ii{t}}\Lambda_{\bh_{G_q}}(x).
\end{equation}
Meanwhile, Corollary \ref{cor1} implies
\begin{equation}\label{eq25}
\cQ_\LL\Lambda_{\bh_{G_q}}(x)
=\biggl(\prod_{\alpha\in\Phi^+}\bigl(1-q^{2\la\rho|\alpha\ra}\bigr)\biggr)^{\frac{1}{2}}
{\Int_T}^\oplus
\pi_{\lambda,w_\circ}\bigl(x|b_\rho|\bigr)\dd{\lambda}.
\end{equation}
Hence, combining \eqref{eq24} and \eqref{eq25} gives
\[
\cQ_\LL{P^{\ii{t}}}\cQ_\LL^*
\Biggl({\Int_T}^\oplus
\pi_{\lambda,w_\circ}\bigl(x|b_\rho|\bigr)\dd{\lambda}\Biggr)
=q^{\la2\rho|\lwt(x)-\rwt(x)\ra\ii{t}}{\Int_T}^\oplus
\pi_{\lambda,w_\circ}\bigl(x|b_\rho|\bigr)\dd{\lambda}.
\]
Let $\mathcal{R}_t$ be the operator from the claim, i.e.
\[
\mathcal{R}_t
\Biggl({\Int_T}^{\oplus}A_{\lambda}\dd{\lambda}\Biggr)={\Int_T}^\oplus\pi_{w_\circ}
\bigl(|b_\rho|\bigr)^{-2\ii{t}}
A_{\lambda\blambda_{2t}}\pi_{w_\circ}\bigl(|b_\rho|\bigr)^{2\ii{t}}\dd{\lambda}.
\]
Since $\cQ_\LL$ is unitary, vectors of the form ${\Int_T}^\oplus\pi_{\lambda,w_\circ}\bigl(x|b_\rho|\bigr)\dd{\lambda}$ constitute a linearly dense set in ${\Int_T}^\oplus\HS(\cH_\lambda)\dd{\lambda}$. Thus it is enough to show that
\begin{equation}\label{eq16}
\mathcal{R}_t\Biggl(
{\Int_T}^\oplus\pi_{\lambda,w_\circ}\bigl(x|b_\rho|\bigr)\dd{\lambda}\Biggr)=
q^{\la2\rho|\lwt(x)-\rwt(x)\ra\ii{t}}
{\Int_T}^\oplus\pi_{\lambda,w_\circ}\bigl(x|b_\rho|\bigr)
\dd{\lambda}.
\end{equation}

We have $\pi_{\lambda,w_\circ}(|b_\rho|)=\pi_{w_\circ}(|b_\rho|)$ for any $\lambda\in{T}$ (equality \eqref{eq32}). Using equation \eqref{eq17} and the definition of $\mathcal{R}_t$ we can calculate the left hand side of \eqref{eq16} as
\begin{align*}
\mathcal{R}_t\Biggl(
{\Int_T}^\oplus\pi_{\lambda,w_\circ}\bigl(x|b_\rho|\bigr)\dd{\lambda}\Biggr)&=
{\Int_T}^\oplus\pi_{w_\circ}\bigl(|b_\rho|\bigr)^{-2\ii{t}}
\pi_{\lambda\blambda_{2t},w_\circ}\bigl(x|b_\rho|\bigr)\pi_{w_\circ}\bigl(|b_\rho|\bigr)^{2\ii{t}}\dd{\lambda}\\
&={\Int_T}^\oplus\pi_{w_\circ}\bigl(|b_\rho|\bigr)^{-2\ii{t}}
\pi_{w_\circ}(x)\bigl\la\lwt(x),\lambda\blambda_{2t}\bigr\ra\pi_{\lambda\blambda_{2t},w_\circ}\bigl(|b_\rho|\bigr)\pi_{w_\circ}\bigl(|b_\rho|\bigr)^{2\ii{t}}\dd{\lambda}\\
&=\bigl\la\lwt(x),\blambda_{2t}\bigr\ra{\Int_T}^\oplus
\pi_{w_\circ}\bigl(|b_\rho|\bigr)^{-2\ii{t}}
\pi_{\lambda,w_\circ}\bigl(x|b_\rho|\bigr)
\pi_{w_\circ}\bigl(|b_\rho|\bigr)^{2\ii{t}}\dd{\lambda}.
\end{align*}
Using
\[
\cQ_{\LL}\modOp[\bh_{G_q}]^{\ii{t}}
\cQ_{\LL}^*=
{\Int_T}^\oplus\pi_{\lambda,w_\circ}\bigl(|b_\rho|\bigr)^{2\ii{t}}\tens\pi_{\lambda,w_\circ}\bigl(|b_\rho|^{-2\ii{t}}\bigr)^\top\dd{\lambda}
\]
(Corollary \ref{cor1}, \cite[Theorem 5.4]{modular}), Lemma \ref{lemma4} and the definition of $\blambda_{2t}$ (Lemma \ref{lemma5}) we can continue the calculation
\begin{align*}
\mathcal{R}_t\Biggl(
{\Int_T}^\oplus\pi_{\lambda,w_\circ}\bigl(x|b_\rho|\bigr)\dd{\lambda}\Biggr)
&=q^{\la2\rho|\lwt(x)\ra2\ii{t}}\cQ_{\LL}\modOp[\bh_{G_q}]^{-\ii{t}}\Lambda_{\bh_{G_q}}(x)\\
&=q^{\la2\rho|\lwt(x)\ra2\ii{t}}
q^{-\la2\rho|\lwt(x)+\rwt(x)\ra\ii{t}}
\cQ_{\LL}\Lambda_{\bh_{G_q}}(x)\\
&=q^{\la2\rho|\lwt(x)-\rwt(x)\ra\ii{t}}
{\Int_T}^\oplus\pi_{\lambda,w_\circ}\bigl(x|b_\rho|\bigr)\dd{\lambda}.
\end{align*}
This shows \eqref{eq16} and proves the claim.
\end{proof}

As a corollary of Proposition \ref{prop3} we obtain an alternative formula for the action of scaling group. Recall (\cite[Theorem 3.3]{KrajczokTypeI}) that $\cQ_{\LL}$ induces isomorphism between $\Linf(G_q)$ and ${\Int_T}^\oplus\B(\cH_\lambda)\tens\I_{\overline{\cH_\lambda}}\dd{\lambda}$.

\begin{theorem}\label{thm1}
For any $t\in\RR$ and $x\in\Linf(G_q)$ with $\cQ_{\LL}x\cQ_{\LL}^*={\Int_T}^{\oplus}x_\lambda\tens\I_{\overline{\cH_\lambda}}\dd{\lambda}$ we have
\[
\cQ_{\LL}\tau^{G_q}_t(x)\cQ_{\LL}^*={\Int_T}^\oplus
\pi_{w_\circ}\bigl(|b_\rho|\bigr)^{-2\ii{t}}x_{\lambda\blambda_{2t}}\pi_{w_\circ}\bigl(|b_\rho|\bigr)^{2\ii{t}}\tens\I_{\overline{\cH_\lambda}}\dd{\lambda}.
\]
\end{theorem}

\begin{proof}
For any ${\Int_T}^{\oplus}A_\lambda\dd{\lambda}\in{\Int_T}^\oplus\HS(\cH_\lambda)\dd{\lambda}$ we have using the formula for $P^{\ii{t}}$ from Proposition \ref{prop3}
\begin{align*}
\cQ_{\LL}\tau^{G_q}_t(x)\cQ_{\LL}^*
\Biggl({\Int_T}^{\oplus}A_\lambda\dd{\lambda}\Biggr)
&=\cQ_{\LL}P^{\ii{t}}xP^{-\ii{t}}\cQ_{\LL}^*
\Biggl({\Int_T}^{\oplus}A_\lambda\dd{\lambda}\Biggr)\\
&=\cQ_{\LL}P^{\ii{t}}\cQ_{\LL}^*\Biggl({\Int_T}^{\oplus}
x_\lambda\pi_{w_\circ}\bigl(|b_\rho|\bigr)^{2\ii{t}}A_{\lambda\blambda_{-2t}}\pi_{w_\circ}\bigl(|b_\rho|\bigr)^{-2\ii{t}}
\dd{\lambda}\Biggr)\\
&={\Int_T}^\oplus\pi_{w_\circ}\bigl(|b_\rho|\bigr)^{-2\ii{t}}
x_{\lambda\blambda_{2t}}\pi_{w_\circ}\bigl(|b_\rho|\bigr)^{2\ii{t}}
A_{\lambda\blambda_{-2t}\blambda_{2t}}
\pi_{w_\circ}\bigl(|b_\rho|\bigr)^{-2\ii{t}}
\pi_{w_\circ}\bigl(|b_\rho|\bigr)^{2\ii{t}}\dd{\lambda}\\
&={\Int_T}^\oplus\pi_{w_\circ}\bigl(|b_\rho|\bigr)^{-2\ii{t}}
x_{\lambda\blambda_{2t}}\pi_{w_\circ}\bigl(|b_\rho|\bigr)^{2\ii{t}}A_{\lambda}\dd{\lambda}
\end{align*}
which proves the claim.
\end{proof}

\begin{remark}
In \cite[Proposition 7.3]{modular} the first named author obtained a similar result in the case of $\oon{SU}_q(2)$. The formula for $P^{\ii{t}}$ is different -- it consists only of rotation, without composition with inner automorphism. This difference is a result of choosing different field of representations $(\pi_{\lambda,w_\circ})_{\lambda\in{T}}$ (c.f.~the beginning of this section and \cite[Proposition 7.1]{modular}).
\end{remark}

As a corollary of Theorem \ref{thm1} we obtain equality of invariants $\TtauAInn(G_q)$ and $\TtauInn(G_q)$.

\begin{proposition}\label{prop7}
We have $\TtauAInn(G_q)=\TtauInn(G_q)$.
\end{proposition}

\begin{proof}
Take $t\in\TtauAInn(G_q)$ and let $(u_n)_{n\in\NN}$ be a sequence of unitaries in $\Linf(G_q)$ such that $\oon{Ad}(u_n)\xrightarrow[n\to\infty]{}\tau^{G_q}_t$. Theorem \ref{thm1} shows that in order to deduce $t\in\TtauInn(G_q)$, it is enough to show that $\blambda_{2t}$ is the neutral element in $T$. Fix an arbitrary measurable subset $\Omega\subset{T}$ and consider central element
\[
x=\cQ_{\LL}^*\Biggl({\Int_T}^\oplus\chi_{\Omega}(\lambda)\I_{\cH_\lambda}\tens\I_{\overline{\cH_\lambda}}\dd{\lambda}\Biggr)\cQ_{\LL}\in\cZ(\Linf(G_q)).
\]
We have
\begin{align*}
\tau^{G_q}_t(x)
&=\cQ_{\LL}^*\Biggl({\Int_T}^\oplus\pi_{w_\circ}
\bigl(|b_\rho|\bigr)^{-2\ii{t}}
\chi_{\Omega}(\lambda\blambda_{2t})
\I_{\cH_\lambda}\pi_{w_\circ}\bigl(|b_\rho|\bigr)^{2\ii{t}}
\tens\I_{\overline{\cH_\lambda}}\dd{\lambda}\Biggr)\cQ_{\LL}\\
&=\cQ_{\LL}^*\Biggl({\Int_T}^\oplus\chi_{\blambda_{-2t}\Omega}(\lambda)\I_{\cH_\lambda}\tens\I_{\overline{\cH_\lambda}}\dd{\lambda}\Biggr)\cQ_{\LL}.
\end{align*}
On the other hand, since $x$ is central we have
\[
\tau^{G_q}_t(x)
={\mathrm{w}^*\text{-}\!\!}\lim_{n\to\infty}\Ad(u_n)(x)
={\mathrm{w}^*\text{-}\!\!}\lim_{n\to\infty}u_nxu_n^*=x
=\cQ_{\LL}^*\Biggl({\Int_T}^\oplus\chi_{\Omega}(\lambda)\I_{\cH_\lambda}\tens\I_{\overline{\cH_\lambda}}\dd{\lambda}\Biggr)\cQ_{\LL}.
\]
It follows that the symmetric difference $\blambda_{-2t}\Omega\div\Omega$ is null for all measurable $\Omega\subset{T}$, which clearly implies that $\blambda_{2t}={\blambda_{-2t}}^{-1}$ is the neutral element in $T$.
\end{proof}

In the next theorem (Theorem \ref{thm2}) we show how to calculate the remaining invariants of Section \ref{sect:defInv} using Lie-theoretic information. We need to start with two lemmas. For what follows, recall that we take the normalization in which $\la\beta|\beta\ra=2$ for short roots $\beta\in\Phi$.

\begin{lemma}\label{lemma6}
\noindent
\begin{enumerate}
\item\label{lemma6-1} $\bigl\{\la2\rho|\alpha\ra\,\bigr|\bigl.\,\alpha\in\bQ\bigr\}=2\ZZ$,
\item\label{lemma6-2} The set $\bigl\{\la2\rho|\varpi\ra\,\bigr|\bigl.\,\varpi\in{\bP}\bigr\}$ is a nontrivial subgroup of $\ZZ$.
\end{enumerate}
\end{lemma}

Lemma \ref{lemma6} allows us to make the following definition.

\begin{definition}\label{def2}
We define $\Upsi{\Phi}\in\NN$ to be the number determined uniquely by
\[
\bigl\{\la2\rho|\varpi\ra\,\bigr|\bigl.\,\varpi\in\bP\bigr\}
=\Upsi{\Phi}\ZZ.
\]
\end{definition}

The symbol we use for the positive integer defined above should in principle include a reference to the compact semisimple simply connected Lie group $G$. However for reasons which will become clear below we chose to use notation which refers to the root system $\Phi$ of the Lie algebra of $G$ (which contains the full information on $G$).

\begin{proof}[Proof of Lemma \ref{lemma6}]
Both sets appearing in the statement of Lemma \ref{lemma6} are subgroups of $\RR$, hence it is enough to show that for $\alpha\in\bQ$ and $\varpi\in\bP$ we have
\begin{equation}\label{eq:lemma611}
\la2\rho|\alpha\ra\in2\ZZ,\qquad\la2\rho|\varpi\ra\in\ZZ
\end{equation}
and that there is a root $\alpha\in\bQ$ satisfying $\la2\rho|\alpha\ra=2$. As $\bQ$ is the $\ZZ$-linear span of simple roots $\{\alpha_1,\dotsc,\alpha_r\}$ and $\bP$ is the $\ZZ$-linear span of the fundamental weights $\{\varpi_1,\dotsc,\varpi_r\}$, in order to show \eqref{eq:lemma611} it is enough to consider $\alpha=\alpha_i$ and $\varpi=\varpi_i$.

Ad \eqref{lemma6-1}. Write $\rho$ as $\rho=\sum\limits_{j=1}^{r}\varpi_j$, then
\begin{equation}\label{eq29}
\la2\rho|\alpha_i\ra=
2\sum_{j=1}^{r}\la\varpi_j|\alpha_i\ra
=2\sum_{j=1}^{r}\delta_{i,j}\tfrac{\la\alpha_i|\alpha_i\ra}{2}
=\la\alpha_i|\alpha_i\ra.
\end{equation}
We claim that $\la\alpha_i|\alpha_i\ra\in2\NN$. Indeed, as any root system is equal to a product of irreducible ones (\cite[Proposition 11.3]{Humphreys}), it is enough to assume that $\Phi$ is irreducible. Then for any root $\beta\in\Phi$ we have $\la\beta|\beta\ra\in\{2,4,6\}$ (see \cite[Lemma 10.4 C]{Humphreys} and its proof) and consequently $\la2\rho|\alpha_i\ra\in2\NN$ by \eqref{eq29}. Furthermore, there is $\alpha_i\in\Phi$ of length $\sqrt{2}$ which proves \eqref{lemma6-1}.

Ad \eqref{lemma6-2}.
Recall that $\rho$ can be expressed also as $\rho=\tfrac{1}{2}\sum\limits_{\beta\in\Phi^+}\beta$. Furthermore, each positive root $\beta\in\Phi^+$ can be written uniquely as $\beta=\sum\limits_{j=1}^{r}c(\beta,j)\alpha_j$ for some $c(\beta,j)\in\ZZ_+$ (\cite[Section 10.1]{Humphreys}). We have
\[
\la2\rho|\varpi_i\ra
=\sum_{\beta\in\Phi^+}\la\beta|\varpi_i\ra
=\sum_{\beta\in\Phi^+}\sum_{j=1}^{r}c(\beta,j)\la\alpha_j|\varpi_i\ra
=\sum_{\beta\in\Phi^+}\sum_{j=1}^{r}c(\beta,j)\tfrac{\la\alpha_j|\alpha_j\ra}{2}\delta_{i,j}
=\sum_{\beta\in\Phi^+}c(\beta,i)\tfrac{\la\alpha_i|\alpha_i\ra}{2}\in\NN
\]
again using $\la\alpha_i|\alpha_i\ra\in2\NN$.
\end{proof}

We can state our main result of this section using number $\Upsi{\Phi}$ introduced in Definition \ref{def2}.

\begin{theorem}\label{thm2}
For any $0<q<1$ and compact, simply connected,
semisimple Lie group $G$ we have
\[
\Ttau(G_q)=\tfrac{\pi}{\log{q}}\ZZ
\]
and
\[
\Mod(\hh{G_q})=\TtauInn(G_q)=\TtauAInn(G_q)=\tfrac{\pi}{\Upsi{\Phi}\log{q}}\ZZ.
\]
\end{theorem}

We begin with a lemma.

\begin{lemma}\label{lemma7}
We have
\[
\bigl\{\lwt(x)-\rwt(x)\,\bigr|\bigl.\,x\in\Pol(G_q)\textnormal{ with well defined weights }\lwt(x),\rwt(x)\in{\bP}\bigr\}=\bQ.
\]
\end{lemma}

\begin{proof}
We will use the result \cite[Proposition 21.3]{Humphreys}: weights of $\varpi\in\bP^+$ are precisely these $\mu\in\bP$ which satisfy the following condition: for all $w$ in the Weyl group, $\varpi-w\mu$ is equal to $0$ or can be written as a sum of positive roots.

Take $0\neq{x}=U^\varpi(\xi,\eta)$, where $\lwt(x)=\wt(\xi),\rwt(x)=\wt(\eta)$. This means that $\wt(\xi),\wt(\eta)$ are weights of $\varpi$, hence the result quoted above implies that $\lwt(x)-\rwt(x)=(\varpi-\wt(\eta))-(\varpi-\wt(\xi))$ belongs to the $\ZZ$-linear span of positive roots, i.e.~root lattice $\bQ$.

Conversely, take $\mu\in\bQ$. By \cite[Theorem 20.1]{Bump}, there exists $w\in{W}$ such that $\varpi=w\mu\in\bQ\cap\bP^+$. We claim that $\mu$ and $0$ are weights of $\varpi$. The first statement follows as $\varpi$ is a weight of representation $\varpi$ and the set of weights of $\varpi$ is invariant under the action of Weyl group.

To see that $0$ is also a weight of $\varpi$, let us write $\varpi=\sum\limits_{i=1}^{r}c_i\alpha_i$ with some $c_i\in\ZZ$. Since $\varpi$ is in the positive Weyl chamber, we have $c_i\geq0$. Indeed, let us decompose $\{1,\dotsc,r\}$ into $I\cup{J}$ in such a way that $c_i\geq0$ for $i\in{I}$ and $c_j<0$ for $j\in{J}$. We have
\begin{equation}\label{eq27}
\biggl\la\sum_{j\in{J}}c_j\alpha_j\biggr|\biggl.\varpi\biggr\ra
=\sum_{j\in{J}}c_j\la\alpha_j|\varpi\ra\leq0
\end{equation}
because $c_j<0$ and $\la\alpha_j|\varpi\ra\geq0$. On the other hand
\begin{equation}\label{eq28}
\biggl\la\sum_{j\in{J}}c_j\alpha_j\biggr|\biggl.\varpi\biggr\ra
=\biggl\|\sum_{j\in{J}}c_j\alpha_j\biggr\|^2+\sum_{j\in{J},i\in{I}}c_jc_i\la\alpha_j|\alpha_i\ra\geq0
\end{equation}
and both terms are non-negative as $\la\alpha_j|\alpha_i\ra\leq0$ by \cite[Lemma 10.1]{Humphreys}. It follows from \eqref{eq27} and \eqref{eq28} that $\sum\limits_{j\in{J}}c_j\alpha_j=0$ and $\varpi$ is a sum of positive roots as claimed. Now \cite[Proposition 21.3]{Humphreys} tells us that $0$ is a weight of $\varpi$. Consequently, we can take $\xi,\eta\in\cH_\varpi\setminus\{0\}$ with $\wt(\xi)=\mu$, $\wt(\eta)=0$ and define $x=U^\varpi(\xi,\eta)$. Then $\mu=\lwt(x)-\rwt(x)$.
\end{proof}

\begin{proof}[Proof of Theorem \ref{thm2}]
We begin by calculating $\Ttau(G_q)$. Take a real number $t$. Using Lemma \ref{lemma4} we see that $t\in\Ttau(G_q)$ if and only if $q^{\la2\rho|\lwt(x)-\rwt(x)\ra\ii{t}}=1$ for all $x\in\Pol(G_q)$ with well defined left and right weight. By Lemma \ref{lemma7} this is equivalent to $q^{\la2\rho|\alpha\ra\ii{t}}=1$ for all $\alpha\in\bQ$ which by Lemma \ref{lemma6} reduces to $q^{2k\ii{t}}=1$ for all $k\in\ZZ$. This last condition is easily seen to be equivalent to $t\in\tfrac{\pi}{\log{q}}\ZZ$.

We already proved in Proposition \ref{prop7} that $\TtauAInn(G_q)=\TtauInn(G_q)$. Next we calculate $\TtauInn(G_q)$.

Take $t\in\RR$. The expression for $\tau^{G_q}_t$ stated in Theorem \ref{thm1} shows that $t\in\TtauInn(G_q)$ if and only if $\blambda_{2t}$ is the trivial element in $T$. Definition of $\blambda_{2t}$ (Lemma \ref{lemma5}) shows that this is equivalent to $q^{\la2\rho|\varpi\ra2\ii{t}}=1$ for all $\varpi\in\bP$ and a similar argument to the one above shows that this is, in turn, equivalent to $t\in\tfrac{2\pi}{2\Upsi{\Phi}\log{q}}\ZZ$.

The last step is to calculate $\Mod(\hh{G_q})$. Recall (\cite[Theorem 3.3 and Remark on page 402]{PodlesWoronowicz}, see also \cite[Proposition 5.5]{modular}) that $\hh{\delta}=\bigoplus\limits_{\varpi\in\bP^+}\uprho_\varpi^{2}$. Consequently, by Lemma \ref{lemma4}, we have $t\in\Mod(\hh{G_q})$ if and only if
\[
\I=\uprho_\varpi^{2\ii{t}}=\pi_\varpi(K_{2\rho})^{2\ii{t}}
\]
for all $\varpi\in\bP^+$. Recall also that $\cH_\varpi$ is the direct sum of weight spaces $\cH_\varpi(\mu)$ and $\pi_\varpi(K_{2\rho})\xi=q^{\la2\rho|\mu\ra}\xi$ for $\xi\in\cH_\varpi(\mu)$. Furthermore, any $\mu\in\bP$ is a weight of some $\varpi\in\bP^+$. Consequently, we have that $t\in\Mod(\hh{G_q})$ if and only if $q^{\la2\rho|\mu\ra2\ii{t}}=1$ for all $\mu\in\bP$. In the previous paragraph we have seen that this condition is equivalent to $t\in\TtauInn(G_q)$.
\end{proof}

Since $\Linf(G_q)$ is semifinite, Theorem \ref{thm2} describes all the invariants of $G_q$ and $\widehat{G_q}$ in terms of one number $\Upsi{\Phi}$ (see Remark \ref{remark1}). This number is easily computable out of Lie-theoretic data. Nonetheless, one might want to obtain a more precise answer, e.g.~what the invariants are for $G_q=\oon{SU}_q(3)$. To provide a more concrete description of $\Upsi{\Phi}$ (and consequently of the invariants) we will decompose the root system $\Phi$ into irreducible parts and describe how to compute $\Upsi{\Phi}$ out of the corresponding numbers for irreducible components. The computation in the irreducible cases will be performed in the next section.

According to \cite[Proposition 11.3]{Humphreys}, we can write $\Phi=\Phi_1\cup\dotsm\cup\Phi_l$ in such a way that for $i\neq{j}$ the roots in $\Phi_i$ and $\Phi_j$ are pairwise orthogonal and each $\Phi_k$ is irreducible. Accordingly, the set of simple roots $\Sigma=\{\alpha_1,\dotsc,\alpha_r\}$ decomposes into $\Sigma=\Sigma_1\cup\dotsm\cup\Sigma_l$, where $\Sigma_k=\Sigma\cap\Phi_k$. It follows that the set of fundamental dominant weights for $\Phi$ is the union of sets of fundamental dominant weights for $\Phi_k$ ($1\leq{k}\leq{l}$). Let $\rho_k$ be the Weyl vector of $\Phi_k$. Using formula $\rho=\tfrac{1}{2}\sum\limits_{\alpha\in\Phi^+}\alpha$ we see that $\rho=\sum\limits_{k=1}^l\rho_k$.

\begin{proposition}\label{prop4}
Let $\Phi=\Phi_1\cup\dotsm\cup\Phi_l$ be the decomposition of $\Phi$ into irreducible parts as described above and let $\Upsi{\Phi}$, $\Upsi{\Phi_k}$ {\rm(}$1\leq{k}\leq{l}${\rm)} be the numbers associated with $\Phi$ and $\Phi_k$ {\rm(}$1\leq{k}\leq{l}${\rm)} in Definition \ref{def2}. Then
\[
\Upsi{\Phi}=\oon{gcd}\bigl(\Upsi{\Phi_1},\dotsc,\Upsi{\Phi_l}\bigr).
\]
\end{proposition}

\begin{proof}
The claim follows from the following calculation:
\begin{align*}
\Upsi{\Phi}\ZZ
=\bigl\{\la2\rho|\varpi\ra\,\bigr|\bigl.\,\varpi\in\bP\bigr\}
&=\oon{span}_{\ZZ}\bigcup_{k=1}^{l}
\bigl\{\la2\rho|\varpi\ra\,\bigr|\bigl.\,\varpi\text{ is a fundamental dominant weight of }\Phi_k\bigr\}\\
&=\oon{span}_{\ZZ}\bigcup_{k=1}^{l}\bigl\{\la2\rho_k|\varpi\ra\,\bigr|\bigl.\,\varpi\text{ is a fundamental dominant weight of }\Phi_k\bigr\}\\
&=\oon{span}_{\ZZ}\bigcup_{k=1}^{l}\Upsi{\Phi_k}\ZZ
\end{align*}
which is equal to $\oon{gcd}(\Upsi{\Phi_1},\dotsc,\Upsi{\Phi_l})\ZZ$ thanks to the B{\'e}zout's identity.
\end{proof}

\subsection{Irreducible cases}\hspace*{\fill}

In this section we calculate numbers $\Upsi{\Phi}$ in the case when $\Phi$ is an irreducible root system. Thanks to Proposition \ref{prop4}, this knowledge is enough to find $\Upsi{\Phi}$ for all root systems of compact, simply connected,
semisimple Lie groups. Theorem \ref{thm2} (see also Remark \ref{remark1}) gives us then modular invariants for all quantum groups $G_q$ and their duals. According to the famous classification result (see \cite{Bump, Humphreys}), if $\Phi$ is irreducible then it belongs to one of $4$ families of classical Cartan types $A_n$ ($n\geq1$), $B_n$ ($n\geq2$), $C_n$ ($n\geq3$), $D_n$ ($n\geq4$) or is one of the $5$ exceptional root systems $E_6$, $E_7$, $E_8$, $F_4$, $G_2$. Let us now state our main result of this section.

\begin{theorem}\label{thm3}
Let $0<q<1$ and let $G$ be a compact, simply connected,
simple Lie group with root system $\Phi$. If $\Phi$ is of classical type then we have the following possibilities:
\begin{itemize}
\item type $A_n${\rm:} $G=\oon{SU}(n+1)$ {\rm(}$n\geq1${\rm)}, if $n$ is odd then $\Upsi{\Phi}=1$ and $\TtauInn(G_q)=\tfrac{\pi}{\log{q}}\ZZ$, if $n$ is even then $\Upsi{\Phi}=2$ and $\TtauInn(G_q)=\tfrac{\pi}{2\log{q}}\ZZ$;
\item type $B_n${\rm:} $G=\oon{Spin}(2n+1)$ {\rm(}$n\geq2${\rm)}, if $n$ is odd then $\Upsi{\Phi}=1$ and $\TtauInn(G_q)=\tfrac{\pi}{\log{q}}\ZZ$, if $n$ is even then $\Upsi{\Phi}=2$ and $\TtauInn(G_q)=\tfrac{\pi}{2\log{q}}\ZZ$;
\item type $C_n${\rm:} $G=\oon{Sp}(2n)$ {\rm(}$n\geq3${\rm)}, then $\Upsi{\Phi}=2$ and $\TtauInn(G_q)=\tfrac{\pi}{2\log{q}}\ZZ$;
\item type $D_n${\rm:} $G=\oon{Spin}(2n)$ {\rm(}$n\geq4${\rm)}, if $n\in4\NN+\{0,1\}$ then $\Upsi{\Phi}=2$ and $\TtauInn(G_q)=\tfrac{\pi}{2\log{q}}\ZZ$, if $n\in4\NN+\{2,3\}$ then $\Upsi{\Phi}=1$ and $\TtauInn(G_q)=\tfrac{\pi}{\log{q}}\ZZ$.
\end{itemize}
In exceptional cases we have the following possibilities:
\begin{itemize}
\item type $E_6${\rm:} then $\Upsi{\Phi}=2$ and $\TtauInn(G_q)=\tfrac{\pi}{2\log{q}}\ZZ$;
\item type $E_7${\rm:} then $\Upsi{\Phi}=1$ and $\TtauInn(G_q)=\tfrac{\pi}{\log{q}}\ZZ$;
\item type $E_8${\rm:} then $\Upsi{\Phi}=2$ and $\TtauInn(G_q)=\tfrac{\pi}{2\log{q}}\ZZ$;
\item type $F_4${\rm:} then $\Upsi{\Phi}=2$ and $\TtauInn(G_q)=\tfrac{\pi}{2\log{q}}\ZZ$;
\item type $G_2${\rm:} then $\Upsi{\Phi}=2$ and $\TtauInn(G_q)=\tfrac{\pi}{2\log{q}}\ZZ$.
\end{itemize}
\end{theorem}

\begin{proof}
Let us start with a general situation, i.e.~$G$ is a compact, simply connected, semisimple Lie group with root system $\Phi$. Let $A=\bigl[a_{i,j}\bigr]_{i,j=1}^{r}=\Bigl[2\tfrac{\la\alpha_i|\alpha_j\ra}{\la\alpha_i|\alpha_i\ra}\Bigr]_{i,j=1}^{r}$ be the Cartan matrix of $\Phi$ and $A^{-1}=\bigl[c_{i,j}\bigr]_{i,j=1}^{r}$ its inverse. Using the defining relation $\la\alpha_i|\varpi_j\ra=\delta_{i,j}\tfrac{\la\alpha_i|\alpha_i\ra}{2}$ we find that
\[
\alpha_i=\sum_{j=1}^{r}2\tfrac{\la\alpha_j|\alpha_i\ra}{\la\alpha_j|\alpha_j\ra}\varpi_j
=\sum_{j=1}^{r}a_{j,i}\varpi_j,
\]
and hence
\begin{equation}\label{eq30}
\sum_{i=1}^{r}c_{i,k}\alpha_i
=\sum_{i,j=1}^{r}c_{i,k}a_{j,i}\varpi_j
=\sum_{i,j=1}^{r}\bigl[A^{-1}\bigr]_{i,k}\bigl[A\bigr]_{j,i}\varpi_j
=\varpi_k.
\end{equation}
Take $1\leq{i}\leq{r}$. Using \eqref{eq30} and the fact that $\rho=\sum\limits_{j=1}^{r}\varpi_j$ we can express $\la2\rho|\varpi_i\ra$ as follows:
\begin{equation}\label{eq31}
\la2\rho|\varpi_i\ra=2\sum_{j=1}^{r}\la\varpi_j|\varpi_i\ra
=2\sum_{j,k=1}^{r}c_{k,j}\la\alpha_k|\varpi_i\ra
=\la\alpha_i|\alpha_i\ra\sum_{j=1}^{r}c_{i,j}.
\end{equation}
We see that in order to find the number $\Upsi{\Phi}$, we need a description of the inverse Cartan matrix $A^{-1}=\bigl[c_{i,j}\bigr]_{i,j=1}^{r}$. We take this description from \cite{WeiZou} (see also \cite[Table 2]{OnishchikVinberg}, note that we have to take the transpose of the matrices given in \cite{WeiZou}, as we use a different convention for Cartan matrices). We will now consider each of the cases listed in Theorem \ref{thm3} separately. Note that $r=n$.

\subsubsection*{Type $A_n$ {\rm(}$n\geq1${\rm)}}\hspace*{\fill}

We have $c_{i,j}=\min\{i,j\}-\tfrac{ij}{n+1}$ ($1\leq{i,j}\leq{n}$) and all roots are short (of length squared $\la\alpha_i|\alpha_i\ra=2$). Next, thanks to \eqref{eq31}
\begin{align*}
\la2\rho|\varpi_i\ra
&=2\biggl(\sum_{j=1}^{n}\min\{i,j\}-\sum_{j=1}^{n}\tfrac{ij}{n+1}\biggr)
=2\biggl(\sum_{j=1}^{i}j+\sum_{j=i+1}^{n}i-\tfrac{i}{n+1}\sum_{j=1}^{n}j\biggr)\\
&=2\bigl(\tfrac{i(i+1)}{2}+(n-i)i-\tfrac{i}{n+1}\tfrac{n(n+1)}{2}\bigr)
=(n+1-i)i
\end{align*}
for $1\leq{i}\leq{n}$. Consequently $\la2\rho|\varpi_i\ra\in\ZZ$ in general and $\la2\rho|\varpi_i\ra\in2\ZZ$ if $n$ is even.

Assume that $n=2k+1$ is odd. If $k=0$ then $n=1$ and $\la2\rho|\varpi_1\ra=1$. Otherwise, assuming $k>0$, we have
\[
\la2\rho|\varpi_n-2k\varpi_1+k\varpi_2\ra
=(n+1-n)n-2k(n+1-1)+2k(n+1-2)=n-2k=1
\]
and we see that $1\in\Upsi{\Phi}\ZZ$. Hence $\Upsi{\Phi}=1$.

If $n=2k$ is even then
\[
\la2\rho|2\varpi_1-\varpi_2\ra=2(n+1-1)-2(n+1-2)=2
\]
and $2\in\Upsi{\Phi}\ZZ$. Thus $\Upsi{\Phi}=2$.

\subsubsection*{Type $B_n$ {\rm(}$n\geq2${\rm)}}\hspace*{\fill}

We have $c_{i,j}=\min\{i,j\}$ ($1\leq{i}\leq{n}$, $1\leq{j}<n$) and $c_{i,n}=\tfrac{i}{2}$ ($1\leq{i}\leq{n}$). Squared lengths of simple roots are $\la\alpha_i|\alpha_i\ra=4$ ($1\leq{i}<n$) and $\la\alpha_n|\alpha_n\ra=2$ (notice the use of transposition and rescaling compared to \cite{WeiZou, Humphreys}). For $1\leq{i}<n$ we have
\[
\la2\rho|\varpi_i\ra
=4\biggl(\sum_{j=1}^{n-1}\min\{i,j\}+\tfrac{i}{2}\biggr)
=4\biggl(\sum_{j=1}^{i}j+\sum_{j=i+1}^{n-1}i+\tfrac{i}{2}\biggr)
=4\bigl(\tfrac{i(i+1)}{2}+(n-1-i)i+\tfrac{i}{2}\bigr)=2i(2n-i),
\]
while for $i=n$
\[
\la2\rho|\varpi_n\ra
=2\biggl(\sum_{j=1}^{n-1}\min\{n,j\}+\tfrac{n}{2}\biggr)
=2\bigl(\tfrac{(n-1)n}{2}+\tfrac{n}{2}\bigr)=n^2.
\]
Consequently, $\la2\rho|\varpi_i\ra\in\ZZ$ in general and $\la2\rho|\varpi_i\ra\in2\ZZ$ when $n$ is even.

Assume that $n=2k$ is even. If $k=1$ then $n=2$ and $\la2\rho|\varpi_1-\varpi_2\ra=2(4-1)-2^2=2$, hence $2\in\Upsi{\Phi}\ZZ$ and $\Upsi{\Phi}=2$. Assume now that $k\geq2$. We have
\begin{align*}
\la2\rho|(6k-4)(2\varpi_1-\varpi_2)-\varpi_3\ra
&=(6k-4)\bigl(4(2n-1)-4(2n-2)\bigr)-6(2n-3)\\
&=4(6k-4)-6(4k-3)=2,
\end{align*}
so $2\in\Upsi{\Phi}\ZZ$, and again $\Upsi{\Phi}=2$.

Consider now the case when $n=2k+1\geq{3}$ is odd. We have
\begin{align*}
\la2\rho|-(k^2+k)(2\varpi_1-\varpi_2)+\varpi_n\ra
&=-(k^2+k)\bigl(4(2n-1)-4(2n-2)\bigr)+n^2\\
&=-4(k^2+k)+(2k+1)^2=1,
\end{align*}
so that $1\in\Upsi{\Phi}\ZZ$ and thus $\Upsi{\Phi}=1$.

\subsubsection*{Type $C_n$ {\rm(}$n\geq3${\rm)}}\hspace*{\fill}

The Cartan matrix of $C_n$ is the transpose of the Cartan matrix of $B_n$. Hence $c_{i,j}=\min\{i,j\}$ ($1\leq{i}<n$, $1\leq{j}\leq{n}$) and $c_{n,j}=\tfrac{j}{2}$ ($1\leq{j}\leq{n}$). Furthermore, we can construct the root system $C_n$ as the dual of $B_n$ and after rescaling we have $\la\alpha_i|\alpha_i\ra=2$ ($1\leq{i}<n$) and $\la\alpha_n|\alpha_n\ra=4$. Using \eqref{eq31} we calculate
\[
\la2\rho|\varpi_i\ra
=2\sum_{j=1}^{n}\min\{i,j\}
=2\biggl(\sum_{j=1}^{i}j+\sum_{j=i+1}^{n}i\biggr)
=2\bigl(\tfrac{i(i+1)}{2}+(n-i)i\bigr)=(2n+1-i)i
\]
when $1\leq{i}<n$ and
\[
\la2\rho|\varpi_n\ra=4\sum_{j=1}^{n}\tfrac{j}{2}=n(n+1)
\]
when $i=n$. Thus we see that $\la2\rho|\varpi_i\ra\in2\ZZ$ for all $1\leq{i}\leq{n}$. Since
\[
\la2\rho|-\varpi_{n-1}+\varpi_n\ra
=-\bigl(2n+1-(n-1)\bigr)(n-1)+n(n+1)
=-(n+2)(n-1)+n(n+1)=2,
\]
we have $\Upsi{\Phi}\ZZ=2\ZZ$ and so $\Upsi{\Phi}=2$.

\subsubsection*{Type $D_n$ {\rm(}$n\geq4${\rm)}}\hspace*{\fill}

The Cartan matrix of $D_n$ is symmetric, hence so is its inverse. It is given by
\[
c_{i,j}=
\begin{cases}
i&1\leq{i}\leq{j}\leq{n-2}\\
\tfrac{i}{2}&1\leq{i}\leq{n-2},\:j\in\{n-1,n\}\\
\frac{n-2}{4}&i=n-1,j=n\\
\frac{n}{4}&i=j\in\{n-1,n\}
\end{cases}
\]
for $1\leq{i}\leq{j}\leq{n}$. Furthermore $\la\alpha_i|\alpha_i\ra=2$ for all $1\leq{i}\leq{n}$. Using equation \eqref{eq31}, we calculate for $1\leq{i}\leq{n-2}$
\begin{align*}
\la2\rho|\varpi_i\ra
&=2\sum_{j=1}^{n}c_{i,j}
=2\biggl(\sum_{j=1}^{i}c_{j,i}+\sum_{j=i+1}^{n-2}c_{i,j}+c_{i,n-1}+c_{i,n}\biggr)\\
&=2\biggl(\sum_{j=1}^{i}j+\sum_{j=i+1}^{n-2}i+i\biggr)
=2\bigl(\tfrac{i(i+1)}{2}+(n-2-i)i+i\bigr)=(2n-i-1)i.
\end{align*}
Next
\begin{align*}
\la2\rho|\varpi_{n-1}\ra
&=2\biggl(\sum_{j=1}^{n-2}c_{j,n-1}+c_{n-1,n-1}+c_{n-1,n}\biggr)\\
&=2\biggl(\sum_{j=1}^{n-2}\tfrac{j}{2}+\tfrac{n}{4}+\tfrac{n-2}{4}\biggr)
=\tfrac{(n-2)(n-1)+2n-2}{2}=\tfrac{n}{2}(n-1)
\end{align*}
and
\[
\la2\rho|\varpi_n\ra
=2\biggl(\sum_{j=1}^{n-2}c_{j,n}+c_{n-1,n}+c_{n,n}\biggr)
=2\biggl(\sum_{j=1}^{n-2}\tfrac{j}{2}+
\tfrac{n-2}{4}+\tfrac{n}{4}\biggr)=\tfrac{n}{2}(n-1).
\]
It follows that $\la2\rho|\varpi_i\ra\in\ZZ$ holds for all $n$.

Assume that $n=4k+a$ for some $k\in\NN$ and $a\in\{0,1\}$. Then $(2n-i-1)i=(8k+2a-i-1)i\in2\ZZ$ and also $\tfrac{n}{2}(n-1)=\tfrac{(4k+a)(4k+a-1)}{2}\in2\ZZ$. Furthermore
\[
\la2\rho|2\varpi_1-\varpi_2\ra=2(2n-2)-2(2n-3)=2
\]
which yields $2\in\Upsi{\Phi}\ZZ$ and $\Upsi{\Phi}=2$.

On the other hand, assume that $n=4k+a$ for $k\in\NN$ and $a\in\{2,3\}$. As before
\[
\la2\rho|2\varpi_1-\varpi_2\ra=2(2n-2)-2(2n-3)=2,
\]
and hence $2\in\Upsi{\Phi}\ZZ$. Furthermore,
\[
\la2\rho|\varpi_n\ra=\tfrac{n(n-1)}{2}=\tfrac{(4k+a)(4k+a-1)}{2}
=8k^2+2k(2a-1)+\tfrac{a(a-1)}{2}\in\Upsi{\Phi}\ZZ
\]
is odd, since $a\in\{2,3\}$. As $\Upsi{\Phi}\ZZ$ is a subgroup of $\ZZ$, we conclude that $\Upsi{\Phi}\ZZ=\ZZ$ and $\Upsi{\Phi}=1$.

\subsubsection*{Type $E_6$}\hspace*{\fill}

We have $\la\alpha_i|\alpha_i\ra=2$ for all $1\leq{i}\leq6$. Since the Cartan matrix is symmetric, we can take it directly from \cite{WeiZou} (these two properties apply also to $E_7$ and $E_8$). By direct calculation we find generators $\la2\rho|\varpi_i\ra=2\sum\limits_{j=1}^{6}c_{i,j}$ ($1\leq{i}\leq{6}$) of $\Upsi{\Phi}\ZZ$:
\[
\bigl(\la2\rho|\varpi_1\ra,\dotsc,\la2\rho|\varpi_6\ra\bigr)=(16,22,30,42,30,16).
\]
Consequently, $\Upsi{\Phi}\ZZ\subset2\ZZ$ and since
\[
\la2\rho|2\varpi_2-\varpi_4\ra=44-42=2,
\]
we conclude that $\Upsi{\Phi}\ZZ=2\ZZ$ giving $\Upsi{\Phi}=2$.

\subsubsection*{Type $E_7$}\hspace*{\fill}

We find that $\la2\rho|\varpi_1\ra=34$ and $\la2\rho|\varpi_2\ra=49$, therefore
\[
\la2\rho|-36\varpi_1+25\varpi_2\ra=-36\cdot34+25\cdot49=1.
\]
Since also $c_{i,j}\in\tfrac{1}{2}\ZZ$ for $1\leq{i,j}\leq{7}$, we obtain $\Upsi{\Phi}\ZZ=\ZZ$, i.e.~$\Upsi{\Phi}=1$.

\subsubsection*{Type $E_8$}\hspace*{\fill}

Since $\la\alpha_i|\alpha_i\ra=2$ and $c_{i,j}\in\ZZ$ for all $1\leq{i,j}\leq{8}$, we have $\la2\rho|\varpi_i\ra\in2\ZZ$ and consequently $\Upsi{\Phi}\ZZ\subset2\ZZ$. Furthermore $\la2\rho|\varpi_1\ra=92$ and $\la2\rho|\varpi_3\ra=182$. Hence
\[
\la2\rho|2\varpi_1-\varpi_3\ra=184-182=2
\]
and $\Upsi{\Phi}\ZZ=2\ZZ$, yielding $\Upsi{\Phi}=2$.

\subsubsection*{Type $F_4$}\hspace*{\fill}

After rescaling (if necessary) the squares of lengths of simple roots are given by $\la\alpha_1|\alpha_1\ra=\la\alpha_2|\alpha_2\ra=4$ and $\la\alpha_3|\alpha_3\ra=\la\alpha_4|\alpha_4\ra=2$. The Cartan matrix of $F_4$ is not symmetric, hence we need to take the transpose of the matrix given in \cite{WeiZou}, namely
\[
A^{-1}=
\begin{bmatrix}2&3&2&1\\
3&6&4&2\\
4&8&6&3\\
2&4&3&2
\end{bmatrix}
\]
It is clear that $\la2\rho|\varpi_i\ra\in2\ZZ$ for $1\leq{i}\leq4$. Using \eqref{eq31} we find that $\la2\rho|\varpi_3\ra=42$ and $\la2\rho|\varpi_4\ra=22$,
and consequently
\[
\la2\rho|-\varpi_3+2\varpi_4\ra=2.
\]
Thus $\Upsi{\Phi}\ZZ=2\ZZ$ and $\Upsi{\Phi}=2$.

\subsubsection*{Type $G_2$}\hspace*{\fill}

In this case $\la\alpha_1|\alpha_1\ra=2$ and $\la\alpha_2|\alpha_2\ra=6$. We again have to take the transpose of the (inverse) Cartan matrix, i.e.~$A^{-1}=
\bigl[\begin{smallmatrix}2&3\\1&2\end{smallmatrix}\bigr]$. We have $\la2\rho|\varpi_1\ra=10$, $\la2\rho|\varpi_2\ra=18$, hence
\[
\la2\rho|2\varpi_1-\varpi_2\ra=2.
\]
and $\Upsi{\Phi}\ZZ=2\ZZ$, i.e.~$\Upsi{\Phi}=2$.
\end{proof}

\subsection{The implementing unitaries}\label{sect:implemUnit}\hspace*{\fill}

Let $0<q<1$ and $G$ be a compact, simply connected, semisimple Lie group with the corresponding number $\Upsi{\Phi}\in\NN$ introduced in Definition \ref{def2} (defined by the root system $\Phi$ of $G$). Proposition \ref{prop4} and Theorem \ref{thm3} show that we must have $\Upsi{\Phi}\in\{1,2\}$. In this section we will be interested in the case $\Upsi{\Phi}=2$. Theorem \ref{thm2} shows in this case that $\TtauInn(G_q)$ is a proper subgroup of $\Ttau(G_q)$, more precisely $\tau^{G_q}_{\frac{\pi}{2\log{q}}}\in\oon{Aut}(\Linf(G_q))$ is inner (i.e.~implemented by some unitary in $\Linf(G_q)$) and non-trivial. The next proposition shows that it cannot however be implemented by a unitary in $\C(G_q)$.

\begin{proposition}\label{prop8}
Assume that $\Upsi{\Phi}=2$ and put $\bt=\tfrac{\pi}{2\log{q}}$, so that $\tau^{G_q}_{\bt}=\Ad(v)$ for a unitary element $v\in\Linf(G_q)$. Then $v\not\in\C(G_q)$. In particular, $\bigl.\tau^{G_q}_{\bt}\bigr|_{\C(G_q)}$ is not inner.
\end{proposition}

Before we prove Proposition \ref{prop8}, we need to recall a formula expressing operators $\pi_{\lambda,w_\circ}(b_\varpi)$ in simpler terms. As we mentioned in Section \ref{sect:prel}, one can calculate that in the case of $G_q=\oon{SU}_q(2)$ we have $b_N=(-\gamma)^N$ for all $N\in\Irr{\oon{SU}_q(2)}=\ZZ_+$. In order to keep track of the deformation parameter, we will denote this operator by $\sideset{_{\!\!q}}{_N}{\oon{\mathnormal{b}}}\in\Linf(\oon{SU}_q(2))$.

\begin{lemma}[{\cite{DeCommer}}]
Take $\varpi\in\bP^+$ and write $w_\circ=s_{i_1}\dotsm{s_{i_{\ell(w_\circ)}}}$. Then
\begin{equation}\label{eq33}
\pi_{\lambda,w_\circ}(b_\varpi)
=\la\varpi,\lambda\ra\,
\theta_{q_{i_1}}\bigl(
\sideset
{_{\!q_{i_1}\!\!}}
{_{\la\beta_1^{\vee}|\varpi\ra}}
{\oon{\mathnormal{b}}}
\bigr)\tens\dotsm\tens
\theta_{q_{i_{\ell(w_\circ)}}}\bigl(
\sideset
{_{\!q_{i_{\ell(w_\circ)}}\!\!}}
{_{\la\beta_{\ell(w_\circ)}^{\vee}|\varpi\ra}}
{\oon{\mathnormal{b}}}
\bigr)\in\B\bigl(\ell^2(\ZZ_+)^{\tens\ell(w_\circ)}\bigr)
\end{equation}
for all $\lambda\in{T}$, where
\begin{align*}
\beta_1^\vee&=\alpha_{i_1}^{\vee},\\
\beta_2^\vee&=s_{i_1}\alpha_{i_2}^{\vee},\\
&\;\;\vdots\\
\beta_{\ell(w_\circ)}^{\vee}&=s_{i_1}\dotsc{s_{i_{\ell(w_\circ)-1}}}\alpha_{i_{\ell(w_\circ)}}^{\vee}
\end{align*}
are elements of the root system $\Phi$.
\end{lemma}

In fact, the elements $\beta_k=\tfrac{\la\beta_k|\beta_k\ra}{2}\beta_k^{\vee}$, ($1\leq{k}\leq\ell(w_\circ)$) are the positive roots (see \cite[Proposition 20.10]{Bump}). The formula for $\pi_{w_\circ}(b_\varpi)$ is given on \cite[Page 203]{DeCommer} and then \eqref{eq17} gives an easy modification to the general case. One way of deriving equation \eqref{eq33} is to first establish $S(b_\varpi)=(-1)^{\sum_{k=1}^{r}a_k}q^{2\la\rho|\varpi\ra}b_{-w_\circ\varpi}$, where $\varpi=\sum\limits_{k=1}^{r}a_k\varpi_k$. Next one can run an argument analogous to the one leading to \cite[Proposition 4.3]{ReshetikhinYakimov}, but for the elements $S(b_\varpi)^*=U^{\varpi}(\eta_{w_\circ\varpi},\xi_\varpi)$. The final step giving \eqref{eq33} is done by tweaking the obtained expression using the unitary antipode and the fact that $w_\circ^{-1}=w_\circ$.

\begin{proof}[Proof of Proposition \ref{prop8}]
Since $\la\varpi,\blambda_{2\bt}\ra=q^{\la2\rho|\varpi\ra2\ii\bt}=1$ for $\varpi\in\bP$, we see that $\blambda_{2\bt}$ is the trivial element in $T$ and consequently Theorem \ref{thm1} gives
\[
\tau^{G_q}_{\bt}=\Ad\Biggl(\cQ_{\LL}^*\Biggl({\Int_T}^\oplus\pi_{w_\circ}
\bigl(|b_\rho|\bigr)^{-2\ii\bt}\tens\I_{\overline{\cH_{\lambda}}}\dd{\lambda}\Biggr)\cQ_{\LL}\Biggr).
\]
It follows that we can find a unitary $x\in\cZ(\Linf(G_q))$ for which
\begin{equation}\label{eq34}
v=x\:\!\cQ_{\LL}^*
\Biggl({\Int_T}^\oplus\pi_{w_\circ}\bigl(|b_\rho|)^{-2\ii\bt}\tens\I_{\overline{\cH_{\lambda}}}\dd{\lambda}\Biggr)\cQ_{\LL}
=\cQ_{\LL}^*
\Biggl({\Int_T}^{\oplus}x_\lambda\pi_{w_\circ}\bigl(|b_\rho|\bigr)^{-2\ii\bt}\tens\I_{\overline{\cH_{\lambda}}}\dd{\lambda}\Biggr)\cQ_{\LL},
\end{equation}
where $\cQ_{\LL}x\cQ_{\LL}^*={\Int_T}^{\oplus}x_\lambda\I_{\HS(\cH_\lambda)}\dd{\lambda}$. Define a bounded linear map $\Psi\colon\Linf(G_q)\to\Linf(T\times\ZZ_+^{\ell(w_\circ)})$ via
\[
\Psi(a)\bigl(\lambda,k_1,\dotsc,k_{\ell(w_\circ)}\bigr)=
\is{e_{k_1,\dotsc,k_{\ell(w_\circ)}}}
{a_\lambda{e_{k_1,\dotsc,k_{\ell(w_\circ)}}}},\qquad\lambda\in{T},\:k_1,\dotsc,k_{\ell(w_\circ)}\in\ZZ_+,
\]
where $\cQ_{\LL}a\cQ_{\LL}^*={\Int_T}^{\oplus}a_\lambda\tens\I_{\overline{\cH_\lambda}}\dd{\lambda}$ and $e_{k_1,\dotsc,k_{\ell(w_\circ)}}=e_{k_1}\tens\dotsm\tens{e_{k_{\ell(w_\circ)}}}$ ($k_1,\dotsc,k_{\ell(w_\circ)}\in\ZZ_+$) is the standard orthonormal basis in $\cH_\lambda=\ell^2(\ZZ_+)^{\tens\ell(w_\circ)}$. We claim that
\begin{equation}\label{eq35}
\Psi(v)\not\in\Linf(T)\tens\Bigl(\CC\I+\mathrm{c}_{0}\bigl(\ZZ_+^{\times\ell(w_\circ)}\bigr)\Bigr).
\end{equation}
Using \eqref{eq34} and \eqref{eq33} we can calculate
\begin{align*}
\Psi(v)\bigl(\lambda,k_1&,\dotsc,k_{\ell(w_\circ)}\bigr)
=x_\lambda\is{e_{k_1,\dotsc,k_{\ell(w_\circ)}}}
{\pi_{w_\circ}\bigl(|b_\rho|\bigr)^{-2\ii\bt}
e_{k_1,\dotsc,k_{\ell(w_\circ)}}}\\
&=x_\lambda
\is{
e_{k_1}
}{
\theta_{q_{i_1}}\bigl(|
\sideset
{_{\!q_{i_1}\!\!}}
{_{\la\beta_1^{\vee}|\rho\ra}}
{\oon{\mathnormal{b}}}|
\bigr)^{-2\ii\bt}e_{k_1}
}\dotsm\is{e_{k_{\ell(w_\circ)}}}
{\theta_{q_{i_{\ell(w_\circ)}}}\bigl(|
\sideset
{_{\!q_{i_{\ell(w_\circ)}}\!\!}}
{_{\la\beta_{\ell(w_\circ)}^{\vee}|\rho\ra}}
{\oon{\mathnormal{b}}}
|\bigr)^{-2\ii\bt}e_{k_{\ell(w_\circ)}}}\\
&=x_\lambda
q_{i_1}^{-2\ii{k_1}\la\beta_1^{\vee}|\rho\ra\bt}\dotsm
q_{i_{\ell(w_\circ)}}^{-2\ii{k_{\ell(w_\circ)}}\la\beta_{\ell(w_\circ)}^{\vee}|\rho\ra\bt}
\end{align*}
for almost all $\lambda\in{T}$ and $k_1,\dotsc,k_{\ell(w_\circ)}\in\ZZ_+$. Since $q_i=q^{\frac{\la\alpha_i|\alpha_i\ra}{2}}$ and $\beta_i^{\vee}=\tfrac{2}{\la\alpha_i|\alpha_i\ra}\beta_i$ for $1\leq{i}\leq{r}$, we obtain
\begin{align*}
\Psi(v)\bigl(\lambda,k_1,\dotsc,k_{\ell(w_\circ)}\bigr)
&=x_\lambda
q^{-2\ii{k_1}\la\beta_1|\rho\ra\bt}\dotsm
q^{-2\ii{k_{\ell(w_\circ)}}\la\beta_{\ell(w_\circ)}|\rho\ra\bt}\\
&=x_\lambda
\ee^{-\pi\ii{k_1}\la\beta_1|\rho\ra}\dotsm
\ee^{-\pi\ii{k_{\ell(w_\circ)}}\la\beta_{\ell(w_\circ)}|\rho\ra}.
\end{align*}
Assume that \eqref{eq35} does not hold. Then writing $x_\bullet$ for the function $T\ni\lambda\mapsto{x_\lambda}\in\CC$, we also have $(x_\bullet^*\tens\I)\Psi(v)\in\Linf(T)\tens(\CC\I+\mathrm{c}_0(\ZZ_+^{\ell(w_\circ)}))$ and integrating over the first variable
\begin{equation}\label{eq36}
\Biggl(\Int_T\cdot\:\!\dd{\lambda}\tens\id\Biggr)
\bigl((x_\bullet^*\tens\I)\Psi(v)\bigr)
\in\CC\I+\mathrm{c}_0\bigl(\ZZ_+^{\ell(w_\circ)}\bigr).
\end{equation}
In Lemma \ref{lemma6} we proved that $\bigl\{\la2\rho|\alpha\ra\,\bigr|\bigl.\,\alpha\in\bQ\bigr\}=2\ZZ$. Now, as $\bigl\{\beta_1,\dotsc,\beta_{\ell(w_\circ)}\bigr\}$ is the set of positive roots, we can find $n_1,\dotsc,n_{\ell(w_\circ)}\in\ZZ$ such that $\la{n_1}\beta_1+\dotsm+n_{\ell(w_\circ)}\beta_{\ell(w_\circ)}|\rho\ra=1$. Finally for $p\in\NN$
\begin{align*}
\Biggl(\Int_T\cdot\:\!\dd{\lambda}\tens\id\Biggr)
\bigl((x_\bullet^*\tens\I)\Psi(v)\bigr)
\bigl(p|n_1|,\dotsc,p|n_{\ell(w_\circ)}|\bigr)
&=\ee^{-\pi\ii{p}|n_1|\la\beta_1|\rho\ra}\dotsm
\ee^{-\pi\ii{p}|n_{\ell(w_\circ)}|\la\beta_{\ell(w_\circ)}|\rho\ra}\\
&=\ee^{-\pi\ii{p}\la{n_1}\beta_1+\dotsm+n_{\ell(w_\circ)}\beta_{\ell(w_\circ)}|\rho\ra}\\&=\ee^{-\pi\ii{p}}
=\begin{cases}
1&p\in2\NN\\
-1&p\in2\ZZ_++1
\end{cases}.
\end{align*}
This contradicts \eqref{eq36} and hence proves validity of \eqref{eq35}.

On the other hand, we will now show that if $a\in\C(G_q)$, then $\Psi(a)\in\Linf(T)\tens(\CC\I+\mathrm{c}_0(\ZZ_+^{\ell(w_\circ)}))$. Because of \eqref{eq35}, this will imply $v\not\in\C(G_q)$.

Since $\Psi$ is continuous and $\Linf(T)\tens(\CC\I+\mathrm{c}_0(\ZZ_+^{\ell(w_\circ)}))$ is norm-closed in $\Linf(T\times\ZZ_+^{\ell(w_\circ)})$, it is enough to consider $a\in\Pol(G_q)$. Then
\[
a_\lambda=\pi_{\lambda,w_\circ}(a)
=(\pi_\lambda\tens\theta_{q_{i_1}}\comp\rho_{i_1}\tens
\dotsm\tens\theta_{q_{i_{\ell(w_\circ)}}}\comp
\rho_{i_{\ell(w_\circ)}})\Delta_{G_q}^{(\ell(w_\circ))}(a)
\]
for $\lambda\in{T}$ (see \cite[Lemma 3]{KrajczokSoltanDyskArxiv}, 
for the the notation cf.~Section \ref{sect:prel}). Since also $\rho_i(\Pol(G_q))\subset\Pol(\oon{SU}_{q_i}(2))$, it is enough to show
\[
\bigl(\is{e_k}{\theta_q(a)e_k}\bigr)_{k\in\ZZ_+}\in\CC\I+\mathrm{c}_0(\ZZ_+)
\]
for any $a\in\Pol(\oon{SU}_q(2))$. By \cite[Theorem 1.2]{su2} this claim follows from
\[
\is{e_k}{\theta_q\bigl(\alpha^r\gamma^s{\gamma^*}^t\bigr)e_k}=
\delta_{r,0}q^{k(s+t)}
=\is{e_k}{\theta_q\bigl({\alpha^*}^r\gamma^s{\gamma^*}^t\bigr)e_k}
\]
for $k,r,s,t\in\ZZ_+$.
\end{proof}

\begin{remark}
The results established above show that we can find a compact quantum group $G_q$ and time $\bt=\tfrac{\pi}{2\log{q}}$ such that the scaling automorphism $\tau^{G_q}_\bt$ is non-trivial, inner and not implemented by any unitary in $\C(G_q)$. In particular it is not implemented by a group-like unitary. Using products and bicrossed products one can construct a compact quantum group where these properties hold at a given finite sequence of times, which are independent over $\QQ$ (c.f.~constructions in \cite{faktory}).
\end{remark}

\section{Other examples}\label{sect:otherEx}

\subsection{The quantum \texorpdfstring{$\oon{E}(2)$}{E(2)} group}\label{sect:E2}\hspace*{\fill}

The quantum $\oon{E}(2)$ (also referred to as the quantum group $\oon{E}_q(2)$) was originally defined in \cite{slw_E2}. The crucial development of determining the Haar measure of $\GG$ was accomplished in \cite{baaj92}. Finally the most comprehensive text dealing with this quantum group is the PhD thesis \cite{Jacobs}. Throughout this section we will write $\GG$ for $\oon{E}_q(2)$.

In order to describe this quantum group we first fix a real number $q\in\left]0,1\right[$. Next we consider two operators $v$ and $n$ on $\ell^2(\ZZ)\tens\ell^2(\ZZ)$ defined as
\[
v=s^*\tens\I,\qquad{n}=Q\tens{s},
\]
where $s$ is the shift defined on the standard basis of $\ell^2(\ZZ)$ as $e_k\mapsto{e_{k+1}}$ ($k\in\ZZ$) and $Q$ is the positive self-adjoint operator such that $Qe_k=q^ke_k$ for all $k$. Noting that $n$ is normal with $\oon{Sp}(n)=\bigl\{z\in\CC\,\bigr|\bigl.\,|z|\in{q}^\ZZ\bigr\}\cup\{0\}$ we define $\C_0(\GG)$ as the closure in $\B(\ell^2(\ZZ)\tens\ell^2(\ZZ))$ of the set of sums of the form
\[
\sum_{k\in\ZZ}v^kf_k(n)
\]
where $f_k\in\C_0(\oon{Sp}(n))$ and $f_k=0$ for almost all $k\in\ZZ$. One easily shows that $v\in\oon{M}(\C_0(\GG))$ and that $n$ is affiliated with $\C_0(\GG)$ (\cite{slw_unbo}). Moreover $\C_0(\GG)$ is generated by $v$ and $n$ in the sense of \cite{slw_gen}.

The comultiplication $\Delta$ describing the quantum group structure of $\GG$ is defined in \cite[Theorem 1.2]{slw_E2}. It is a \emph{morphism} of \cst-algebras from $\C_0(\GG)$ to $\C_0(\GG)\tens\C_0(\GG)$ (see \cite{slw_unbo}) determined by its values on $v$ and $n$ which are
\[
\Delta(v)=v\tens{v},\qquad\Delta(n)=v\tens{n}\;\dot{+}\;n\tens{v^*}
\]
with $\dot{+}$ denoting the closure of the (closable) operator $v\tens{n}+n\tens{v^*}$.

The Haar measure $\bh$ of $\oon{E}_q(2)$ (which is both left and right invariant) was first described in \cite{baaj92}. The GNS representation corresponding to $\bh$ acts on the Hilbert space $\ell^2(\ZZ)\tens\ell^2(\ZZ)\tens\ell^2(\ZZ)$ and maps $v$ to $\I\tens{s^*}\tens{s^*}$ and $n$ to $s\tens{Q}\tens\I$. Since the unitary $X\colon{e_k}\tens{e_l}\mapsto{e_k}\tens{e_{l-k}}$ conjugates $s^*\tens{s^*}$ onto $s^*\tens\I$ while $X(Q\tens\I)X^*=Q\tens\I$, we find that $\Linf(\GG)$ is isomorphic to $\Linf(\TT)\vtens\mathsf{N}$, where $\mathsf{N}$ is the von Neumann algebra on $\ell^2(\ZZ)$ generated by $s$ and bounded measurable functions of $Q$. Since the latter is easily seen to contain all compact operators on $\ell^2(\ZZ)$, we obtain
\begin{equation}\label{eq:LE2iso}
\Linf(\GG)\cong\Linf(\TT)\vtens\B\bigl(\ell^2(\ZZ)\bigr)
\end{equation}
Baaj also proves $\modOp[\bh]=\I\tens\I\tens{Q^2}$. These two facts immediately show that
\[
\TsigInn(\GG)=\RR,\qquad
\Tsig(\GG)=\tfrac{\pi}{\log{q}}\ZZ,
\]
and consequently also $\TsigAInn(\GG)=\RR$.

The scaling group of $\GG$ is described in detail in \cite{Jacobs}. We have $\tau^{\GG}_t(v)=v$ for all $t\in\RR$ while $\tau^{\GG}_t(n)=q^{-2\ii{t}}n$. It follows from the uniqueness of the polar decomposition that $\tau^{\GG}_t$ fixes $|n|$ and hence multiplies the phase factor of $n$ by $q^{-2\ii{t}}$. In the isomorphism \eqref{eq:LE2iso} this phase factor is easily seen to be the unitary generator of the center $\Linf(\TT)\tens\I$ of $\Linf(\GG)\cong\Linf(\TT)\vtens\B(\ell^2(\ZZ))$ which implies that the non-trivial scaling automorphisms are not inner nor approximately inner, i.e.
\[
\TtauInn(\GG)=\TtauAInn(\GG)=\Ttau(\GG)=\tfrac{\pi}{\log{q}}\ZZ.
\]
It is known that $\GG$ is unimodular and hence $\Mod(\GG)=\RR$.

The values of invariants for the dual quantum group $\hh{\oon{E}_q(2)}$ can be calculated in a similar fashion. However we can use the many results contained in \cite{Jacobs} to make some shortcuts. Denoting $\hh{\oon{E}_q(2)}$ by $\hh{\GG}$ we first of all have $\Ttau(\hh{\GG})=\Ttau(\GG)=\tfrac{\pi}{\log{q}}\ZZ$. Next we note that by \cite[Lemma 2.7.31]{Jacobs} the modular group for the left Haar measure of $\hh{\GG}$ coincides with $\tau^{\hh{\GG}}$ (this also follows from a more general fact, see \cite[Remarks before Lemma 6.3]{modular}). Hence $\Ttau(\hh{\GG})=\Tsig(\hh{\GG})=\tfrac{\pi}{\log{q}}\ZZ$.

Furthermore \cite[Proposition 2.8.4]{Jacobs} gives a direct description of the modular element of $\hh{\GG}$, which immediately implies $\Mod(\hh{\GG})=\tfrac{\pi}{\log{q}}\ZZ$.

The invariants left to determine are $\TtauInn(\hh{\GG})$, $\TtauAInn(\hh{\GG})$, $\TsigInn(\hh{\GG})$, and $\TsigAInn(\hh{\GG})$. For this let us recall the presentation of the \cst-algebra $\C_0(\hh{\GG})$ proposed in \cite{Jacobs} which is slightly different from the one originally introduced in \cite{slw_E2}. Namely in the GNS representation of the left Haar measure, the \cst-algebra $\C_0(\hh{\GG})$ is the \cst-algebra of operators on $\ell^2(\ZZ)\tens\ell^2(\ZZ)\tens\ell^2(\ZZ)$ generated by $a,a^{-1}$ and $b$, where $a=Q^{\mhalf}\tens{Q}\tens\I$ and $b=Q^{\half}\tens{s}\tens\I$ (cf.~\cite[Proposition 2.7.33]{Jacobs} or \cite[Proposition 5.2]{baaj92}). Furthermore, as shown in \cite[Proposition 2.6.50]{Jacobs}, we have
\[
\tau^{\hh{\GG}}_t(a)=a\quad\text{and}\quad\tau^{\hh{\GG}}_t(b)=q^{-2\ii{t}}b
\]
for all $t$ which clearly shows that all these automorphisms are inner, implemented e.g.~by $\{a^{-2\ii{t}}\}_{t\in\RR}$. It follows that
\[
\TtauInn(\hh{\GG})=\TtauAInn(\hh{\GG})=\TsigInn(\hh{\GG})
=\TsigAInn(\hh{\GG})=\RR.
\]

\subsection{The quantum ``\texorpdfstring{$az+b$}{az+b}'' group}\hspace*{\fill}

There are three classes of quantum deformations of the classical ``$az+b$'' group described in \cite{azb,nazb}. They are defined by the choice of the deformation parameter
\begin{enumerate}
\item\label{case:1} $q=\ee^{\frac{2\pi\ii}{N}}$ with $N\in\{6,8,\dotsc\}$,
\item\label{case:2} $q\in\left]0,1\right[$,
\item\label{case:3} $q=\ee^{\frac{1}{\rho}}$ with $\oon{Re}{\rho}<0$, $\oon{Im}{\rho}=\tfrac{N}{2\pi}$ with $N\in\{\pm2,\pm4,\dotsc\}$.
\end{enumerate}
Reserving now the symbol $\GG$ for the quantum ``$az+b$'' group we have that the \cst-algebra $\C_0(\GG)$ is generated in the sense of \cite{slw_gen} by two normal unbounded elements $a$ and $b$ affiliated with $\C_0(\GG)$ such that
\[
ab=q^2ba,\qquad
ab^*=b^*a
\]
and $\oon{Sp}(a),\oon{Sp}(b)\subset\overline{\Gamma_q}$, where $\Gamma_q$ is a multiplicative subgroup of $\CC\setminus\{0\}$ generated by $q$ and $\bigl\{q^{\ii{t}}\,\bigr|\bigl.t\in\RR\bigr\}$ (for details see \cite{azb,nazb}). Note for future reference that in case \eqref{case:2}, we have
\[
\Gamma_q=\bigl\{z\in\CC\,\bigr|\bigl.\,|z|\in{q^\ZZ}\bigr\}.
\]
The comultiplication on $\C_0(\GG)$ is determined by its values on the two generators:
\[
\Delta(a)=a\tens{a},\quad\Delta(b)=a\tens{b}\;\dot{+}\;b\tens\I,
\]
where, as in Section \ref{sect:E2}, the symbol $\dot{+}$ indicates the closure of the sum of unbounded elements. The Haar measure of $\GG$ was determined in \cite{VanDaeleTheHaarMeasure} in case \eqref{case:1} and in \cite{SLWHaarMeasure} in all three cases.

It is known that in all cases the locally compact quantum group $\GG$ is anti-isomorphic to its dual $\hh{\GG}$ (i.e.~isomorphic to $\hh{\GG}^{\text{\tiny{\rm{op}}}}$). It follows that
\[
T^\tau_{\bullet\!}(\GG)=T^\tau_{\bullet\!}(\hh{\GG})
\]
for all choices of $\bullet$. Moreover, it was shown in \cite[Section 7]{modular} that $\Linf(\GG)$ (and hence also $\Linf(\hh{\GG})$) is a factor of type $\mathrm{I}_\infty$, so
\[
\TsigInn(\GG)=\TsigAInn(\GG)=\TsigInn(\hh{\GG})=\TsigAInn(\hh{\GG})=\RR.
\]

It is known that the scaling automorphisms of $\GG$ are all inner and implemented by $\{|a|^{\ii{t}}\}_{t\in\RR}$ (this is explicitly proved in \cite{nazb} for case \eqref{case:3} and in \cite{SLWHaarMeasure} for cases \eqref{case:1} and \eqref{case:2}). Hence
\[
\TtauInn(\GG)=\TtauAInn(\GG)=\TtauInn(\hh{\GG})=\TtauAInn(\hh{\GG})=\RR,
\]
which by \cite[Corollary 2.9.32]{Sakai} also follows from the fact that $\Linf(\GG)$ is a type $\mathrm{I}_\infty$ factor. Furthermore, we have $\tau_t^\GG(a)=a$ and $\tau_t^\GG(b)=q^{2\ii{t}}b$ for all $t$ which immediately implies that $\Ttau(\GG)=\{0\}$ in cases \eqref{case:1} and \eqref{case:3}, while
\[
\Ttau(\GG)=\tfrac{\pi}{\log{q}}\ZZ
\]
in case \eqref{case:2}. Thus, by \eqref{eq20} and \eqref{eq23} in Proposition \ref{prop5}, we obtain
\[
\Tsig(\GG)=\{0\}\quad\text{and}\quad
\Mod(\GG)=\{0\}
\]
in cases \eqref{case:1} and \eqref{case:3}. The same holds for $\hh{\GG}$ (cf.~Remark \ref{rem:przedProp5}\eqref{rem:przedProp5-2}).

The remaining invariants are $\Tsig(\GG)$ and $\Mod(\GG)$ in case \eqref{case:2} (which are the same for $\hh{\GG}$). For the first one we note that by the discussion preceding \cite[Proposition 7.7]{modular} we have $\sigma^\psi_t(x)=|b|^{2\ii{t}}x|b|^{-2\ii{t}}$ for all $x\in\Linf(\GG)$ and it follows that $\Tsig(\GG)=\{0\}$ in cases \eqref{case:1} and \eqref{case:3}, while $\Tsig(\GG)=\tfrac{\pi}{\log{q}}\ZZ$ in case \eqref{case:2}.

Now that we know that $\Tsig(\GG)=\Ttau(\GG)$ (and that $\GG$ is anti-self-dual), equation \eqref{eq20} implies that $\Tsig(\GG)\subset\Mod(\GG)$, while \eqref{eq23} gives $\Mod(\GG)\subset\tfrac{1}{2}\Tsig(\GG)$. In other words
\[
\tfrac{\pi}{\log{q}}\ZZ\subset\Mod(\GG)\subset\tfrac{\pi}{2\log{q}}\ZZ
\]
(and consequently $\Mod(\GG)$ is either of the two groups above). To see that it is $\tfrac{\pi}{\log{q}}\ZZ$ let us take advantage of the many results in \cite{modular}.

More precisely by \cite[Proposition 5.5]{modular} and the explicit determination of the Plancherel objects for $\GG$ in \cite[Section 7]{modular} we have that $t\in\Mod(\GG)$ if and only if
\begin{equation}\label{eq:delit1}
|qa^{-1}b|^{-2\ii{t}}|b|^{2\ii{t}}=\I
\end{equation}
(for this we also need the fact that the scaling constant is trivial in case \eqref{case:2}, cf.~\cite{SLWHaarMeasure,modular}). In order to determine for which $t$ \eqref{eq:delit1} holds we use the fact that the GNS representation of $\Linf(\GG)$ is equivalent to the representation on $\Ltwo(\Gamma_q)$ with (the images of) $a$ and $b$ determined by
\[
(b\psi)(\gamma)=\gamma\psi(\gamma),\qquad\psi\in\Ltwo(\Gamma_q),\:\gamma\in\Gamma_q
\]
and denoting the polar decomposition of $a$ by $(\oon{Phase}a)|a|$ we have
\[
\left.
\begin{array}{r@{\:=\:}l}
\bigl((\oon{Phase}a)\psi\bigr)(\gamma)&\psi(q\gamma)\\[3pt]
\bigl(|a|^{\ii{t}}\psi\bigr)(\gamma)&\psi(q^{\ii{t}}\gamma)
\end{array}\right.,
\qquad\psi\in\Ltwo(\Gamma_q),\:\gamma\in\Gamma_q,\:t\in\RR.
\]
Now let $\{g_{k,l}\}_{k,l\in\ZZ}$ be the orthonormal basis of $\Ltwo(\Gamma_q)$ defined by
\[
g_{k,l}(\gamma)=\begin{cases}
\left(\frac{\gamma}{|\gamma|}\right)^k&|\gamma|=q^l\\
0&|\gamma|\neq{q^l}
\end{cases}.
\]
Then one can check by direct computation that $|qa^{-1}b|g_{k,l}=q^{l-k}g_{k,l}$ and $|b|g_{k,l}=q^lg_{k,l}$ for all $k,l$ and consequently
\[
|qa^{-1}b|^{-2\ii{t}}|b|^{2\ii{t}}g_{k,l}
=q^{-2\ii{t}(l-k)+2\ii{t}l}g_{k,l}=q^{2\ii{t}k}g_{k,l},\qquad{k,l}\in\ZZ,
\]
so that \eqref{eq:delit1} holds if and only if $t\in\tfrac{\pi}{\log{q}}\ZZ$.

\subsection{The quantum group \texorpdfstring{$\oon{U}_F^+$}{UF+}}\label{sect:UFplus}\hspace*{\fill}

Let $F\in\oon{GL}(N,\CC)$ ($N\geq{2}$) be an arbitrary invertible matrix and $\oon{U}_F^+$ the corresponding free unitary quantum group whose Haar measure $\bh_{\oon{U}_F^+}$ and scaling automorphisms $\tau^{\oon{U}_F^+}_t$ will throughout this section be denoted by $\bh$ and $\tau_t$. We will also see $\Linf(\oon{U}_F^+)$ as a subspace of $\Ltwo(\oon{U}_F^+)$ via the GNS map of $\bh$. As in \cite{banica} (and Section \ref{sect:UqU}) we identify $\Irr{\oon{U}_F^+}$ with the free product monoid $\ZZ_+\star\ZZ_+$ with generators $\alpha$ and $\beta$ corresponding to the defining representation and its conjugate respectively.

We want to calculate $T^\tau_\bullet$ invariants of $\oon{U}_F^+$. First, we establish an analogue of \cite[Proposition 32]{DeCommerFreslonYamashita} for the scaling group. We follow closely \cite[Appendix]{DeCommerFreslonYamashita}.

As in Section \ref{sect:UqU}, for $x\in\ZZ_+\star\ZZ_+$ choose a representative $u^x$ which acts on $\cH_x$ and define a state $\omega_x$ on $\B(\cH_x)$ by
\[
\omega_x(T)=\tfrac{1}{\qdim{x}}\Tr(T\uprho_x^{-1}),\qquad{T}\in\B(\cH_x)
\]
which satisfies \cite[Proof of Theorem 1.4.3]{NeshveyevTuset}
\[
(\bh\tens\id)\bigl({u^x}^*(\I\tens{T})u^x\bigr)
=\omega_x(T)\I,\qquad{T}\in\B(\cH_x).
\]
Next, let $\cK_x$ be the GNS Hilbert space defined by $\omega_x$, i.e.~$\cK_x$ is equal to $\B(\cH_x)$ as a vector space, with inner product $\is{T}{T'}=\omega_x(T^*T')$ ($T,T'\in\cK_x$). Since
\begin{equation}\label{eq28-drugie}
(\sigma^{\bh}_t\tens\id)u^x
=(\I\tens\uprho_x^{\ii{t}})u^x(\I\tens\uprho_x^{\ii{t}}),\qquad{t}\in\RR,\:x\in\ZZ_+\star\ZZ_+
\end{equation}
(\cite[Page 30]{NeshveyevTuset}), we obtain
\begin{equation}\label{eq1-drugie}
(\bh\tens\omega_x)\bigl(u^x(a\tens\I){u^x}^*\bigr)=\bh(a),\qquad{a}\in\Linf(\oon{U}_F^+).
\end{equation}
Indeed,
\begin{align*}
(\bh\tens\omega_x)\bigl(u^x(a\tens\I){u^x}^*\bigr)
&=\Bigl(\bh\tens\tfrac{\Tr}{\qdim{x}}\Bigr)
\bigl(u^x(a\tens\I){u^x}^*(\I\tens\uprho_x^{-1})\bigr)\\
&=\Bigr(\bh\tens\tfrac{\Tr}{\qdim{x}}\Bigr)
\bigl((a\tens\I){u^x}^*(\I\tens\uprho_x^{-1})(\sigma^{\bh}_{-\ii}\tens\id)(u^x)\bigr)\\
&=\Bigl(\bh\tens\tfrac{\Tr}{\qdim{x}}\Bigr)
\bigl((a\tens\I){u^x}^*u^x(\I\tens\uprho_x)\bigr)
=\bh(a).
\end{align*}
For $x\in\ZZ_+\star\ZZ_+$ and $t\in\RR$ define $V_x^{\tau,t}\colon\Ltwo(\oon{U}_F^+)\to\Ltwo(\oon{U}_F^+)\tens\cK_x$ as the isometric extension of
\[
\Linf(\oon{U}_F^+)\ni{a}\longmapsto
(\tau_t\tens\id)(u^x)(a\tens\I){u^x}^*.
\]
To see that this map is isometric take $a,b\in\Linf(\oon{U}_F^+)\subset\Ltwo(\oon{U}_F^+)$. Then using \eqref{eq1-drugie} we obtain
\begin{align*}
\is{V^{\tau,t}_x(a)}{V^{\tau,t}_x(b)}
&=(\bh\tens\omega_x)\bigl(u^x(a^*\tens\I)(\tau_t\tens\id)({u^x}^*u^x)
(b\tens\I){u^x}^*\bigr)\\
&=(\bh\tens\omega_x)\bigl(u^x(a^*b\tens\I){u^x}^*\bigr)=\bh(a^*b)=\is{a}{b}.
\end{align*}

The following proposition is (together with its proof) an analogue of \cite[Proposition 32]{DeCommerFreslonYamashita} with the modular group replaced by the scaling group.

\begin{proposition}\label{prop:Vaes33}
For $a\in\Linf(\oon{U}_F^+)$ and $t\in\RR$ we have
\[
\bigl\|a-\bh(a)\I\bigr\|_2\leq{14}\max\Bigl\{
\bigl\|a\tens\I-V^{\tau,t}_{\alpha\beta}(a)\bigr\|,
\bigl\|a\tens\I-V^{\tau,t}_{\alpha^2\beta}(a)\bigr\|,
\bigl\|a\tens\I-V^{\tau,t}_{\beta\alpha}(a)\bigr\|\Bigr\}.
\]
\end{proposition}

Before continuing let us note that the operator $\uprho_\alpha$ can be expressed in terms of $F$ (with an appropriate identification of $\cH_\beta$ with $\CC^N$ on which $F$ acts). Indeed, if $\Tr(F^*F)=\Tr\bigl((F^*F)^{-1}\bigr)$ we have
\[
\uprho_\alpha=(F^*F)^\top
\]
(\cite[Example 1.4.2]{NeshveyevTuset})
while in the general case
\[
\uprho_\alpha=\lambda(F^*F)^\top,
\]
where
\[
\lambda=\sqrt{\tfrac{\Tr((F^*F)^{-1})}{\Tr(F^*F)}}.
\]

For each $x\in\Irr{\oon{U}_F^+}$ choose a basis of the Hilbert space $\cH_x$ which diagonalizes $\uprho_x$ and let the corresponding rank one operators be denoted by $e^x_{i,j}\in\B(\cH_x)$.

\begin{theorem}\label{thm:claim}
Assume $\oon{U}_F^+$ is not of Kac type. We have
\begin{equation}\label{eq:claim}
\TtauAInn(\oon{U}_F^+)
=\TtauInn(\oon{U}_F^+)
=\Ttau(\oon{U}_F^+)
=\bigcap_{\Lambda\in\oon{Sp}(F^*F\tens(F^*F)^{-1})\setminus\{1\}}\,\tfrac{2\pi}{\log{\Lambda}}\ZZ.
\end{equation}
\end{theorem}

Note that the equality $\TtauInn(\oon{U}_F^+)
=\Ttau(\oon{U}_F^+)$ shows in particular that Conjecture \ref{conj:main} holds for $\oon{U}_F^+$ and all invertible matrices $F$.

\begin{proof}[Proof of Theorem \ref{thm:claim}]
Take $t\in\TtauAInn(\oon{U}_F^+)$ and let $(v_n)_{n\in\NN}$ be a sequence of unitaries in $\Linf(\oon{U}_F^+)$ such that $\Ad(v_n)\xrightarrow[n\to\infty]{}\tau_t$. This means that for any $\omega\in\Lone(\oon{U}_F^+)$ we have
\begin{equation}\label{eq23-drugie}
\bigl\|\omega\comp\Ad(v_n)-\omega\comp\tau_t\bigr\|\xrightarrow[n\to\infty]{}0.
\end{equation}
Assume by contradiction that $t\not\in\Ttau(\oon{U}_F^+)$. Then there exists $k,l\in\{1,\dotsc,\dim{\alpha}\}$ with
\[
\tau_t(u^{\alpha}_{k,l})=\uprho_{\alpha,k}^{\ii{t}}\uprho_{\alpha,l}^{-\ii{t}}u^{\alpha}_{k,l}
\]
and $\uprho_{\alpha,k}^{\ii{t}}\uprho_{\alpha,l}^{-\ii{t}}\neq{1}$.

Set $c=|1-\uprho_{\alpha,k}^{\ii{t}}\uprho_{\alpha,l}^{-\ii{t}}|>0$ and define $\eps>0$ via
\[
\eps=\tfrac{1}{4}\min\biggl\{1,c^2
\tfrac{\uprho_{\alpha,l}^{2}}{(\qdim{\alpha})^2}
\Bigl(1+(3+\|\uprho_\alpha\|^2)28
\max\bigl\{(\dim{\alpha\beta})^2,(\dim{\alpha^2\beta})^2,(\dim{\beta\alpha})^2\bigr\}
\Bigr)^{-2}
\biggr\}.
\]
Next, from \eqref{eq23-drugie} we obtain $n\in\NN$ such that
\begin{equation}\label{eq27-drugie}
\bigl\|\bh\bigl(\cdot\,{u^x_{i,j}}^*)\comp\Ad(v_n)-\bh(\cdot\,{u^x_{i,j}}^*)\comp\tau_t\bigr\|\leq\eps
\end{equation}
for $x\in\{\alpha\beta,\alpha^2\beta,\beta\alpha\}$ and $i,j\in\{1,\dotsc,\dim{x}\}$. Using Proposition \ref{prop:Vaes33} and the fact that $V^{\tau,t}_x$ is isometric we have
\begin{align*}
\bigl\|v_n^*-\bh(v_n^*)\I\bigr\|_2^2
&\leq14^2\max_{x\in\{\alpha\beta,\alpha^2\beta,\beta\alpha\}}\bigl\|v_n^*\tens\I-V^{\tau,-t}_x(v_n^*)\bigr\|^2_2\\
&=14^2\max_{x\in\{\alpha\beta,\alpha^2\beta,\beta\alpha\}}\bigl(2-2\oon{Re}\is{v_n^*\tens\I}{V_x^{\tau,-t}(v_n^*)}\bigr)\\
&=14^2\max_{x\in\{\alpha\beta,\alpha^2\beta,\beta\alpha\}}\Bigl(2-2\oon{Re}(\bh\tens\omega_{x})
\bigl((v_n\tens\I)(\tau_{-t}\tens\id)(u^x)(v_n^*\tens\I){u^x}^*\bigr)\Bigr)\\
&=14^2\max_{x\in\{\alpha\beta,\alpha^2\beta,\beta\alpha\}}\biggl(2-2\sum_{i,j,r,s=1}^{\dim{x}}\oon{Re}
\Bigl(\bh\bigl(\Ad(v_n)\bigl(\tau_{-t}(u^x_{i,j})\bigr){u^x_{r,s}}^*\bigr)\omega_x(e^{x}_{i,j}e^x_{s,r})\Bigr)
\biggr)\\
&=14^2\max_{x\in\{\alpha\beta,\alpha^2\beta,\beta\alpha\}}\biggl(2-2\sum_{i,j,r,s=1}^{\dim{x}}
\oon{Re}\Bigl(
\bh\bigl(\tau_t(\tau_{-t}(u^x_{i,j})){u^x_{r,s}}^*\bigr)\,
\omega_x(e^x_{i,j}e^x_{s,r})\Bigr)\\
&\qquad
-2\sum_{i,j,r,s=1}^{\dim{x}}
\oon{Re}\Bigl(\bh\bigl(\Ad(v_n)(\tau_{-t}(u^x_{i,j})){u^x_{r,s}}^*\bigr)-
\bh\bigl(\tau_t(\tau_{-t}(u^x_{i,j})){u^x_{r,s}}^*\bigr)
\omega_x(e^{x}_{i,j}e^x_{s,r})\Bigr)\biggr).
\end{align*}
The first term is
\begin{align*}
2-2\sum_{i,j,r,s=1}^{\dim{x}}\oon{Re}\bigl(\bh(u^x_{i,j}{u^x_{r,s}}^*)\omega_x(e^x_{i,j}e^x_{s,r})\bigr)=
2-2\oon{Re}(\bh\tens\omega_x)(u^x{u^x}^*)=0
\end{align*}
for each $x$, while the second can be bounded by
\begin{align*}
2\cdot14^2&\max_{x\in\{\alpha\beta,\alpha^2\beta,\beta\alpha\}}
\sum_{i,j,r,s=1}^{\dim{x}}
\Bigl|
\bh\Bigl(\Ad(v_n)\bigl(\tau_{-t}(u^x_{i,j})\bigr){u^x_{r,s}}^*\Bigr)-
\bh\Bigl(\tau_t\bigl(\tau_{-t}(u^x_{i,j})\bigr){u^x_{r,s}}^*\Bigr)
\omega_x(e^{x}_{i,j}e^x_{s,r})
\Bigr|\\
&\leq2\cdot14^2
\max_{x\in\{\alpha\beta,\alpha^2\beta,\beta\alpha\}}
\sum_{i,j,r,s=1}^{\dim{x}}
\bigl\|\bh(\cdot\,{u^x_{r,s}}^*)\comp\Ad(v_n)-
\bh(\cdot\,{u^x_{r,s}}^*)\comp\tau_t\bigr\|\bigl\|\tau_{-t}(u^x_{i,j})\bigr\|\bigl|\omega_x(e^x_{i,j}e^x_{s,r})\bigr|\\
&\leq2\cdot14^2\biggl(\max_{x\in\{\alpha\beta,\alpha^2\beta,\beta\alpha\}}(\dim{x})^4\biggr)\eps,
\end{align*}
and consequently
\begin{equation}\label{eq26-drugie}
\begin{aligned}
\bigl\|v_n^*-\bh(v_n^*)\I\bigr\|_2
&\leq\sqrt{2}\cdot14\biggl(\max_{x\in\{\alpha\beta,\alpha^2\beta,\beta\alpha\}}(\dim{x})^2\biggr)\sqrt{\eps}\\
&\leq28\biggl(\max_{x\in\{\alpha\beta,\alpha^2\beta,\beta\alpha\}}(\dim{x})^2\biggr)\sqrt{\eps}.
\end{aligned}
\end{equation}
It follows that
\begin{equation}\label{eq24-drugie}
1=\|v_n^*\|_2\leq\bigl\|v_n^*-\bh(v_n^*)\I\bigr\|_2+\bigl|\bh(v_n^*)\bigr|
\leq28\biggl(\max_{x\in\{\alpha\beta,\alpha^2\beta,\beta\alpha\}}(\dim{x})^2\biggr)\sqrt{\eps}+\bigl|\bh(v_n^*)\bigr|.
\end{equation}
Observe that
\begin{align*}
\bigl|\bh(v_nu^{\alpha}_{k,l}&v_n^*{u^\alpha_{k,l}}^*)-\bh(u^\alpha_{k,l}{u^\alpha_{k,l}}^*)\bigr|
\leq\Bigl|\bh\Bigl(\bigl(v_n-\bh(v_n)\I\bigr)u^\alpha_{k,l}v_n^*{u^\alpha_{k,l}}^*\Bigr)\Bigl|
+\Bigl|\bh\Bigl(u^\alpha_{k,l}\bigl(\bh(v_n)v_n^*-\I\bigr){u^\alpha_{k,l}}^*\Bigr)\Bigr|\\
&\le\bigl\|v_n^*-\bh(v_n^*)\I\bigr\|_2
+\Bigl|\bh\Bigl(u^\alpha_{k,l}\bigl(\bh(v_n)v_n^*-|\bh(v_n^*)|^2\I\bigr){u^\alpha_{k,l}}^*\Bigr)\Bigr|
+\Bigl|\bigl|\bh(v_n^*)\bigr|^2-1\Bigr|\bigl|\bh(u^\alpha_{k,l}{u^\alpha_{k,l}}^*)\bigr|\\
&\leq\bigl\|v_n^*-\bh(v_n^*)\I\bigr\|_2+\bigl|\bh(v_n)\bigr|
\Bigl|\bh\Bigl(u^\alpha_{k,l}\bigl(v_n^*-\bh(v_n^*)\I\bigr){u^\alpha_{k,l}}^*\Bigr)\Bigr|
+\Bigl|\bigl|\bh(v_n^*)\bigr|^2-1\Bigr|\\
&\leq\Bigl(1+\bigl\|\sigma^{\bh}_{-\ii}(u^\alpha_{k,l})\bigr\|\Bigr)\bigl\|v_n^*-\bh(v_n^*)\I\bigr\|_2
+\Bigl|\bigl|\bh(v_n^*)\bigr|^2-1\Bigr|\\
&\leq\bigl(1+\|\uprho_\alpha\|^2\bigr)\bigl\|v_n^*-\bh(v_n^*)\I\bigr\|_2+2\Bigl|\bigl|\bh(v_n^*)\bigr|-1\Bigr|,
\end{align*}
and using \eqref{eq26-drugie}, \eqref{eq24-drugie}
\begin{equation}\label{eq25-drugie}
\bigl|\bh(v_nu^\alpha_{k,l}v_n^*{u^\alpha_{k,l}}^*)-\bh(u^{\alpha}_{k,l}{u^\alpha_{k,l}}^*)\bigr|
\leq\bigl(3+\|\uprho_\alpha\|^2\bigr)28
\biggl(\max_{x\in\{\alpha\beta,\alpha^2\beta,\beta\alpha\}}(\dim{x})^2\biggr)\sqrt{\eps}.
\end{equation}
Finally, using \eqref{eq27-drugie} and \eqref{eq25-drugie}, we have
\begin{align*}
\eps&\geq\Bigl|\bh(\cdot\,{u^\alpha_{k,l}}^*)\comp\bigl(\Ad(v_n)-\tau_t\bigr)(u^\alpha_{k,l})\Bigr|
=\Bigl|\bh(v_nu^\alpha_{k,l}v_n^*{u^\alpha_{k,l}}^*)-\bh\bigl(\tau_t(u^\alpha_{k,l}){u^\alpha_{k,l}}^*\bigr)\Bigr|\\
&\geq\bigl|\bh(u^\alpha_{k,l}{u^\alpha_{k,l}}^*)-\uprho_{\alpha,k}^{\ii{t}}\uprho_{\alpha,l}^{-\ii{t}}\bh(u^\alpha_{k,l}{u^\alpha_{k,l}}^*)\bigr|
-\bigl|\bh(v_nu^\alpha_{k,l}v_n^*{u^\alpha_{k,l}}^*)-\bh(u^\alpha_{k,l}{u^\alpha_{k,l}}^*)\bigr|\\
&\geq{c}\,\tfrac{\uprho_{\alpha,l}}{\qdim{\alpha}}-\bigl(3+\|\uprho_\alpha\|^2\bigr)28
\biggl(\max_{x\in\{\alpha\beta,\alpha^2\beta,\beta\alpha\}}(\dim{x})^2\biggr)\sqrt{\eps},
\end{align*}
and using definition of $\eps$ we arrive at
\begin{align*}
c&\leq\tfrac{\qdim{\alpha}}{\uprho_{\alpha,l}}\Bigl(\eps+\bigl(3+\|\uprho_\alpha\|^2\bigr)28
\bigl(\,\max_{x\in\{\alpha\beta,\alpha^2\beta,\beta\alpha\}}(\dim{x})^2\bigr)\sqrt{\eps}\Bigr)\\
&\leq\tfrac{\qdim{\alpha}}{\uprho_{\alpha,l}}\Bigl(1+\bigl(3+\|\uprho_\alpha\|^2\bigr)28
\bigl(\,\max_{x\in\{\alpha\beta,\alpha^2\beta,\beta\alpha\}}(\dim{x})^2\bigr)\Bigr)\sqrt{\eps}\leq\tfrac{c}{2}.
\end{align*}
Thus we can conclude that $t\in\Ttau(\oon{U}_F^+)$.

It is now left to identify this group. Since matrix elements of $u^\alpha$ generate $\Linf(\oon{U}_F^+)$, a number $t\in\RR$ belongs to $\Ttau(\oon{U}_F^+)$ if and only if $u^\alpha_{i,j}=\tau_t(u^\alpha_{i,j})=\bigl(\tfrac{\uprho_{\alpha,i}}{\uprho_{\alpha,j}}\bigr)^{\ii{t}}u^\alpha_{i,j}$ for all $i,j\in\{1,\dotsc,N\}$, i.e.~if and only if $\bigl(\tfrac{\uprho_{\alpha,i}}{\uprho_{\alpha,j}}\bigr)^{\ii{t}}=1$ for all $i,j\in\{1,\dotsc,N\}$ and this condition is equivalent to
\[
\forall\:\Lambda\in\oon{Sp}(\uprho_\alpha\tens\uprho_\alpha^{-1})\setminus\{1\}\quad{t}\in\tfrac{2\pi}{\log{\Lambda}}\ZZ.
\]
(Note that if $\oon{U}_F^+$ is of Kac type then this condition is vacuous and $\Ttau(\oon{U}_F^+)=\RR$.) Finally recall that $\uprho_\alpha=\sqrt{\tfrac{\Tr((F^*F)^{-1})}{\Tr(F^*F)}}(F^*F)^\top$. The claim \eqref{eq:claim} follows from
\[
\oon{Sp}(\uprho_\alpha\tens\uprho_\alpha^{-1})=
\oon{Sp}\Bigl((F^*F)^\top\tens\bigl((F^*F)^\top\bigr)^{-1}\Bigr)=
\oon{Sp}\bigl(F^*F\tens(F^*F)^{-1}\bigr).\qedhere
\]
\end{proof}

\begin{remark}
Equality $\TtauInn(\oon{U}_F^+)=\TtauAInn(\oon{U}_F^+)$ follows also from the fact that $\Linf(\oon{U}_F^+)$ is full (\cite[Theorem 33]{DeCommerFreslonYamashita}).
\end{remark}

Clearly $\Mod(\oon{U}_F^+)=\RR$, but $\hh{\oon{U}_F^+}$ is not always unimodular. In order to determine $\Mod(\hh{\oon{U}_F^+})$ let us recall that the modular element of a discrete quantum group was already described in \cite[Theorem 3.3 and Remark on page 402]{PodlesWoronowicz} (cf.~the proof of Theorem \ref{thm2} in Section \ref{sect:generalRS}). Namely, if $\bbGamma$ is any discrete quantum group with
\[
\ell^{\infty}(\bbGamma)=\prod_{x\in\Irr{\hh{\bbGamma}}}\B(\cH^x)
\]
then the modular element of $\bbGamma$ is the element affiliated with $\ell^{\infty}(\bbGamma)$ defined as $\bigoplus\limits_{x\in\Irr{\hh{\bbGamma}}}\uprho_x^2$. It follows that $\Mod(\bbGamma)$ is the set of $t\in\RR$ such that $\uprho_x^{2\ii{t}}=\I$ for all $x\in\Irr{\hh{\bbGamma}}$. Clearly this is the same as the set of those $t\in\RR$ for which $\uprho_V^{2\ii{t}}=\I$ for all finite dimensional representations $V$ of $\hh{\bbGamma}$. If $\hh{\bbGamma}$ happens to be a matrix quantum group, i.e.~all representations are subrepresentations of tensor products of the fundamental representation $U$ and its conjugate $\overline{U}$, we find that $\Mod(\bbGamma)$ is the set of those $t\in\RR$ which satisfy $\uprho_V^{\ii{t}}=\I$ for all $V$ which are finite tensor products of $U$ and $\overline{U}$.

Returning to $\oon{U}_F^+$, let $V$ be a finite tensor product of $u^\alpha$ and $u^\beta$. Then $\uprho_V$ is the corresponding tensor product of $\uprho_\alpha$ and $\uprho_\beta=\uprho_{\overline{\alpha}}$ (the latter being the transpose of $\uprho_\alpha^{-1}$, cf.~\cite[Proposition 1.4.7]{NeshveyevTuset}). Thus $\uprho_V^{2\ii{t}}=\I$ if and only if $\uprho_\alpha^{2\ii{t}}=\I$.

Thus, if $\oon{U}_F^+$ is not of Kac type then
\begin{equation}\label{eq48}
\Mod(\hh{\oon{U}_F^+})=\bigcap_{\Lambda\in\oon{Sp}(F^*F)\setminus\{\lambda^{-1}\}}\tfrac{\pi}{\log{\lambda}+\log{\Lambda}}\ZZ
\end{equation}
where $\lambda=\sqrt{\tfrac{\Tr((F^*F)^{-1})}{\Tr(F^*F)}}$. Obviously if $\oon{U}_F^+$ is of Kac type then $\Mod(\hh{\oon{U}_F^+})=\RR$.

In view of Remark \ref{remark1} and the fact that $\Linf(\oon{U}_F^+)$ is full, which means that $\oon{Inn}(\Linf(\oon{U}_F^+))=\oon{\overline{Inn}}(\Linf(\oon{U}_F^+))$, the one remaining invariant from our list is $\TsigInn(\oon{U}_F^+)=T(\Linf(\oon{U}_F^+))$.

In order to compute this invariant, we recall again the main result of \cite[Appendix]{DeCommerFreslonYamashita} which says that for any $F\in\oon{GL}(N,\CC)$ the von Neumann algebra is $\Linf(\oon{U}_F^+)$ is a full factor and its invariant $\Sd$ (\cite[Definition 1.2]{ConnesIII1}) is the multiplicative subgroup of $\RR_{>0}$ generated by the eigenvalues of $\uprho_\alpha\tens\uprho_\alpha$.

Now the fact that $\Linf(\oon{U}_F^+)$ is full lets us conclude by \cite[Corollary 4.11]{ConnesIII1} that $S(\Linf(\oon{U}_F^+))$ is the closure of $\Sd(\Linf(\oon{U}_F^+))$. In particular it cannot be the set $\{0,1\}$ and it follows that $\Linf(\oon{U}_F^+)$ is not a factor of type $\mathrm{III}_0$. Furthermore it is not a factor of type $\mathrm{II}_\infty$ because then $S(\Linf(\oon{U}_F^+))$ would be $\{1\}$, but this would mean that $\uprho_\alpha=\I$ and $\oon{U}_F^+$ is then of Kac type, which makes $\Linf(\oon{U}_F^+)$ a type $\mathrm{II}_1$ factor.

Consequently $\Linf(\oon{U}_F^+)$ is either a factor of type $\mathrm{II}_1$ or of type $\mathrm{III}_\mu$ with $\mu\in\left]0,1\right]$. The first case happens when $\uprho_\alpha=\I$, then $T(\Linf(\oon{U}_F^+))=\RR$. In all the latter cases the Connes invariant $T(\Linf(\oon{U}_F^+))$ is determined by the invariant $S(\Linf(\oon{U}_F^+))$ (\cite[Corollary 3.12, Proposition 3.7]{Connes}). More precisely, in case the subgroup of $\RR_{>0}$ generated by the eigenvalues of $\uprho_\alpha\tens\uprho_\alpha$ is dense (case $\mathrm{III}_1$) we have $T(\Linf(\oon{U}_F^+))=\{0\}$. On the other hand, if this subgroup is non-trivial, but not dense, then it is of the form $\{\mu^n\,|\,n\in\ZZ\}$ for some $\mu\in\left]0,1\right[$ and $\Linf(\oon{U}_F^+)$ is a type $\mathrm{III}_\mu$ factor with Connes' invariant $S(\Linf(\oon{U}_F^+))=\{\mu^n\,|\,n\in\ZZ\}$. Consequently $T(\Linf(\oon{U}_F^+))=\tfrac{2\pi
}{\log{\mu}}\ZZ$.

We have calculated all the modular invariants of $\oon{U}_F^+$ and its dual in terms of eigenvalues of $F^*F$. We finish this section with two results: the first proposition tells that $\Mod(\hh{\oon{U}_F^+})$ is almost determined by the type of $\Linf(\oon{U}_F^+)$, next we describe some interesting examples.

\begin{proposition}\label{prop9}
\noindent\begin{enumerate}
\item If $\Linf(\oon{U}_F^+)$ is a factor of type $\mathrm{II}_1$ then $\oon{Mod}(\hh{\oon{U}_F^+})=\RR$.
\item If $\Linf(\oon{U}_F^+)$ is a factor of type $\mathrm{III}_\mu$ for some $\mu\in\left]0,1\right[$, then either $\Mod(\hh{\oon{U}_F^+})=\tfrac{\pi}{\log{\mu}}\ZZ$ or $\Mod(\hh{\oon{U}_F^+})=\tfrac{2\pi}{\log{\mu}}\ZZ$.
\item If $\Linf(\oon{U}_F^+)$ is a factor of type $\mathrm{III}_1$, then $\Mod(\hh{\oon{U}_F^+})=\{0\}$.
\end{enumerate}
\end{proposition}

\begin{proof}
The first case is trivial, as then $\oon{U}_F^+$ is of Kac type. Assume that $\Linf(\oon{U}_F^+)$ is a factor of type $\mathrm{III}_\mu$ for some $\mu\in\left]0,1\right[$. According to the discussion above, this means that $\oon{Sp}(\uprho_\alpha\tens\uprho_\alpha)\subset\{\mu^n\,|\,n\in\ZZ\}$, so in particular $\oon{Sp}(\uprho_\alpha)=\oon{Sp}(\lambda{F^*F})=\bigl\{\mu^{k_{1}\:\!\!/2},\dotsc,\mu^{k_{N}\:\!\!/2}\bigr\}$ for certain $k_1,\dotsc,k_N\in\ZZ$ and, furthermore, the set $\bigl\{\mu^{(k_i+k_j)\:\!\!/2}\,\bigr|\bigl.\,1\leq{i,j}\leq{N}\bigr\}$ generates $\{\mu^n\,|\,n\in\ZZ\}$. Write $g=\oon{gcd}\bigl(\{|k_1|,\dotsc,|k_N|\}\setminus\{0\}\bigr)$. Then using \eqref{eq48} we have $\Mod(\hh{\oon{U}_F^+})=\bigcap\limits_{i=1,\:k_i\neq{0}}^{N}\tfrac{2\pi}{k_i\log{\mu}}\ZZ=\tfrac{2\pi}{g\log{\mu}}\ZZ$, and hence we need to show that $g\in\{1,2\}$. Write $k_i=gk_i'$ ($1\leq{i}\leq{N}$) and $g=2^mg'$, where $m\in\ZZ_+$ and $g'\in\NN$ is not divisible by $2$. Observe that for $1\leq{i,j}\leq{N}$ we have $\tfrac{k_i+k_j}{2}=2^{m-1}g'(k_i'+k_j')\in\ZZ$. Since $g'$ is not divisible by $2$, it must be the case that $2^{m-1}(k'_i+k'_j)\in\ZZ$. If $g'\neq{1}$ then $\bigl\{\mu^{(k_i+k_j)\:\!\!/2}\,\bigr|\bigl.\,1\leq{i,j}\leq{N}\bigr\}\subset\bigl\{\mu^{g'n}\,\bigr|\bigl.\,n\in\ZZ\bigr\}$, which is a proper subgroup of $\{\mu^n\,|\,n\in\ZZ\}$. This is a contradiction. It follows that $g'=1$ and $g=2^m$. If $m\geq{2}$ then similarly $\bigl\{\mu^{(k_i+k_j)\:\!\!/2}\,\bigr|\bigl.\,1\leq{i,j}\leq{N}\bigr\}\subset\bigl\{\mu^{2^{m-1}n}\,\bigr|\bigl.\,n\in\ZZ\bigr\}$ and consequently $g\in\{1,2\}$ as claimed.

Assume now that $\Linf(\oon{U}_F^+)$ is a type $\mathrm{III}_1$ factor, i.e.~$\oon{Sp}(\uprho_\alpha\tens\uprho_\alpha)$ generates a dense subgroup of $\RR_{>0}$. Assume by contradiction that $\Mod(\hh{\oon{U}_F^+})\neq\{0\}$. Then, since $\Mod(\hh{\oon{U}_F^+})$ is closed and not equal to $\RR$ (this would mean that $\oon{U}_F^+$ is of Kac type), it is a cyclic subgroup of $\RR$, say $\Mod(\hh{\oon{U}_F^+})=c\ZZ$ for some $c>0$. Let $\oon{Sp}(\uprho_\alpha)=\oon{Sp}(\lambda{F^*F})=\{x_1,\dotsc,x_N\}$. It follows from \eqref{eq48} that for any $1\leq{i}\leq{N}$ such that $x_i\neq{1}$ we have $c=\tfrac{\pi{k_i}}{\log{x_i}}$ for some $k_i\in\ZZ$, or equivalently, $x_i=\ee^{\pi{k_i/c}}$. If $x_i=1$ then we also have $x_i=\ee^{\pi{k_i/c}}$ with $k_i=0$. Consequently for any $1\leq{i,j}\leq{N}$
\[
x_ix_j=
\ee^{\pi(k_i+k_j)/c}\in\bigl\{\ee^{\pi{n/c}}\,\bigr|\bigl.\,n\in\ZZ\bigr\}
\]
and $\oon{Sp}(\uprho_\alpha\tens\uprho_\alpha)$ does not generate a dense subgroup. This contradiction shows that we must have $\Mod(\hh{\oon{U}_F^+})=\{0\}$.
\end{proof}

We finish this section with examples which in particular show that all the possibilities of Proposition \ref{prop9} can occur.

\begin{example}
\noindent\begin{enumerate}
\item\label{ex:diffF1} Fix natural $n\geq{2}$ and let $\mu\in\left]0,1\right[$ be a (dependent on $n$) solution of $\mu^{n}+\mu^{3n+1}+\mu^{-3n-2}=\mu^{-n}+\mu^{-3n-1}+\mu^{3n+2}$. Next, define $F=\oon{diag}(\mu^{n/2},\mu^{(3n+1)/2},\mu^{-(3n+2)/2})$ and consider the associated unitary quantum group $\oon{U}_F^+$. By our choice of $\mu$ and $F$ we have $\lambda=1$ and $\oon{Sp}(\uprho_\alpha)=\{\mu^n,\mu^{3n+1},\mu^{-3n-2}\}$. From the results discussed above we see that $\Linf(\oon{U}_F^+)$ is a factor of type $\mathrm{III}_\mu$, hence $T(\Linf(\oon{U}_F^+))=\tfrac{2\pi}{\log{\mu}}\ZZ$ and since $\oon{gcd}(2n,6n+2,6n+4)=2$, the reasoning in the proof of Proposition \ref{prop9} gives $\Mod(\oon{U}_F^+)=\tfrac{\pi}{\log{\mu}}\ZZ$. Finally Theorem \ref{thm:claim} gives $\TtauInn(\oon{U}_F^+)=\tfrac{2\pi}{(2n+1)\log{\mu}}\ZZ$.
\item Fix a natural $n\in\NN$ and let, as in the previous example, $\mu\in\left]0,1\right[$ be a (dependent on $n$) solution of the equation $\mu^{n+1/2}+\mu^{3n+5/2}+\mu^{-(3n+7/2)}=\mu^{-(n+1/2)}+\mu^{-(3n+5/2)}+\mu^{3n+7/2}$. Next, set $F=\oon{diag}(\mu^{(n+1/2)/2},\mu^{(3n+5/2)/2},\mu^{-(3n+7/2)/2})$. As in example \eqref{ex:diffF1}, we have $\lambda=1$ and $\oon{Sp}(\uprho_\alpha)=\{\mu^{n+1/2},\mu^{3n+5/2},\mu^{-(3n+7/2)}\}$. Consequently $\Linf(\oon{U}_F^+)$ is a factor of type $\mathrm{III}_\mu$ and $T(\Linf(\oon{U}_F^+))=\tfrac{2\pi}{\log{\mu}}\ZZ$. Furthermore $\oon{gcd}(2n+1,6n+5,6n+7)=1$, and hence the proof of Proposition \ref{prop9} gives $\Mod(\hh{\oon{U}_F^+})=\tfrac{2\pi}{\log{\mu}}\ZZ$. Theorem \ref{thm:claim} shows $\TtauInn(\oon{U}_F^+)=\tfrac{2\pi}{(2n+2)\log{\mu}}\ZZ$.
\item Take $x>1$ which is not an algebraic number and let $\kappa=\sqrt{\tfrac{2+1/x}{2+x}}\in\left]0,1\right[$, so that $2\kappa+\kappa{x}=\tfrac{2}{\kappa}+\tfrac{1}{\kappa{x}}$. Next, define $F=\oon{diag}(\sqrt{\kappa},\sqrt{\kappa},\sqrt{\kappa{x}})$ and consider the associated quantum group $\oon{U}_F^+$. We have $\lambda=1$ and $\oon{Sp}(\uprho_\alpha)=\{\kappa,\kappa{x}\}$. Consequently, the multiplicative group generated by $\oon{Sp}(\uprho_\alpha\tens\uprho_\alpha)$ is the group generated by $\kappa^2,x$ which is easily seen to be dense in $\RR_{>0}$ as $x$ is assumed to be a transcendental number (c.f.~\cite[Example 5.27]{faktory}). Consequently $\Linf(\oon{U}_F^+)$ is a factor of type $\mathrm{III}_1$ and Proposition \ref{prop9} gives $\Mod(\hh{\oon{U}_F^+})=\{0\}$. Since $\oon{Sp}(\uprho_\alpha\tens\uprho_\alpha^{-1})=\{1,x,x^{-1}\}$, Theorem \ref{thm:claim} gives $\TtauInn(\oon{U}_F^+)=\tfrac{2\pi}{\log{x}}\ZZ$.
\end{enumerate}
\end{example}

\section{i.c.c.-type conditions}\label{sect:icc}

Given a locally compact quantum group $\GG$ and $n\in\NN$ we denote by $\Delta_\GG^{(n)}$ the map
\[
(\underbrace{\id\tens\dotsm\tens\id}_{n-1}\tens\Delta_\GG)\comp\dotsm\comp(\id\tens\Delta_\GG)\comp\Delta_\GG\colon\Linf(\GG)\longrightarrow\underbrace{\Linf(\GG)\vtens\dotsm\vtens\Linf(\GG)}_{n+1}.
\]
(cf.~\cite[p.~5]{sweedler}, \cite[p.~57]{abe} as well as Section \ref{sect:prel}). In particular $\Delta_\GG^{(1)}=\Delta_\GG$.

\begin{proposition}\label{prop:firstICC}
Let $\GG$ be a locally compact quantum group and assume that
\[
\Delta_\GG^{(n)}\bigl(\Linf(\GG)\bigr)'\cap\underbrace{\Linf(\GG)\vtens\dotsm\vtens\Linf(\GG)}_{n+1}=\CC\I
\]
for some $n\in\NN$. Then $\Linf(\GG)$ is a factor.
\end{proposition}

\begin{proof}
Take $z\in\cZ(\Linf(\GG))$. Then $z\tens\I^{\tens{n}}$ commutes with the range of $\Delta_\GG^{(n)}$ and hence $z\tens\I^{\tens{n}}\in\CC\I^{\tens(n+1)}$ which implies $z\in\CC\I$.
\end{proof}

For a discrete group $\Gamma$ the group von Neumann algebra $\oon{\mathnormal{L}}(\Gamma)$ of $\Gamma$ is the von Neumann algebra of operators on $\ell^2(\Gamma)$ generated by the operators of the left regular representation $\{\lambda_t\,|\,t\in\Gamma\}$. The dual $\hh{\Gamma}$ of $\Gamma$ is by definition the compact quantum group $\GG$ such that $\Linf(\GG)=\oon{\mathnormal{L}}(\Gamma)$ with the unique comultiplication mapping $\lambda_t$ to $\lambda_t\tens\lambda_t$ ($t\in\Gamma$).

\begin{proposition}
Let $\Gamma$ be a discrete group. Then the following are equivalent:
\begin{enumerate}
\item\label{prop:classICC1} $\Gamma$ is i.c.c.,
\item\label{prop:classICC2} $\oon{\mathnormal{L}}(\Gamma)$ is a factor,
\item\label{prop:classICC3} $\Delta_{\hh{\Gamma}}^{(n)}\bigl(\oon{\mathnormal{L}}(\Gamma)\bigr)'\cap\underbrace{\oon{\mathnormal{L}}(\Gamma)\vtens\dotsm\vtens\oon{\mathnormal{L}}(\Gamma)}_{n+1}=\CC\I$ for some $n\in\NN$,
\item\label{prop:classICC4} $\Delta_{\hh{\Gamma}}^{(n)}\bigl(\oon{\mathnormal{L}}(\Gamma)\bigr)'\cap\underbrace{\oon{\mathnormal{L}}(\Gamma)\vtens\dotsm\vtens\oon{\mathnormal{L}}(\Gamma)}_{n+1}=\CC\I$ for all $n\in\NN$.
\end{enumerate}
\end{proposition}

\begin{proof}
Equivalence of \eqref{prop:classICC1} and \eqref{prop:classICC2} is well-known, \eqref{prop:classICC4} $\Rightarrow$ \eqref{prop:classICC3} is clear, while \eqref{prop:classICC3} $\Rightarrow$ \eqref{prop:classICC2} is contained in Proposition \ref{prop:firstICC}.

For \eqref{prop:classICC1} $\Rightarrow$ \eqref{prop:classICC4} take $x\in\underbrace{\oon{\mathnormal{L}}(\Gamma)\vtens\dotsm\vtens\oon{\mathnormal{L}}(\Gamma)}_{n+1}$ and write $x$ as the $\|\cdot\|_2$-convergent series
\[
x=\sum\limits_{a_1,\dotsc,a_{n+1}\in\Gamma}x_{a_1,\dotsc,a_{n+1}}\lambda_{a_1}\tens\dotsm\tens\lambda_{a_{n+1}}.
\]
If $x$ commutes with $\Delta_{\hh{\Gamma}}^{(n)}(\lambda_c)$ for all $c$ then
\begin{align*}
\sum_{a_1,\dotsc,a_{n+1}\in\Gamma}x_{a_1,\dotsc,a_{n+1}}\lambda_{a_1}\tens\dotsm&\tens\lambda_{a_{n+1}}=x=(\lambda_c\tens\dotsm\tens\lambda_c)x(\lambda_c\tens\dotsm\tens\lambda_c)^{-1}\\
&=\sum_{a_1,\dotsc,a_{n+1}\in\Gamma}x_{a_1,\dotsc,a_{n+1}}\lambda_{c{a_1}c^{-1}}\tens\dotsm\tens\lambda_{c{a_{n+1}}c^{-1}}\\
&=\sum_{a_1,\dotsc,a_{n+1}\in\Gamma}x_{c^{-1}a_1c,\dotsc,c^{-1}a_{n+1}c}\lambda_{a_1}\tens\dotsm\tens\lambda_{a_{n+1}}.
\end{align*}
Now if additionally $x\not\in\CC\I$ then there exists an $(n+1)$-tuple $(s_1,\dotsc,s_{n+1})\neq(e,\dotsc,e)$ with $x_{s_1,\dotsc,s_{n+1}}\neq{0}$. Let $k$ be such that $s_k\neq{e}$. Then the conjugacy class of $s_k$ is infinite. Consequently there are infinitely many $(n+1)$-tuples $(a_1,\dotsc,a_{n+1})$ such that $(c^{-1}a_1c,\dotsc,c^{-1}a_{n+1}c)=(s_1,\dotsc,s_{n+1})$ for some $c$ (namely take $a_k$ from the conjugacy class of $s_k$, so that $a_k=cs_kc^{-1}$ for some $c$, and put $a_l=cs_lc^{-1}$ for the remaining $l$). But this leads to the contradiction
\[
+\infty>\|x\|_2^2=\sum_{p_1,\dotsc,p_{n+1}\in\Gamma}|x_{p_1,\dotsc,p_{n+1}}|^2=+\infty.
\]
It follows that $x$ must be proportional to $\I$.
\end{proof}

\begin{definition}
Let $\bbGamma$ be a discrete quantum group. We say that $\bbGamma$ is $n$-i.c.c.~if
\[
\Delta_{\hh{\bbGamma}}^{(n)}\bigl(\Linf(\hh{\bbGamma})\bigr)'\cap\underbrace{\Linf(\hh{\bbGamma})\vtens\dotsm\vtens\Linf(\hh{\bbGamma})}_{n+1}=\CC\I.
\]
\end{definition}

We immediately see that for a classical discrete group $\Gamma$ we have that $\Gamma$ is an i.c.c.~group if and only if $\Gamma$ is $1$-i.c.c., which happens if and only if $\Gamma$ is $n$-i.c.c.~for all $n\in\NN$.

\begin{remark}
Let $\bbGamma$ be a discrete quantum group. If $y\in\Delta_{\hh{\bbGamma}}^{(n-1)}\bigl(\Linf(\hh{\bbGamma})\bigr)'\cap\underbrace{\Linf(\hh{\bbGamma})\vtens\dotsm\vtens\Linf(\hh{\bbGamma})}_{n}$ for some $n\geq{2}$ then $\bigl(\Delta_{\hh\bbGamma}\tens\id^{\tens(n-1)}\bigr)y$ belongs to $\Delta_{\hh{\bbGamma}}^{(n)}\bigl(\Linf(\hh{\bbGamma})\bigr)'\cap\underbrace{\Linf(\hh{\bbGamma})\vtens\dotsm\vtens\Linf(\hh{\bbGamma})}_{n+1}
$ which shows that if $\bbGamma$ is $n$-i.c.c.~then it is also $m$-i.c.c.~for all $m<n$.
\end{remark}

\begin{theorem}
Assume that $\GG$ is a second countable compact quantum group such that $\hh{\GG}$ is $1$-i.c.c.~and $\TtauInn(\GG)=\RR$. Then $\GG$ is of Kac type.
\end{theorem}

\begin{proof}
Write $\tau^{\GG}_t=\Ad(b_0^{\ii{t}})$ for a strictly positive, self-adjoint operator $b_0$ affiliated with $\Linf(\GG)$ (the existence of $b_0$ follows from \cite[Theorem 0.1]{Kallman} and Stone's theorem, cf.~e.g.~the proof of Theorem \ref{thm:Un}). Using $\Delta_{\GG}\comp\tau^{\GG}_t=(\tau^\GG_t\tens\tau^\GG_t)\comp\Delta_{\GG}$ we obtain
\[
(b_0^{-\ii{t}}\tens{b_0^{-\ii{t}}})\Delta_{\GG}(b_0^{\ii{t}})
\in\Delta_{\GG}\bigl(\Linf(\GG)\bigr)'\cap\Linf(\GG)\vtens\Linf(\GG)=\CC\I,
\]
so defining $z_t\in\TT$ by
\begin{equation}\label{eq:defzt}
(b_0^{-\ii{t}}\tens{b_0^{-\ii{t}}})\Delta_{\GG}(b_0^{\ii{t}})=z_t\I
\end{equation}
we obtain a continuous map $\RR\to\TT$. This map is a homomorphism, since
\begin{align*}
z_{t+s}=(b_0^{-\ii(t+s)}\tens{b_0^{-\ii(t+s)}})\Delta_\GG(b_0^{\ii(t+s)})
&=(b_0^{-\ii{t}}\tens{b_0^{-\ii{t}}})\bigl((b_0^{-\ii{s}}\tens{b_0^{-\ii{s}}})
\Delta_\GG(b_0^{\ii{s}})\bigr)\Delta_\GG(b_0^{\ii{t}})\\
&=(b_0^{-\ii{t}}\tens{b_0^{-\ii{t}}})\Delta_\GG(b_0^{\ii{t}})
\bigl((b_0^{-\ii{s}}\tens{b_0^{-\ii{s}}})
\Delta_\GG(b_0^{\ii{s}})\bigr)=z_tz_s.
\end{align*}
It follows that we can write $z_t=\lambda^{\ii{t}}$ for some $\lambda>0$. Define $b=\lambda{b_0}$. Then we still have $\tau_t=\Ad(b^{\ii{t}})$ and by \eqref{eq:defzt}
\[
\Delta_\GG(b^{\ii{t}})=\lambda^{\ii{t}}\Delta_\GG(b_0^{\ii{t}})=\lambda^{\ii{t}}z_t(b_0^{\ii{t}}\tens{b_0^{\ii{t}}})
=(\lambda^{\ii{t}}b_0^{\ii{t}}\tens\lambda^{\ii{t}}b_0^{\ii{t}})=(b^{\ii{t}}\tens{b^{\ii{t}}}).
\]

Fix $t\in\RR$ and for $n\in\NN$ define a bounded measurable function $f_n\colon\left]0,+\infty\right[\to\CC$ by $f_n(x)=\Int_{t-\frac{1}{n}}^{t+\frac{1}{n}}x^{\ii{s}}\dd{s}$. Then
\[
f_n(b)=\Int_{t-\frac{1}{n}}^{t+\frac{1}{n}}b^{\ii{s}}\dd{s}
\]
and
\[
\Delta_{\GG}\bigl(f_n(b)\bigr)
=\Int_{t-\frac{1}{n}}^{t+\frac{1}{n}}\Delta_{\GG}(b^{\ii{s}})\dd{s}
=\Int_{t-\frac{1}{n}}^{t+\frac{1}{n}}(b^{\ii{s}}\tens{b^{\ii{s}}})\dd{s}.
\]
Applying $\bh_\GG\tens\id$ and using right invariance of $\bh_\GG$ gives
\[
\Biggl(\,\Int_{t-\frac{1}{n}}^{t+\frac{1}{n}}\bh_\GG(b^{\ii{s}})\dd{s}\Biggr)\I
=\bh_\GG\bigl(f_n(b)\bigr)\I=(\bh_\GG\tens\id)\Delta_{\GG}\bigl(f_n(b)\bigr)
=\Int_{t-\frac{1}{n}}^{t+\frac{1}{n}}\bh_\GG(b^{\ii{s}})b^{\ii{s}}\dd{s}.
\]
Now multiplying the above by $2n$ and passing to the limit $n\to\infty$ gives
\[
\bh_\GG(b^{\ii{t}})\I=\bh_\GG(b^{\ii{t}})b^{\ii{t}},\qquad{t}\in\RR.
\]
Since $\bh_\GG(b^{\ii0})=\bh_\GG(\I)=1$ and $t\mapsto\bh_\GG(b^{\ii{t}})$ is continuous, there exists $\eps>0$ such that $\bh_\GG(b^{\ii{t}})\neq{0}$ for $|t|<\eps$. Consequently $b^{\ii{t}}=\I$ for such $t$. But then for any $m\in\NN$ we have $b^{\ii{m}t}=(b^{\ii{t}})^m=\I^m=\I$, and hence $b=\I$ and $\GG$ is of Kac type.
\end{proof}

\subsection{Example: \texorpdfstring{$\oon{U}_F^+$}{UF+}}\label{sect:exUFpl}\hspace*{\fill}

In this section we show that for an appropriate choice of the invertible matrix $F$ (this is made more precise in Theorem \ref{thm2-drugie} below) the quantum group $\hh{\oon{U}_F^+}$ satisfies the $n$-i.c.c.~condition. This, in particular, gives an alternative proof that Conjecture \ref{conj:main} holds for $\oon{U}_F^+$, but is meant more as an illustration of the results of Section \ref{sect:icc}. As in Section \ref{sect:UFplus} we adapt the ideas and proofs contained in \cite[Appendix]{DeCommerFreslonYamashita}. We retain the notation introduced in Section \ref{sect:UFplus}. To further ease the notation, we also denote $\Delta_{\oon{U}_F^+}$ simply by $\Delta$ and identify $a\in\Linf(\oon{U}_F^+)$ with its image in $\Ltwo(\oon{U}_F^+)$.

Fix $n\geq{1}$, $x\in\ZZ_+\star\ZZ_+$ and define a bounded map $V^{\Delta,n}_x\colon\Ltwo(\oon{U}_F^+)^{\tens(n+1)}\to\Ltwo(\oon{U}_F^+)^{\tens(n+1)}\tens\cK_x$ by setting
\[
V^{\Delta,n}_x(a)=(\Delta^{(n)}\tens\id)(u^x)(a\tens\I)(\Delta^{(n)}\tens\id)(u^x)^*,\qquad{a}\in\Linf(\oon{U}_F^+)^{\vtens(n+1)}
\]
on a dense set. We will shortly show that $V_x^{\Delta,n}$ as defined above is indeed bounded, hence extends uniquely to the whole space $\Ltwo(\oon{U}_F^+)^{\tens(n+1)}$.

It seems that map $V_x^{\Delta,n}$, unlike $V_x^{\tau,t}$ from Section \ref{sect:UFplus}, is not isometric -- it is however close to being isometric when $\uprho_x$ is close to $\I$, which we will now quantify. Take any $a,b\in\Linf(\oon{U}_F^+)^{\vtens(n+1)}$. Then using equation \eqref{eq28-drugie} and employing the leg-numbering notation $[\,\cdot\,]_{k\,l}$ we have
\begin{align*}
&\!\!\!\!\is{V^{\Delta,n}_x(a)}{V^{\Delta,n}_x(b)}=
(\bh^{\tens(n+1)}\tens\omega_x)
\bigl(
(\Delta^{(n)}\tens\id)(u^x)(a^*b\tens\I)(\Delta^{(n)}\tens\id)({u^x}^*)
\bigr)\\
&=\Bigl(\bh^{\tens(n+1)}\tens\tfrac{\Tr}{\qdim{x}}\Bigr)
\bigl(
(\Delta^{(n)}\tens\id)(u^x)(a^*b\tens\I)(\Delta^{(n)}\tens\id)({u^x}^*)
(\I^{\tens(n+1)}\tens\uprho_x^{-1})\bigr)\\
&=\Bigl(\bh^{\tens(n+1)}\tens\tfrac{\Tr}{\qdim{x}}\Bigr)
\bigl((a^*b\tens\I)(\Delta^{(n)}\tens\id)({u^x}^*)
(\I^{\tens(n+1)}\tens\uprho_x^{-1})
(\sigma^{\bh^{\tens(n+1)}}_{-\ii}\tens\id)(\Delta^{(n)}\tens\id)(u^x)\bigr)\\
&=\Bigl(\bh^{\tens(n+1)}\tens\tfrac{\Tr}{\qdim{x}}\Bigr)
\Bigl((a^*b\tens\I)
\bigl[{u^x}^*\bigr]_{n+1\,n+2}\dotsm\bigl[{u^x}^*\bigr]_{1\,n+2}
(\I^{\tens(n+1)}\tens\uprho_x^{-1})\\
&\qquad
\bigl[(\sigma^{\bh}_{-\ii}\tens\id)(u^x)\bigr]_{1\,n+2}
\dotsm
\bigl[(\sigma^{\bh}_{-\ii}\tens\id)(u^x)\bigr]_{n+1\,n+2}\Bigr)\\
&=\Bigl(\bh^{\tens(n+1)}\tens\tfrac{\Tr}{\qdim{x}}\Bigr)
\Bigl((a^*b\tens\I)
\bigl[{u^x}^*\bigr]_{n+1\,n+2}\dotsm\bigl[{u^x}^*\bigr]_{1\,n+2}
(\I^{\tens(n+1)}\tens\uprho_x^{-1})\\
&\qquad
(\I^{\tens(n+1)}\tens\uprho_x)
[u^x]_{1\,n+2}
\bigl((\I^{\tens(n+1)}\tens\uprho_x^2)
[u^x]_{2\,n+2}\bigr)
\dotsm
\bigl((\I^{\tens(n+1)}\tens\uprho_x^2)
[u^x]_{n+1\,n+2}\bigr)
(\I^{\tens(n+1)}\tens\uprho_x)\Bigr)\\
&=(\bh^{\tens(n+1)}\tens\omega_x)
\Bigl((a^*b\tens\I)\bigl[{u^x}^*\bigr]_{n+1\,n+2}\dotsm
\bigl[{u^x}^*\bigr]_{2\,n+2}
(\I^{\tens(n+1)}\tens\uprho_x^2)\\
&\qquad
\bigl([u^x]_{2\,n+2}
(\I^{\tens(n+1)}\tens\uprho_x^2)\bigr)
\dotsm
\bigl([u^x]_{n+1\,n+2}
(\I^{\tens(n+1)}\tens\uprho_x^2)\bigr)\Bigr).
\end{align*}
Consequently, writing $J$ for $J_{\bh^{\tens\,(n+1)}\tens\,\omega_x}$ we obtain
\begin{align*}
&\!\!\!\!
\bigl|\is{V^{\Delta,n}_x(a)}{V^{\Delta,n}_x(b)}-\is{a}{b}\bigr|\\
&=\Bigl|(\bh^{\tens(n+1)}\tens\omega_x)
\Bigl(
(a^*b\tens\I)\Bigl(
\bigl[{u^x}^*\bigr]_{n+1\,n+2}\dotsm
\bigl[{u^x}^*\bigr]_{2\,n+2}
(\I^{\tens(n+1)}\tens\uprho_x^2)\\
&\qquad\qquad\quad
\bigl([u^x]_{2\,n+2}
(\I^{\tens(n+1)}\tens\uprho_x^2)\bigr)
\dotsm
\bigl([u^x]_{n+1\,n+2}
(\I^{\tens(n+1)}\tens\uprho_x^2)\bigr)-\I^{\tens(n+1)}\tens\I
\Bigr)
\Bigr)\Bigr|\\
&=\Bigl|\Bigl\langle{a}\tens\I\Big|
(b\tens\I)\Bigl(
\bigl[{u^x}^*\bigr]_{n+1\,n+2}\dotsm
\bigl[{u^x}^*\bigr]_{2\,n+2}
(\I^{\tens(n+1)}\tens\uprho_x^2)\\
&\qquad\qquad\quad
\bigl([u^x]_{2\,n+2}
(\I^{\tens(n+1)}\tens\uprho_x^2)\bigr)
\dotsm
\bigl([u^x]_{n+1\,n+2}
(\I^{\tens(n+1)}\tens\uprho_x^2)\bigr)-\I^{\tens(n+1)}\tens\I
\Bigr)
\Bigr\rangle\Bigr|\\
&=\Bigl|\Bigl\langle{a}\tens\I\Big|
J\sigma^{\bh^{\tens(n+1)}\tens\omega_x}_{\ihalf}\Bigl(
\bigl[{u^x}^*\bigr]_{n+1\,n+2}\dotsm
\bigl[{u^x}^*\bigr]_{2\,n+2}
(\I^{\tens(n+1)}\tens\uprho_x^2)\\
&\qquad\qquad\quad
\bigl([u^x]_{2\,n+2}
(\I^{\tens(n+1)}\tens\uprho_x^2)\bigr)
\dotsm
\bigl([u^x]_{n+1\,n+2}
(\I^{\tens(n+1)}\tens\uprho_x^2)\bigr)-\I^{\tens(n+1)}\tens\I
\Bigr)^*J(b\tens\I)
\Bigr\rangle\Bigr|\\
&\leq\|a\tens\I\|_2
\Bigl\|\sigma_{\ihalf}^{\bh^{\tens(n+1)}\tens\omega_x}
\Bigl(
\bigl[{u^x}^*\bigr]_{n+1\,n+2}\dotsm
\bigl[{u^x}^*\bigr]_{2\,n+2}
(\I^{\tens(n+1)}\tens\uprho_x^2)\\
&\qquad\qquad\quad
\bigl([u^x]_{2\,n+2}
(\I^{\tens(n+1)}\tens\uprho_x^2)\bigr)
\dotsm
\bigl([u^x]_{n+1\,n+2}
(\I^{\tens(n+1)}\tens\uprho_x^2)\bigr)-\I^{\tens(n+1)}\tens\I
\Bigr)
\Bigr\|\|b\tens\I\|_2\\
&=\|a\|_2
\Bigl\|\sigma_{\ihalf}^{\bh^{\tens(n+1)}\tens\omega_x}
\Bigl(
\bigl[{u^x}^*\bigr]_{n+1\,n+2}\dotsm
\bigl[{u^x}^*\bigr]_{2\,n+2}
(\I^{\tens(n+1)}\tens\uprho_x^2)\\
&\qquad\qquad\quad
\bigl([u^x]_{2\,n+2}
(\I^{\tens(n+1)}\tens\uprho_x^2)\bigr)
\dotsm
\bigl([u^x]_{n+1\,n+2}
(\I^{\tens(n+1)}\tens\uprho_x^2)\bigr)-\I^{\tens(n+1)}\tens\I
\Bigr)
\Bigr\|\|b\|_2.
\end{align*}
This shows that $V^{\Delta,n}_x$ extends to a bounded map. Since $\sigma_t^{\bh^{\tens(n+1)}\tens\omega_x}=(\sigma^\bh_t)^{\tens(n+1)}\tens\sigma^{\omega_x}_t$ and
\[
\sigma^{\omega_x}_t(T)=\uprho_x^{-\ii{t}}T\uprho_x^{\ii{t}},\qquad{T}\in\B(\cH_x),\:t\in\RR,
\]
we can continue the calculation:
\begin{align*}
&\!\!\!\!\!\!\!\!\!\!\!\!\bigl|\is{V^{\Delta,n}_x(a)}{V^{\Delta,n}_x(b)}-\is{a}{b}\bigr|\\
&\leq\|a\|_2\|b\|_2
\Bigl\|
(\I^{\tens(n+1)}\tens\uprho_x^{\half})\\
&\qquad\Bigl(
\bigl[(\sigma^{\bh}_{-\ihalf}\tens\id)(u^x)\bigr]_{n+1\,n+2}^*\dotsm
\bigl[(\sigma^{\bh}_{-\ihalf}\tens\id)(u^x)\bigr]_{2\,n+2}^*
(\I^{\tens(n+1)}\tens\uprho_x^2)\\
&\qquad\:\:\,
\bigl[(\sigma^{\bh}_{\ihalf}\tens\id)(u^x)
(\I\tens\uprho_x^2)\bigr]_{2\,n+2}
\dotsm
\bigl[(\sigma^{\bh}_{\ihalf}\tens\id)(u^x)
(\I\tens\uprho_x^2)\bigr]_{n+1\,n+2}\\
&\qquad\qquad\qquad
\qquad\qquad\qquad
\qquad\qquad\qquad
-\I^{\tens(n+1)}\tens\I
\Bigr)(\I^{\tens(n+1)}\tens\uprho_x^{-\half})
\Bigr\|\\
&\leq\|a\|_2\|b\|_2
\Bigl\|
(\I^{\tens(n+1)}\tens\uprho_x^{\half})\\
&\qquad
\Bigl(
\bigl[(\I\tens\uprho_x^{\half})
{u^x}^*
(\I\tens\uprho_x^{\half})\bigr]_{n+1\,n+2}
\dotsm
\bigl[(\I\tens\uprho_x^{\half})
{u^x}^*
(\I\tens\uprho_x^{\half})\bigr]_{2\,n+2}
(\I^{\tens(n+1)}\tens\uprho_x^2)\\
&
\qquad\:\:\,
\bigl[
(\I\tens\uprho_x^{-\half})u^x
(\I\tens\uprho_x^{\threehalf})\bigr]_{2\,n+2}
\dotsm
\bigl[(\I\tens\uprho_x^{-\half})u^x
(\I\tens\uprho_x^{\threehalf})\bigr]_{n+1\,n+2}\\
&\qquad\qquad\qquad
\qquad\qquad\qquad
\qquad\qquad\qquad
-\I^{\tens(n+1)}\tens\I
\Bigr)(\I^{\tens(n+1)}\tens\uprho_x^{-\half})
\Bigr\|\\
&=\|a\|_2\|b\|_2\\
&\qquad\Bigl\|
\Bigl((\I^{\tens(n+1)}\tens\uprho_x)
\bigl[{u^x}^*\bigr]_{n+1\,n+2}\Bigr)\dotsm
\Bigl((\I^{\tens(n+1)}\tens\uprho_x)
\bigl[{u^x}^*\bigr]_{2\,n+2}\Bigr)
(\I^{\tens(n+1)}\tens\uprho_x^2)\\
&\qquad\qquad\qquad
\bigl([u^x]_{2\,n+2}(\I^{\tens(n+1)}\tens\uprho_x)\bigr)
\dotsm
\bigl([u^x]_{n+1\,n+2}
(\I^{\tens(n+1)}\tens\uprho_x)\bigr)
-\I^{\tens(n+1)}\tens\I
\Bigr\|.
\end{align*}
The last norm can be estimated as follows
\begin{align*}
&\!\!\!\!\!\!\!\!\Bigl\|
\Bigl((\I^{\tens(n+1)}\tens\uprho_x)
\bigl[{u^x}^*\bigr]_{n+1\,n+2}\Bigr)\dotsm
\Bigl((\I^{\tens(n+1)}\tens\uprho_x)
\bigl[{u^x}^*\bigr]_{2\,n+2}\Bigr)
(\I^{\tens(n+1)}\tens\uprho_x^2)\\
&\qquad\qquad\qquad\qquad
\bigl([u^x]_{2\,n+2}(\I^{\tens(n+1)}\tens\uprho_x)\bigr)
\dotsm
\bigl([u^x]_{n+1\,n+2}
(\I^{\tens(n+1)}\tens\uprho_x)\bigr)
-\I^{\tens(n+1)}\tens\I
\Bigr\|\\
&\leq\|\uprho_x^2-\I\|\|\uprho_x\|^{2n}+
\Bigl\|\bigl((\I^{\tens(n+1)}\tens\uprho_x)
\bigl[{u^x}^*\bigr]_{n+1\,n+2}\bigr)\dotsm
\bigl((\I^{\tens(n+1)}\tens\uprho_x)
\bigl[{u^x}^*\bigr]_{3\,n+2}\bigr)
(\I^{\tens(n+1)}\tens\uprho_x^2)\\
&\qquad\qquad\qquad\qquad
\bigl([u^x]_{3\,n+2}(\I^{\tens(n+1)}\tens\uprho_x)\bigr)
\dotsm
\bigl([u^x]_{n+1\,n+2}
(\I^{\tens(n+1)}\tens\uprho_x)\bigr)
-\I^{\tens(n+1)}\tens\I
\Bigr\|
\end{align*}
and proceeding in this way we obtain
\begin{align*}
&\!\!\!\!\!\!\!\!\Bigl\|
\Bigl((\I^{\tens(n+1)}\tens\uprho_x)
\bigl[{u^x}^*\bigr]_{n+1\,n+2}\Bigr)\dotsm
\Bigl((\I^{\tens(n+1)}\tens\uprho_x)
\bigl[{u^x}^*\bigr]_{2\,n+2}\Bigr)
(\I^{\tens(n+1)}\tens\uprho_x^2)\\
&\qquad\qquad\qquad\qquad
\bigl([u^x]_{2\,n+2}(\I^{\tens(n+1)}\tens\uprho_x)\bigr)
\dotsm
\bigl([u^x]_{n+1\,n+2}
(\I^{\tens(n+1)}\tens\uprho_x)\bigr)
-\I^{\tens(n+1)}\tens\I
\Bigr\|\\
&\leq
\|\uprho_x^2-\I\|\|\uprho_x\|^{2n}+
\dotsm+
\|\uprho_x^2-\I\|\|\uprho_x\|^2+
\bigl\|\I^{\tens(n+1)}\tens\uprho_x^2-\I^{\tens(n+1)}\tens\I\bigr\|=
\|\uprho_x^2-\I\|\tfrac{\|\uprho_x\|^{2(n+1)}-1}{\|\uprho_x\|^2-1}
\end{align*}
(where the last equality holds when $\uprho_x\neq\I$). This gives us the following bound
\begin{equation}\label{eq30-drugie}
\Bigl|\is{V^{\Delta,n}_x(a)}{V^{\Delta,n}_x(b)}-\is{a}{b}\Bigr|
\leq\|a\|_2\|b\|_2\|\uprho_x^2-\I\|\tfrac{\|\uprho_x\|^{2(n+1)}-1}{\|\uprho_x\|^2-1}
\end{equation}
for $a,b\in\Linf(\oon{U}_F^+)^{\vtens(n+1)}$, provided $\uprho_x\neq\I$. It will be useful in the proof of the next proposition based on \cite[Proposition 32]{DeCommerFreslonYamashita}. The proof is similar as in \cite{DeCommerFreslonYamashita}, but somewhat complicated by the facts that the operators $V^{\Delta,n}_x$ ($x\in\ZZ_+\star\ZZ_+$) are not isometric and we work in $(n+1)$-fold tensor product space.

Let us introduce
\[
D_{x,n}=
\begin{cases}
\|\uprho_x^2-\I\|\tfrac{\|\uprho_x\|^{2(n+1)}-1}{\|\uprho_x\|^2-1}&\uprho_x\neq\I,\\
0&\uprho_x=\I
\end{cases},
\qquad{x}\in\ZZ_+\star\ZZ_+
\]
and put
\begin{equation}\label{eq:defD}
D_n=\max\bigl\{
D_{\alpha\beta,n},
D_{\beta\alpha,n},
D_{\alpha^2\beta,n}
\bigr\}.
\end{equation}

It follows from \eqref{eq30-drugie} that $\|V^{\Delta,n}_x\|\leq(1+D_{x,n})^{\half}$. Next, for $x\in\ZZ_+\star\ZZ_+\setminus\{\eps\}$ introduce subspaces of $\Ltwo(\oon{U}_F^+)$
\begin{align*}
C_x&=\oon{span}
\bigl\{(\id\tens\omega)u^x\,\bigr|\bigl.\,\omega\in\B(\cH_x)_*\bigr\},\\
L_x&=\oon{\overline{span}}\bigcup_{z\in\ZZ_+\star\ZZ_+}C_{xz},\qquad{L_\eps}=\CC\I.
\end{align*}

Given an operator $a\in\Linf(\oon{U}_F^+)^{\vtens(n+1)}\,(n\in\NN)$ we will decompose it as follows
\[
a=\sum_{x_1,\dotsc,x_{n+1}\in\{\eps,\alpha,\beta\}}a_{x_1,\dotsc,x_{n+1}}
\]
with respect to the decomposition
\[
\Ltwo(\oon{U}_F^+)^{\tens(n+1)}=\bigoplus_{x_1,\dotsc,x_{n+1}\in\{\eps,\alpha,\beta\}}L_{x_1}\tens\dotsm\tens{L_{x_{n+1}}}.
\]

\begin{proposition}\label{prop3-drugie}
Take $n\in\NN$ and $a\in\Linf(\oon{U}_F^+)^{\vtens(n+1)}$. If $D_n<1-\tfrac{1}{\sqrt{2}}$, then for any $1\leq{k}\leq{n}$ we have
\begin{align*}
\biggl\|a-&\sum_{x_i\in\{\eps,\alpha,\beta\},\:i\neq{k}}
a_{x_1,\dotsc,x_{k-1},\eps,x_{k+1},\dotsc,x_{n+1}}
\biggr\|_2\\
&\leq
\tfrac{2(7-4D_n)}{2(1-D_n)^2-1}
\Bigl(D_n\|a\|_2+(1+D_n)^{\half}
\max_{x\in\{\alpha\beta,\beta\alpha,\alpha^2\beta\}}
\bigl\|a\tens\I-V^{\Delta,n}_{x}(a)\bigr\|_2
\Bigr).
\end{align*}
\end{proposition}

\begin{proof}
Without loss of generality assume that $\|a\|_2=1$ and set
\[
M=\max_{x\in\{\alpha\beta,\beta\alpha,\alpha^2\beta\}}
\bigl\|a\tens\I-V^{\Delta,n}_{x}(a)\bigr\|_2.
\]

Observe that for $x,y\in\ZZ_+\star\ZZ_+$
\begin{align*}
V^{\Delta,n}_x\bigl(
\Ltwo(\oon{U}_F^+)^{\tens(k-1)}&\tens{C_y}\tens\Ltwo(\oon{U}_F^+)^{\tens(n+1-k)}\bigr)\\
&\subset
\oon{\overline{span}}\bigcup_{z\subset{x}\tp{y}\tp\overline{x}}
\Ltwo(\oon{U}_F^+)^{\tens(k-1)}\tens{C_z}\tens\Ltwo(\oon{U}_F^+)^{\tens(n+1-k)}
\tens\cK_x
\end{align*}
and consequently
\begin{equation}\label{eq32-drugie}
\begin{aligned}
V^{\Delta,n}_{\beta\alpha}\bigl(
\Ltwo(\oon{U}_F^+)^{\tens(k-1)}\tens{L_\alpha}\tens\Ltwo(\oon{U}_F^+)^{\tens(n+1-k)}\bigr)
&\subset
\Ltwo(\oon{U}_F^+)^{\tens(k-1)}\tens{L_{\beta\alpha}}\tens\Ltwo(\oon{U}_F^+)^{\tens(n+1-k)}\tens\cK_{\beta\alpha},\\
V^{\Delta,n}_{\alpha\beta}\bigl(
\Ltwo(\oon{U}_F^+)^{\tens(k-1)}\tens{L_\beta}\tens\Ltwo(\oon{U}_F^+)^{\tens(n+1-k)}\bigr)
&\subset
\Ltwo(\oon{U}_F^+)^{\tens(k-1)}\tens{L_{\alpha\beta}}\tens\Ltwo(\oon{U}_F^+)^{\tens(n+1-k)}\tens\cK_{\alpha\beta},\\
V^{\Delta,n}_{\alpha^2\beta}\bigl(
\Ltwo(\oon{U}_F^+)^{\tens(k-1)}\tens{L_\beta}\tens\Ltwo(\oon{U}_F^+)^{\tens(n+1-k)}\bigr)
&\subset
\Ltwo(\oon{U}_F^+)^{\tens(k-1)}\tens{L_{\alpha^2}}\tens\Ltwo(\oon{U}_F^+)^{\tens(n+1-k)}
\tens\cK_{\alpha^2\beta}.
\end{aligned}
\end{equation}

Indeed, if $w=\alpha{y}$ for some $y\in\ZZ_+\star\ZZ_+$, then using fusion rules of $\oon{U}_F^+$ (equation \eqref{eq41}) we have $\beta\alpha\tp{w}\tp\overline{\beta\alpha}=
\beta\alpha^2y\tp\beta\alpha$ and any $z\subset\beta\alpha\tp{w}\tp\overline{\beta\alpha}$ has to begin with $\beta\alpha$. This shows the first inclusion in \eqref{eq32-drugie}, the remaining ones can be argued analogously.

To ease the notation, for $x\in\{\eps,\alpha,\beta\}$ let us introduce
\[
b_x=\sum_{x_i\in\{\eps,\alpha,\beta\},\:i\neq{k}}a_{x_1,\dotsc,x_{k-1},x,x_{k+1},\dotsc,x_{n+1}}\in\Ltwo(\oon{U}_F^+)^{\tens(k-1)}\tens{L_x}\tens\Ltwo(\oon{U}_F^+)^{\tens(n+1-k)},
\]
so that $a=b_\eps+b_\alpha+b_\beta$. Using \eqref{eq32-drugie} and \eqref{eq30-drugie}, we obtain the estimate
\begin{subequations}
\begin{equation}\label{eq34-drugie}
\begin{aligned}
&\!\!\!\!\bigl|\bigl\langle{b_\beta}\tens\I-V^{\Delta,n}_{\beta\alpha}(b_\alpha)\big|V^{\Delta,n}_{\beta\alpha}(b_\alpha)\bigr\rangle\bigr|\\
&\leq
\bigl|\bigl\langle{a}\tens\I-V^{\Delta,n}_{\beta\alpha}(a)\big|V^{\Delta,n}_{\beta\alpha}(b_\alpha)\bigr\rangle\bigr|
+\bigl|\bigl\langle(b_\eps+b_\alpha)\tens\I-V^{\Delta,n}_{\beta\alpha}(b_\eps+b_\beta)\big|V^{\Delta,n}_{\beta\alpha}(b_\alpha)\bigr\rangle\bigr|\\
&=\bigl|\bigl\langle{a}\tens\I-V^{\Delta,n}_{\beta\alpha}(a)\big|V^{\Delta,n}_{\beta\alpha}(b_\alpha)\bigr\rangle\bigr|
+\bigl|\bigl\langle{V^{\Delta,n}_{\beta\alpha}}(b_\eps+b_\beta)\big|V^{\Delta,n}_{\beta\alpha}(b_\alpha)\bigr\rangle\bigr|\\
&\leq\bigl\|a\tens\I-V^{\Delta,n}_{\beta\alpha}(a)\bigr\|_2
\bigl\|V^{\Delta,n}_{\beta\alpha}(b_\alpha)\bigr\|
+D_{\beta\alpha,n}\|b_\eps+b_\beta\|_2\|b_\alpha\|_2\\
&\leq{M}(1+D_{\beta\alpha,n})^{\half}\|b_\alpha\|_2
+D_{\beta\alpha,n}\|b_\eps+b_\beta\|_2\|b_\alpha\|_2\\
&\leq{M}(1+D_n)^{\half}\|a-b_\eps\|_2+D_n\|a-b_\eps\|_2
\end{aligned}
\end{equation}
and similarly to \eqref{eq34-drugie} we derive
\begin{equation}\label{eq34-trzecie}
\bigl|\bigl\langle
b_\alpha\tens\I-V^{\Delta,n}_{\alpha\beta}(b_\beta)\big|
V^{\Delta,n}_{\alpha\beta}(b_\beta)\bigr\rangle\bigr|
\leq{M}(1+D_n)^{\half}\|a-b_\eps\|_2+D_n\|a-b_\eps\|_2
\end{equation}
as well as
\begin{equation}\label{eq34-czwarte}
\bigl|\bigl\langle{b_\alpha}\tens\I-V^{\Delta,n}_{\alpha^2\beta}(b_\beta)\big|
V^{\Delta,n}_{\alpha^2\beta}(b_\beta)\bigr\rangle\bigr|
\leq{M}(1+D_n)^{\half}\|a-b_\eps\|_2+D_n\|a-b_\eps\|_2.
\end{equation}
\end{subequations}

Using \eqref{eq34-drugie} and again \eqref{eq30-drugie} we arrive at
\begin{equation}\label{eq33-drugie}
\begin{aligned}
\|b_\beta\|_2^2
&=\bigl\|b_\beta\tens\I-V^{\Delta,n}_{\beta\alpha}(b_\alpha)
+V^{\Delta,n}_{\beta\alpha}(b_\alpha)\bigr\|^2_2\\
&\geq
\bigl\|b_\beta\tens\I-V^{\Delta,n}_{\beta\alpha}(b_\alpha)\bigr\|^2
+\bigl\|V^{\Delta,n}_{\beta\alpha}(b_\alpha)\bigr\|^2_2
-2\bigl|\bigl\langle{b_\beta}\tens\I-V^{\Delta,n}_{\beta\alpha}(b_\alpha)\big|
V^{\Delta,n}_{\beta\alpha}(b_\alpha)\bigr\rangle\bigr|\\
&\geq(1-D_n)\|b_\alpha\|_2^2
-2\bigl|\bigl\langle{b_\beta}\tens\I-
V^{\Delta,n}_{\beta\alpha}(b_\alpha)\big|
V^{\Delta,n}_{\beta\alpha}(b_\alpha)\bigr\rangle\bigr|\\
&\geq(1-D_n)\|b_\alpha\|_2^2
-2\bigl(M(1+D_n)^{\half}+D_n\bigr)\|a-b_\eps\|_2.
\end{aligned}
\end{equation}

Define bounded operators
\begin{align*}
\cV^{\Delta,n}_{\alpha\beta,12}&\colon\Ltwo(\oon{U}_F^+)^{\tens(n+1)}\ni\xi\longmapsto\bigl(V^{\Delta,n}_{\alpha\beta}\xi\bigr)_{12}
\in\Ltwo(\oon{U}_F^+)^{\tens(n+1)}\tens\cK_{\alpha\beta}\tens\cK_{\alpha^2\beta},\\
\cV^{\Delta,n}_{\alpha^2\beta,13}&\colon\Ltwo(\oon{U}_F^+)^{\tens(n+1)}\ni\xi\longmapsto\bigl(V^{\Delta,n}_{\alpha^2\beta}\xi\bigr)_{13}
\in\Ltwo(\oon{U}_F^+)^{\tens(n+1)}\tens\cK_{\alpha\beta}\tens\cK_{\alpha^2\beta}.
\end{align*}

It follows from \eqref{eq32-drugie} that the images of $\Ltwo(\oon{U}_F^+)^{\tens(k-1)}\tens{L_\beta}\tens\Ltwo(\oon{U}_F^+)^{\tens(n+1-k)}$ under $\cV^{\Delta,n}_{\alpha\beta,12}$ and $\cV^{\Delta,n}_{\alpha^2\beta,13}$ are orthogonal. Using \eqref{eq34-trzecie} we obtain
\begin{align*}
&\quad\;
\bigl|\bigl\langle{b_\alpha}\tens\I\tens\I
-\cV^{\Delta,n}_{\alpha\beta,12}(b_\beta)-\cV^{\Delta,n}_{\alpha^2\beta,13}(b_\beta)
\big|\cV^{\Delta,n}_{\alpha\beta,12}(b_\beta)\bigr|
=\bigl|\bigl\langle
b_\alpha\tens\I\tens\I
-\cV^{\Delta,n}_{\alpha\beta,12}(b_\beta)\big|
\cV^{\Delta,n}_{\alpha\beta,12}(b_\beta)\bigr\rangle\bigr|\\
&=\bigl|\bigl\langle{b_\alpha}\tens\I
-V^{\Delta,n}_{\alpha\beta}(b_\beta)\big|
V^{\Delta,n}_{\alpha\beta}(b_\beta)\bigr\rangle\bigr|\leq\bigl(M(1+D_n)^{\half}+D_n\bigr)\|a-b_\eps\|_2.
\end{align*}
Analogously, using \eqref{eq34-czwarte} we get
\begin{align*}
\bigl|\bigl\langle{b_\alpha}\tens\I\tens\I
-\cV^{\Delta,n}_{\alpha\beta,12}(b_\beta)-
\cV^{\Delta,n}_{\alpha^2\beta,13}(b_\beta)\big|
\cV^{\Delta,n}_{\alpha^2\beta,13}(b_\beta)\bigr\rangle\bigr|
&=\bigl|\bigl\langle{b_\alpha}\tens\I
-V^{\Delta,n}_{\alpha^2\beta}(b_\beta)\big|
V^{\Delta,n}_{\alpha^2\beta}(b_\beta)\bigr\rangle\bigr|\\
&\leq\bigl(M(1+D_n)^{\half}+D_n\bigr)\|a-b_\eps\|_2.
\end{align*}
Thus
\begin{equation}\label{eq37-drugie}
\begin{aligned}
\|b_\alpha\|_2^2
&=\bigl\|\bigl(b_\alpha\tens\I\tens\I
-\cV^{\Delta,n}_{\alpha\beta,12}(b_\beta)-\cV^{\Delta,n}_{\alpha^2\beta,13}(b_\beta)\bigr)
+\cV^{\Delta,n}_{\alpha\beta,12}(b_\beta)+\cV^{\Delta,n}_{\alpha^2\beta,13}(b_\beta)\bigr\|_2^2\\
&\geq\bigl\|b_\alpha\tens\I\tens\I
-\cV^{\Delta,n}_{\alpha\beta,12}(b_\beta)-\cV^{\Delta,n}_{\alpha^2\beta,13}(b_\beta)\bigr\|_2^2
+2(1-D_n)\|b_\beta\|^2_2\\
&\qquad\qquad\qquad\qquad\qquad\qquad\qquad\qquad\qquad-4\bigl(M(1+D_n)^{\half}+D_n\bigr)\|a-b_\eps\|_2\\
&\geq2(1-D_n)\|b_\beta\|^2_2-4\bigl(M(1+D_n)^{\half}+D_n\bigr)\|a-b_\eps\|_2.
\end{aligned}
\end{equation}
Now we combine inequalities \eqref{eq33-drugie} and \eqref{eq37-drugie}:
\begin{align*}
\|b_\alpha\|^2_2
&\geq2(1-D_n)\Bigl((1-D_n)\|b_\alpha\|_2^2
-2\bigl(M(1+D_n)^{\half}+D_n\bigr)\|a-b_\eps\|_2\Bigr)
-4\bigl(M(1+D_n)^{\half}+D_n\bigr)\|a-b_\eps\|_2\\
&=2(1-D_n)^2\|b_\alpha\|_2^2-4(2-D_n)\bigl(M(1+D_n)^{\half}+D_n\bigr)\|a-b_\eps\|_2
\end{align*}
and as a consequence obtain
\begin{equation}\label{eq38}
\|b_\alpha\|_2^2\leq
\tfrac{4(2-D_n)}{2(1-D_n)^2-1}\bigl(M(1+D_n)^{\half}+D_n\bigr)\|a-b_\eps\|_2.
\end{equation}
Similarly
\begin{align*}
\|b_\beta\|_2^2&\geq(1-D_n)\Bigl(2(1-D_n)\|b_\beta\|_2^2
-4\bigl(M(1+D_n)^{\half}+D_n\bigr)\|a-b_\eps\|_2\Bigr)
-2\bigl(M(1+D_n)^{\half}+D_n\bigr)\|a-b_\eps\|_2\\
&=2(1-D_n)^2\|b_{\beta}\|_2^2
-\bigl(4(1-D_n)+2\bigr)\bigl(M(1+D_n)^{\half}+D_n\bigr)\|a-b_\eps\|_2
\end{align*}
leads to
\begin{equation}\label{eq39}
\|b_\beta\|_2^2\leq\tfrac{4(1-D_n)+2}{2(1-D_n)^2-1}
\bigl(M(1+D_n)^{\half}+D_n\bigr)\|a-b_\eps\|_2.
\end{equation}

Inequalities \eqref{eq38}, \eqref{eq39} combine to
\begin{align*}
\|a-b_\eps\|^2_2&=\|b_\alpha+b_\beta\|_2^2=
\|b_\alpha\|_2^2+\|b_\beta\|_2^2\\
&\leq
\Bigl(\tfrac{4(2-D_n)}{2(1-D_n)^2-1}+\tfrac{4(1-D_n)+2}{2(1-D_n)^2-1}\Bigr)
\bigl(M(1+D_n)^{\half}+D_n\bigr)\|a-b_\eps\|_2\\
&=\tfrac{2(7-4D_n)}{2(1-D_n)^2-1}\bigl(M(1+D_n)^{\half}+D_n\bigr)
\|a-b_\eps\|_2
\end{align*}
which ends the proof.
\end{proof}

As a corollary we obtain

\begin{theorem}\label{thm2-drugie}
Take $n\in\NN$ and assume that the constant $D_n$ defined by \eqref{eq:defD} satisfies $D_n<1-\tfrac{1}{\sqrt{2}}$ and $\tfrac{2(7-4D_n)D_n}{2(1-D_n)^2-1}<\tfrac{1}{\sqrt{n+1}}$. Then $\hh{\oon{U}_F^+}$ is $n$-i.c.c., i.e.~$\Delta^{(n)}(\Linf(\oon{U}_F^+))'\cap\Linf(\oon{U}_F^+)^{\vtens(n+1)}=\CC\I^{\tens(n+1)}$.
\end{theorem}

The conditions $D_n<1-\tfrac{1}{\sqrt{2}}$, $\tfrac{2(7-4D_n)D_n}{2(1-D_n)^2-1}<\tfrac{1}{\sqrt{n+1}}$ can be translated to conditions on matrix $F$. Furthermore, $D_n=D_{\alpha^2\beta,n}$, see Corollary \ref{cor2}.

\begin{proof}[Proof of Theorem \ref{thm2-drugie}]
Take a non-zero $c\in\Delta^{(n)}(\Linf(\oon{U}_F^+))'\cap\Linf(\oon{U}_F^+)^{\vtens(n+1)}$ and assume that $c\not\in\CC\I^{\tens(n+1)}$. Setting
\[
a=\frac{c-\bh^{\tens(n+1)}(c)\,\I^{\tens(n+1)}}
{\|c-\bh^{\tens(n+1)}(c)\,\I^{\tens(n+1)}\|_2}
\]
we obtain $a\in\Delta^{(n)}(\Linf(\oon{U}_F^+))'\cap\Linf(\oon{U}_F^+)^{\vtens(n+1)}$, $\|a\|_2=1$ and $\bh^{\tens(n+1)}(a)=0$. Write as in the proof of Proposition \ref{prop3-drugie}
\begin{equation}\label{eq42}
a=\sum_{x_1,\dotsc,x_{n+1}\in\{\eps,\alpha,\beta\}}a_{x_1,\dotsc,x_{n+1}}.
\end{equation}
and
\[
b_{k,x}=
\sum_{x_i\in\{\eps,\alpha,\beta\},\:i\neq{k}}a_{x_1,\dotsc,x_{k-1},x,x_{k+1},\dotsc,x_{n+1}},\qquad{k}\in\{1,\dotsc,n+1\},\;x\in\{\eps,\alpha,\beta\}.
\]
Since $a$ belongs to $\Delta^{(n)}(\Linf(\oon{U}_F^+))'$ and $\|a\|_2=1$, Proposition \ref{prop3-drugie} gives us
\[
\|b_{k,\alpha}\|_2^2+\|b_{k,\beta}\|_2^2=\|b_{k,\alpha}+b_{k,\beta}\|_2^2\leq
\bigl(\tfrac{2(7-4D_n)}{2(1-D_n)^2-1}D_n\bigr)^2
\]
for $1\leq{k}\leq{n+1}$, hence as $1=\|a\|_2^2=\sum_{x\in\{\eps,\alpha,\beta\}}\|b_{k,x}\|_2^2$ we obtain
\[
\|b_{k,\eps}\|_2^2\geq{1}-\Bigl(\tfrac{2(7-4D_n)}{2(1-D_n)^2-1}D_n\Bigr)^2
\]
and consequently
\begin{equation}\label{eq43}
\sum_{k=1}^{n+1}\|b_{k,\eps}\|_2^2\geq(n+1)\biggl(
1-\Bigl(\tfrac{2(7-4D_n)}{2(1-D_n)^2-1}D_n\Bigr)^2\biggr).
\end{equation}
On the other hand, since the decomposition in \eqref{eq42} is orthogonal
\[
\begin{aligned}
\sum_{k=1}^{n+1}\|b_{k,\eps}\|_2^2&
=\sum_{k=1}^{n+1}
\biggl\|\sum_{x_i\in\{\eps,\alpha,\beta\},\:i\neq{k}}a_{x_1,\dotsc,x_{k-1},\eps,x_{k+1},\dotsc,x_{n+1}}\biggr\|_2^2\\
&=\sum_{k=1}^{n+1}
\sum_{x_i\in\{\eps,\alpha,\beta\},\:i\neq{k}}
\bigl\|a_{x_1,\dotsc,x_{k-1},\eps,x_{k+1},\dotsc,x_{n+1}}\bigr\|_2^2\\
&\leq
n\sum_{x_1,\dotsc,x_{n+1}\in\{\eps,\alpha,\beta\}}\|a_{x_1,\dotsc,x_{n+1}}\|_2^2=n.
\end{aligned}
\]
The reason why in the above inequality we have number $n$, and not $n+1$ is that every term $\|a_{x_1,\dotsc,x_{n+1}}\|_2^2$ appears at most $n$ times, unless $(x_1,\dotsc,x_{n+1})=(\eps,\dotsc,\eps)$, but $a_{\eps,\dotsc,\eps}=0$. Together with inequality \eqref{eq43} and assumption $\tfrac{2(7-4D_n)D_n}{2(1-D_n)^2-1}<\tfrac{1}{\sqrt{n+1}}$ it gives a contradiction because
\[
(n+1)\biggl(1-
\Bigl(\tfrac{2(7-4D_n)}{2(1-D_n)^2-1}D_n\Bigr)^2\biggr)\leq
\sum_{k=1}^{n+1}\|b_{k,\eps}\|_2^2\leq{n}
\]
implies $\tfrac{1}{\sqrt{n+1}}\leq\tfrac{2(7-4D_n)}{2(1-D_n)^2-1}D_n$.
\end{proof}

In the last result of this section we obtain a sufficient condition under which $\widehat{\oon{U}_F^+}$ is $n$-i.c.c. We note that one can easily obtain stronger (although less transparent) conditions using \eqref{eq46}.

\begin{lemma}\label{lemma8}
For $n\in\NN$ we have $D_n=D_{\alpha^2\beta,n}$.
\end{lemma}

\begin{corollary}\label{cor2}
Take $n\in\NN$ and write $\mathfrak{c}=\max\bigl\{\|\lambda{F^*F}-\I\|,\|(\lambda{F^*F})^{-1}-\I\|\bigr\}$,
where $\lambda=\sqrt{\tfrac{\Tr((F^*F)^{-1})}{\Tr(F^*F)}}$. If
\[
\sqrt{n}(n+1)\mathfrak{c}(2+\mathfrak{c})(1+\mathfrak{c})^{4+6n}<\tfrac{1}{72},
\]
then $\widehat{\oon{U}_F^+}$ is $n$-i.c.c. In particular $\widehat{\oon{U}_N^+}$ is $n$-i.c.c.~for any $N\geq{2}$.
\end{corollary}

\begin{proof}[Proof of Lemma \ref{lemma8} and Corollary \ref{cor2}]

Recall that by definition
\[
D_{n}=\max\bigl\{D_{\alpha\beta,n},D_{\beta\alpha,n},
D_{\alpha^2\beta,n}\bigr\},
\]
where
\[
D_{x,n}=\begin{cases}\|\uprho_x^2-\I\|\tfrac{
\|\uprho_x\|^{2(n+1)}-1}{\|\uprho_x\|^2-1}&\uprho_x\neq\I\\
0&\uprho_x=\I
\end{cases},\qquad{x}\in\ZZ_+\star\ZZ_+.
\]

The fusion rules of $\oon{U}_F^+$ (equation \eqref{eq41}) show $\alpha\tp\beta=\alpha\beta\oplus\eps$ and $\beta\tp\alpha=\beta\alpha\oplus\eps$, hence using \cite[Theorem 1.4.9]{NeshveyevTuset} we obtain
\[
\|\uprho_{\alpha\beta}\|=\|\uprho_{\beta\alpha}\|,\quad
\|\uprho_{\alpha\beta}^2-\I\|=
\|\uprho_{\beta\alpha}^2-\I\|.
\]
Next, $\alpha\tp\alpha\beta=\alpha^2\beta$, hence writing $\simeq$ for unitary equivalence, we have $\uprho_\alpha\tens\uprho_{\alpha\beta}\simeq\uprho_{\alpha^2\beta}$ and
\[
\|\uprho_{\alpha^2\beta}\|=\|\uprho_\alpha\|\|\uprho_{\alpha\beta}\|\geq\|\uprho_{\alpha\beta}\|
\]
and
\[
\|\uprho_{\alpha^2\beta}^2-\I\|
=\max\bigl\{|\mu|\,\bigr|\bigl.\,\mu\in\oon{Sp}(\uprho_{\alpha^2\beta}^2-\I)\bigr\}
=\max\bigl\{|\mu-1|\,\bigr|\bigl.\,\mu\in\oon{Sp}(\uprho_\alpha^2\tens\uprho_{\alpha\beta}^2)\bigr\}
\geq\|\uprho_{\alpha\beta}^2-\I\|
\]
(note that $\max\oon{Sp}(\uprho_\alpha)\geq{1},\min\oon{Sp}(\uprho_\alpha)\leq{1}$). Consequently
\[
D_n=D_{\alpha^2\beta,n}.
\]

To use Theorem \ref{thm2-drugie}, we need to check $D_n<1-\tfrac{1}{\sqrt{2}}$ and $\tfrac{2(7-4D_n)D_n}{2(1-D_n)^2-1}<\tfrac{1}{\sqrt{n+1}}$. Clearly these conditions are met if $\uprho_\alpha=\I$ (as then $D_n=0$), hence from now on assume otherwise.

As mentioned after Proposition \ref{prop:Vaes33}, by \cite[Examples 1.4.2]{NeshveyevTuset} we have $\uprho_\alpha=\lambda(F^*F)^\top$, hence by \cite[Proposition 1.4.7]{NeshveyevTuset}, $\uprho_\beta=\lambda^{-1}(F^*F)^{-1}$. Since $\alpha\tp\alpha\tp\beta=\alpha^2\beta\oplus\alpha$, we have $\uprho_\alpha\tens\uprho_\alpha\tens\uprho_\beta\simeq\uprho_{\alpha^2\beta}\oplus\uprho_\alpha$ and
\begin{equation}\label{eq46}
\|\uprho_{\alpha^2\beta}\|=\lambda\|F^*F\|^2\|(F^*F)^{-1}\|,\quad
\|\uprho_{\alpha^2\beta}^2-\I\|=
\|\lambda^2(F^*F)^2\tens(F^*F)^2\tens(F^*F)^{-2}-\I\tens\I\tens\I\|.
\end{equation}
Writing $\mathfrak{c}=\max\bigl\{\|\lambda{F^*F}-\I\|,\bigl\|(\lambda{F^*F})^{-1}-\I\bigr\|\bigr\}$ we have $\max\bigl\{\|\lambda{F^*F}\|,\bigl\|(\lambda{F^*F})^{-1}\bigr\|\bigr\}\leq{1}+\mathfrak{c}$ and $\|\uprho_{\alpha^2\beta}\|\leq(1+\mathfrak{c})^3$. Next
\begin{align*}
\bigl\|\uprho_{\alpha^2\beta}^2-\I\bigr\|
&=\bigl\|(\lambda{F^*F})^2-\I\bigr\|\bigl\|(F^*F)^2\bigr\|\bigl\|(F^*F)^{-2}\bigr\|+
\|(F^*F)^2\tens(F^*F)^{-2}-\I\tens\I\|\\
&\leq\bigl\|(\lambda{F^*F})^2-\I\bigr\|\|F^*F\|^2\bigl\|(F^*F)^{-1}\bigr\|^2+
\bigl\|(\lambda{F^*F})^2-\I\bigr\|\bigl\|(\lambda{F^*F})^{-1}\bigr\|^2+
\bigl\|(\lambda{F^*F})^{-2}-\I\bigr\|\\
&=
\bigl\|(\lambda{F^*F})^2-\I\bigr\|\bigl\|(\lambda{F^*F})^{-1}\bigr\|^2\bigl(
\bigl\|\lambda{F^*F}\bigr\|^2+1\bigr)+
\bigl\|(\lambda{F^*F})^{-2}-\I\bigr\|.
\end{align*}
Hence, since $\bigl\|(\lambda{F^*F})^2-\I\bigr\|=\bigl\|(\lambda{F^*F}+\I)(\lambda{F^*F}-\I)\bigr\|\leq(2+\mathfrak{c})\mathfrak{c}$ (and similarly $\bigl\|(\lambda{F^*F})^{-2}-\I\bigr\|\leq(2+\mathfrak{c})\mathfrak{c}$), we have
\[
\|\uprho_{\alpha^2\beta}^2-\I\|\leq
(2+\mathfrak{c})\mathfrak{c}(1+\mathfrak{c})^2\bigl((1+\mathfrak{c})^2+1\bigr)+
(2+\mathfrak{c})\mathfrak{c}=\mathfrak{c}(2+\mathfrak{c})\bigl((1+\mathfrak{c})^4+(1+\mathfrak{c})^2+1\bigr)\leq
3\mathfrak{c}(2+\mathfrak{c})(1+\mathfrak{c})^4.
\]
It follows that
\begin{align*}
D_n&=
\|\uprho_{\alpha^2\beta}^2-\I\|
\frac{\|\uprho_{\alpha^2\beta}\|^{2(n+1)}-1}
{\|\uprho_{\alpha^2\beta}\|^2-1}\leq
3\mathfrak{c}(2+\mathfrak{c})(1+\mathfrak{c})^4
\tfrac{(1+\mathfrak{c})^{6(n+1)}-1}{(1+\mathfrak{c})^6-1}\\
&=3\mathfrak{c}(2+\mathfrak{c})(1+\mathfrak{c})^4\sum_{k=0}^{n}(1+\mathfrak{c})^{6k}
\leq3(n+1)\mathfrak{c}(2+\mathfrak{c})(1+\mathfrak{c})^{4+6n}<1-\tfrac{1}{\sqrt{2}}
\end{align*}
by our assumption. Next we need to check that $\tfrac{2(7-4D_n)D_n}{2(1-D_n)^2-1}<\tfrac{1}{\sqrt{n+1}}$, equivalently
\[
2(1+4\sqrt{n+1})D_n^2-2(2+7\sqrt{n+1})D_n+1>0.
\]
An elementary argument shows that this inequality is satisfied when
\begin{equation}\label{eq47}
0<D_n<\tfrac{2+7\sqrt{n+1}-\sqrt{49{n}+20\sqrt{n+1}+51}}{2(1+4\sqrt{n+1})}.
\end{equation}
Since
\[
\tfrac{2+7\sqrt{n+1}-\sqrt{49n+20\sqrt{n+1}+51}}{2(1+4\sqrt{n+1})}
=\tfrac{1}{7\sqrt{n}}\,\bigl(\tfrac{2}{7\sqrt{n}}+\sqrt{1+1/n}
+\sqrt{1+\tfrac{20}{49\sqrt{n}}\sqrt{1+1/n}+\tfrac{51}{49n}}
\bigr)^{-1}\geq\tfrac{1}{24\sqrt{n}}
\]
and $D_n\leq3(n+1)\mathfrak{c}(2+\mathfrak{c})(1+\mathfrak{c})^{4+6n}<\tfrac{1}{24\sqrt{n}}$ by assumption, inequality \eqref{eq47} holds and the claim follows.
\end{proof}


\begin{thebibliography}{10}

\bibitem{abe}
Eiichi Abe.
\newblock {\em Hopf algebras}, volume~74 of {\em Cambridge Tracts in
  Mathematics}.
\newblock Cambridge University Press, Cambridge-New York, 1980.
\newblock Translated from the Japanese by Hisae Kinoshita and Hiroko Tanaka.

\bibitem{baaj92}
Saad Baaj.
\newblock Repr\'{e}sentation r\'{e}guli\`ere du groupe quantique {$E_\mu(2)$}
  de {W}oronowicz.
\newblock {\em C. R. Acad. Sci. Paris S\'{e}r. I Math.}, 314(13):1021--1026,
  1992.

\bibitem{banica}
Teodor Banica.
\newblock Le groupe quantique compact libre {${\rm U}(n)$}.
\newblock {\em Comm. Math. Phys.}, 190(1):143--172, 1997.

\bibitem{Bump}
Daniel Bump.
\newblock {\em Lie groups}, volume 225 of {\em Graduate Texts in Mathematics}.
\newblock Springer, New York, second edition, 2013.

\bibitem{CaspersKoelink}
Martijn Caspers and Erik Koelink.
\newblock Modular properties of matrix coefficients of corepresentations of a
  locally compact quantum group.
\newblock {\em J. Lie Theory}, 21(4):905--928, 2011.

\bibitem{typeI}
Alexandru Chirvasitu, Jacek Krajczok, and Piotr~M. So{\l}tan.
\newblock Compact quantum group structures on type-{I} {$\rm C^*$}-algebras.
\newblock {\em J. Noncommut. Geom.}, 17(3):1129--1143, 2023.

\bibitem{ConnesIII1}
Alain Connes.
\newblock Almost periodic states and factors of type {${\rm III}_{1}$}.
\newblock {\em J. Functional Analysis}, 16:415--445, 1974.

\bibitem{Connes}
Alain Connes.
\newblock The {T}omita-{T}akesaki theory and classification of type-{III}
  factors.
\newblock In {\em {$C\sp*$}-algebras and their applications to statistical
  mechanics and quantum field theory ({P}roc. {I}nternat. {S}chool of {P}hysics
  ``{E}nrico {F}ermi'', {C}ourse {LX}, {V}arenna, 1973)}, pages 29--46, 1976.

\bibitem{DeCommer}
Kenny De~Commer.
\newblock {$I$}-factorial quantum torsors and {H}eisenberg algebras of
  quantized universal enveloping type.
\newblock {\em J. Funct. Anal.}, 274(1):152--221, 2018.

\bibitem{DeCommerFreslonYamashita}
Kenny De~Commer, Amaury Freslon, and Makoto Yamashita.
\newblock C{CAP} for universal discrete quantum groups.
\newblock {\em Comm. Math. Phys.}, 331(2):677--701, 2014.
\newblock With an appendix by Stefaan Vaes.

\bibitem{DixmierC}
Jaques Dixmier.
\newblock {\em {$C\sp*$}-algebras}.
\newblock North-Holland Publishing Co., Amsterdam-New York-Oxford, 1977.

\bibitem{DixmiervNA}
Jaques Dixmier.
\newblock {\em Von {N}eumann algebras}, volume~27 of {\em North-Holland
  Mathematical Library}.
\newblock North-Holland Publishing Co., Amsterdam-New York, 1981.

\bibitem{Folland}
Gerald~B. Folland.
\newblock {\em A course in abstract harmonic analysis}.
\newblock Textbooks in Mathematics. CRC Press, Boca Raton, FL, second edition,
  2016.

\bibitem{Humphreys}
James~E. Humphreys.
\newblock {\em Introduction to {L}ie algebras and representation theory},
  volume~9 of {\em Graduate Texts in Mathematics}.
\newblock Springer-Verlag, New York-Berlin, 1978.
\newblock Second printing, revised.

\bibitem{Jacobs}
Arnoud Jacobs.
\newblock {\em {T}he quantum {$E(2)$} group}.
\newblock 2005.
\newblock Thesis (Ph.D.)--Katholieke Universiteit Leuven.

\bibitem{Kallman}
Robert~R. Kallman.
\newblock A generalization of free action.
\newblock {\em Duke Math. J.}, 36:781--789, 1969.

\bibitem{KlimykSchmudgen}
Anatoli Klimyk and Konrad Schm\"{u}dgen.
\newblock {\em Quantum groups and their representations}.
\newblock Texts and Monographs in Physics. Springer-Verlag, Berlin, 1997.

\bibitem{KorogodskiSoibelman}
Leonid~I. Korogodski and Yan~S. Soibelman.
\newblock {\em Algebras of functions on quantum groups. {P}art {I}}, volume~56
  of {\em Mathematical Surveys and Monographs}.
\newblock American Mathematical Society, Providence, RI, 1998.

\bibitem{symmetry}
Jacek Krajczok.
\newblock Symmetry of eigenvalues of operators associated with representations
  of compact quantum groups.
\newblock {\em Colloq. Math.}, 156(2):267--272, 2019.

\bibitem{KrajczokTypeI}
Jacek Krajczok.
\newblock Type {I} locally compact quantum groups: integral characters and
  coamenability.
\newblock {\em Dissertationes Math.}, 561:151, 2021.

\bibitem{modular}
Jacek Krajczok.
\newblock Modular properties of type {I} locally compact quantum groups.
\newblock {\em J. Operator Theory}, 87(2):319--354, 2022.

\bibitem{KrajczokSoltanDyskArxiv}
Jacek {Krajczok} and Piotr~M. {So{\l}tan}.
\newblock {The quantum disk is not a quantum group}.
\newblock {\em arXiv e-prints}, page arXiv:2005.02967, 2020.

\bibitem{faktory}
Jacek {Krajczok} and Piotr~M. {So{\l}tan}.
\newblock {Examples of compact quantum groups with
  $\operatorname{L}^{\!\infty}(\mathbb{G})$ a factor}.
\newblock {\em arXiv e-prints}, page arXiv:2203.10976, 2022.

\bibitem{qdisk}
Jacek Krajczok and Piotr~M. So{\l}tan.
\newblock The quantum disk is not a quantum group.
\newblock {\em J. Topol. Anal.}, 15(2):401--411, 2023.

\bibitem{KrajczokWasilewski}
Jacek Krajczok and Mateusz Wasilewski.
\newblock On the von {N}eumann algebra of class functions on a compact quantum
  group.
\newblock {\em J. Funct. Anal.}, 283(5):Paper No. 109549, 29, 2022.

\bibitem{KustermansVaes}
Johan Kustermans and Stefaan Vaes.
\newblock Locally compact quantum groups.
\newblock {\em Ann. Sci. \'Ecole Norm. Sup. (4)}, 33(6):837--934, 2000.

\bibitem{KustermansVaesVNA}
Johan Kustermans and Stefaan Vaes.
\newblock Locally compact quantum groups in the von {N}eumann algebraic
  setting.
\newblock {\em Math. Scand.}, 92(1):68--92, 2003.

\bibitem{NeshveyevTuset}
Sergey Neshveyev and Lars Tuset.
\newblock {\em Compact quantum groups and their representation categories},
  volume~20 of {\em Cours Sp\'{e}cialis\'{e}s [Specialized Courses]}.
\newblock Soci\'{e}t\'{e} Math\'{e}matique de France, Paris, 2013.

\bibitem{OnishchikVinberg}
Arkadij~L. Onishchik and Ernest.~B. Vinberg.
\newblock {\em Lie groups and algebraic groups}.
\newblock Springer Series in Soviet Mathematics. Springer-Verlag, Berlin, 1990.
\newblock Translated from the Russian and with a preface by D. A. Leites.

\bibitem{PodlesWoronowicz}
Piotr Podle\'{s} and Stanis{\l}aw~L. Woronowicz.
\newblock Quantum deformation of {L}orentz group.
\newblock {\em Comm. Math. Phys.}, 130(2):381--431, 1990.

\bibitem{ReshetikhinYakimov}
Nicolai Reshetikhin and Milen Yakimov.
\newblock Quantum invariant measures.
\newblock {\em Comm. Math. Phys.}, 224(2):399--426, 2001.

\bibitem{Sakai}
Sh\^{o}ichir\^{o} Sakai.
\newblock {\em {$C^*$}-algebras and {$W^*$}-algebras}.
\newblock Classics in Mathematics. Springer-Verlag, Berlin, 1998.
\newblock Reprint of the 1971 edition.

\bibitem{nazb}
Piotr~M. So{\l}tan.
\newblock New quantum ``{$az+b$}'' groups.
\newblock {\em Rev. Math. Phys.}, 17(3):313--364, 2005.

\bibitem{primer}
Piotr~M. So{\l}tan.
\newblock {\em A primer on {H}ilbert space operators}.
\newblock Compact Textbooks in Mathematics. Birk\-h\"{a}u\-ser/Sprin\-ger,
  Cham, 2018.

\bibitem{SoltanWoronowicz}
Piotr~M. So{\l}tan and Stanis{\l}aw~L. Woronowicz.
\newblock From multiplicative unitaries to quantum groups. {II}.
\newblock {\em J. Funct. Anal.}, 252(1):42--67, 2007.

\bibitem{sweedler}
Moss~E. Sweedler.
\newblock {\em Hopf algebras}.
\newblock Mathematics Lecture Note Series. W. A. Benjamin, Inc., New York,
  1969.

\bibitem{Takesaki2}
Masamichi Takesaki.
\newblock {\em Theory of operator algebras. {II}}, volume 125 of {\em
  Encyclopaedia of Mathematical Sciences}.
\newblock Springer-Verlag, Berlin, 2003.
\newblock Operator Algebras and Non-commutative Geometry, 6.

\bibitem{Timmermann}
Thomas Timmermann.
\newblock {\em An invitation to quantum groups and duality}.
\newblock EMS Textbooks in Mathematics. European Mathematical Society (EMS),
  Z\"{u}rich, 2008.
\newblock From Hopf algebras to multiplicative unitaries and beyond.

\bibitem{Tuset}
Lars Tuset.
\newblock {\em Analysis and quantum groups}.
\newblock Springer, Cham, [2022] \copyright 2022.

\bibitem{VaesRN}
Stefaan Vaes.
\newblock A {R}adon-{N}ikodym theorem for von {N}eumann algebras.
\newblock {\em J. Operator Theory}, 46(3, suppl.):477--489, 2001.

\bibitem{StrictlyOuter}
Stefaan Vaes.
\newblock Strictly outer actions of groups and quantum groups.
\newblock {\em J. Reine Angew. Math.}, 578:147--184, 2005.

\bibitem{VanDaeleTheHaarMeasure}
Alfons {Van Daele}.
\newblock {The Haar measure on some locally compact quantum groups}.
\newblock {\em arXiv Mathematics e-prints}, page math/0109004, September 2001.

\bibitem{VanDaeleQGvN}
Alfons Van~Daele.
\newblock Locally compact quantum groups. {A} von {N}eumann algebra approach.
\newblock {\em SIGMA Symmetry Integrability Geom. Methods Appl.}, 10:Paper 082,
  41, 2014.

\bibitem{WangVanDaele}
Alfons Van~Daele and Shuzhou Wang.
\newblock Universal quantum groups.
\newblock {\em Internat. J. Math.}, 7(2):255--263, 1996.

\bibitem{WeiZou}
Yangjiang Wei and Yi~Ming Zou.
\newblock Inverses of {C}artan matrices of {L}ie algebras and {L}ie
  superalgebras.
\newblock {\em Linear Algebra Appl.}, 521:283--298, 2017.

\bibitem{su2}
Stanis{\l}aw~L. Woronowicz.
\newblock Twisted {${\rm SU}(2)$} group. {A}n example of a noncommutative
  differential calculus.
\newblock {\em Publ. Res. Inst. Math. Sci.}, 23(1):117--181, 1987.

\bibitem{slw_E2}
Stanis{\l}aw~L. Woronowicz.
\newblock Quantum {$E(2)$} group and its {P}ontryagin dual.
\newblock {\em Lett. Math. Phys.}, 23(4):251--263, 1991.

\bibitem{slw_unbo}
Stanis{\l}aw~L. Woronowicz.
\newblock Unbounded elements affiliated with {$C^*$}-algebras and noncompact
  quantum groups.
\newblock {\em Comm. Math. Phys.}, 136(2):399--432, 1991.

\bibitem{slw_gen}
Stanis{\l}aw~L. Woronowicz.
\newblock {$C^*$}-algebras generated by unbounded elements.
\newblock {\em Rev. Math. Phys.}, 7(3):481--521, 1995.

\bibitem{azb}
Stanis{\l}aw~L. Woronowicz.
\newblock Quantum ``{$az+b$}'' group on complex plane.
\newblock {\em Internat. J. Math.}, 12(4):461--503, 2001.

\bibitem{SLWHaarMeasure}
Stanis{\l}aw~L. Woronowicz.
\newblock {H}aar weight on some quantum groups.
\newblock Number 173, pages 763--772. 2002.
\newblock Physical and Mathematical Aspects of Symmetries: Proceedings of the
  24th International Colloquium on Group Theoretical Methods in Physics, Paris,
  15-20 July 2002.

\end{thebibliography}
\end{document}